%% file: ActuatorExplicit_a.tex
\documentclass{cocv0}

\usepackage{amscd,amsthm,amsmath,amssymb,mathtools}
\usepackage{verbatim}
\usepackage{color}
\usepackage{xcolor}
\usepackage{curve2e}
\usepackage{hyperref}
\usepackage{url}
\usepackage{enumerate,enumitem}
\usepackage{graphicx}
\usepackage{epstopdf,epsfig,subfigure}
\usepackage{curve2e}
\usepackage{longtable}

\include{slidesymbols}

\usepackage[]{algorithm2e}

\theoremstyle{plain}
\newtheorem{MT}{Main Theorem}
\numberwithin{MT}{section}
\newtheorem*{MT*}{Main Result}

\newtheorem{theorem}{Theorem}[section]
\newtheorem{proposition}[theorem]{Proposition}
\newtheorem{lemma}[theorem]{Lemma}
\newtheorem{corollary}[theorem]{Corollary}

\theoremstyle{definition}
\newtheorem{definition}[theorem]{Definition}
\newtheorem{remark}[theorem]{Remark}

\newcommand{\ben}{\begin{equation}}
\newcommand{\een}{\end{equation}}

\newcommand{\benn}{\begin{equation*}}
\newcommand{\eenn}{\end{equation*}}

\numberwithin{equation}{section}

 \usepackage[scrtime]{prelim2e}

\newcommand{\linspan}{\mathop{\rm span}\nolimits}

\newcommand{\rest}{\left.\kern-2\nulldelimiterspace\right|_}
\newcommand{\norm}[2]{\left|#1\right|_{#2}}
\newcommand{\opnorm}[2]{\left\|#1\right\|_{#2}}
\newcommand{\fractx}[2]{{\textstyle\frac{#1}{#2}}}

\newcommand{\p}{\partial}

\newcommand{\ed}{\mathrm d}

\newcommand{\blue}{\color{blue}} 
 
\newcommand{\red}{\color{red}} 
\newcommand{\green}{\color{green}} 
\newcommand{\magenta}{\color{magenta}} 
\newcommand{\cyan}{\color{cyan}} 
\newcommand{\orange}{\color{orange}} 

\newcommand*{\Bigcdot}{\raisebox{-.25ex}{\scalebox{1.25}{$\cdot$}}}

\newcommand{\R}{{\mathbb R}}

\newcommand{\RRR}{{\mathbf R}}

\newcommand{\KKK}{{\mathbf K}}
\newcommand{\PPP}{{\mathbf P}}

\newcommand{\GGG}{{\mathbf G}}

\newcommand{\DD}{{\mathcal D}}

\newcommand{\KK}{{\mathcal K}}
\newcommand{\LL}{{\mathcal L}}
\newcommand{\MM}{{\mathcal M}}

\newcommand{\Ma}{{\mathbf M}}
\newcommand{\St}{{\mathbf S}}

\begin{document}
\title{On the explicit feedback stabilisation of~\textsc{1D} linear nonautonomous parabolic equations via oblique projections}
\author{S\'ergio S.~Rodrigues}
\address{Johann Radon Institute for Computational and Applied Mathematics,
Altenbergerstra{\normalfont\ss}e 69, 4040 Linz, Austria.\newline \indent
Tel. +43 732 2468 5241.\quad e-mail: {\small\tt sergio.rodrigues@ricam.oeaw.ac.at}}
\author{Kevin Sturm}
\address{TU Wien, Institute for Analysis and Scientific Computing, Wiedner Hauptstra{\normalfont\ss}e 8-10, 1040 Wien, Austria. \newline \indent
 e-mail: {\small\tt kevin.sturm@tuwien.ac.at}}

\begin{abstract}
In recently proposed stabilisation techniques for parabolic equations, a crucial role is played by a
suitable sequence of oblique projections in~Hilbert spaces, onto the linear span of
a suitable set of~$M$
actuators, and along the subspace orthogonal to the space
spanned by ``the'' first~$M$ eigenfunctions of the Laplacian operator. 
 This new approach uses an explicit feedback law, 
which is stabilising provided that the sequence of operator
norms of such oblique projections remains bounded.

The main result of the paper is the proof
that, for suitable explicitly given sequences of sets of
actuators, the operator norm of the corresponding oblique projections remains bounded.

In the final part of the paper we provide numerical results,
showing the performance of the explicit feedback control for both Dirichlet and Neumann homogeneous
boundary conditions.

\end{abstract}

\subjclass[2010]{93D15,47A75,47N70}

\keywords{Exponential stabilisation, nonautonomous parabolic equations, nonorthogonal projections, eigenvalue problems}

\maketitle

\vspace*{-3em}
{\tiny
}

\pagestyle{myheadings} \thispagestyle{plain} \markboth{\sc S. S. Rodrigues and K. Sturm}
{\sc Stabilisation of parabolic equations via oblique projections }

\section{Introduction}
In the recent work \cite{KunRod-pp17} a new explicit feedback control for nonautonomous parabolic equations such as
\begin{equation}\label{sys-y-par-intro}
  \fractx{\p}{\p t} y -\nu\Delta y +ay=\sum_{j=1}^M u_j1_{\omega_j} \quad  \text{ in } (0,L)\times (0,\infty), \qquad    y(0)=y_0 \quad \text{ in } (0,L),
\end{equation}
supplemented with either homogeneous Dirichlet or Neumann boundary conditions, has been proposed.
The reaction coefficient~$a=a(x,t)\in\bbR$ is given and allowed to be time and space dependent.
The scalar functions $u_1(t),u_2(t),\dots,u_M(t)$ are time-dependent controls at our disposal.
The indicator functions~$1_{\omega_i}=1_{\omega_i}(x)$ associated with domains~$\omega_i\subset(0,L)$, are our actuators.
Note that in this way, our control   on the right hand side of \eqref{sys-y-par-intro}, 
is finite dimensional, that is, for all given time $t\ge0$ our control is a linear combination of our (finite number of) actuators. Further, it is
supported (localised) in the union~$\bigcup_{j=1}^M\omega_j$.

In \cite[Sect.~3 and~5]{KunRod-pp17} it has been proven that for all positive~$\lambda>0$, the explicit feedback
control~$\sum_{j=1}^M u_j(t)1_{\omega_j}$ defined by  
\begin{equation}\label{FeedKy}
  y\mapsto\KK y\coloneqq P_ {U_M}^{E_{M}^\perp}\left(-\Delta y +ay-\lambda y\right)
\end{equation}
is exponentially stabilising, provided we can find a sequence of actuator domains  $\omega_1^M,\ldots, \omega_M^M$, such that the subspaces~$U_{M}=\linspan\{1_{\omega_1^M},1_{\omega_2^M},\ldots, 1_{\omega_{M}^M}\}$ and $E_M\subset L^2(0,L)$ satisfy the conditions
\begin{subequations}
  \begin{align}
       L^2(0,L)&={U_M}\oplus E_M^\perp,\label{eq:splitting}\\
    \intertext{and}
       \norm{ P_{U_M}^{E_M^{\perp}}}{\LL(L^2(0,L)) }&\le C_P,\quad\mbox{for all}\quad M\ge1,\label{bdOpnorm}
   \end{align}
   for a suitable constant $C_P\ge1$, independent of~$M$. 
 \end{subequations}
Here $E_M$ is the linear span of the first $M$ eigenfunctions of the Dirichlet
or Neumann Laplacian~$-\Delta$ on $(0,L)$. 
Further $P_{{U_M}}^{E_M^{\perp}}$
is the oblique/nonorthogonal projection in~$L^2(0,L)$ onto the linear span~${U_M}$ of our
actuators defined through the direct sum \eqref{eq:splitting}. So, the first
condition~\eqref{eq:splitting} merely tells us that the oblique projection is well-defined.  
We underline that verifying that the operator norm~\eqref{bdOpnorm} remains bounded is challenging since
a nonorthogonal projection has an operator norm strictly bigger than~$1$ and,
depending  on the choice  of~$U_{M}$,  a finite constant $C_P$ satisfying \eqref{bdOpnorm} may not exist.  
However, in~\cite[Sections~4.6 and~4.7]{KunRod-pp17}, it was observed numerically that
the operator norm will likely remain bounded for suitable placements of the
actuators. The main contribution of this paper is to give a rigorous proof of this numerical observation.

 We will  consider for a given~$r\in(0,1)$, and for
a given integer~$M\ge 1$ equally sized intervals $\omega_1^M,\ldots, \omega_M^M \subset (0,L)$,
each of length  $\frac{rL}{M}$, with centers~$c_1^{M},\ldots, c_M^M\in(0,L)$ given by 
 \begin{equation}\label{act-eqlen}
\omega_j^{M}\coloneqq(c_j^{M}-\fractx{rL}{2M}, c_j^{M}+ \fractx{rL}{2M}),
\qquad j\in\{1,2,\dots,M\}.
\end{equation}
With these domains we can state our main result. 
\begin{MT*}
  Let $r\in (0,1)$ and $L>0$. Let $E_M$, $M\ge 1$, be the linear span of the
  first eigenfunctions of the Dirichlet or Neumann Laplacian in $L^2(0,L)$. Then there exists for each $M\ge 1$, centers $c_1^M,\dots,c_M^M\in (0,L)$, such that the corresponding intervals  $\omega_1^M,\ldots, \omega_M^M$ defined via \eqref{act-eqlen} are disjoint and the sequence
  of linear spans $U_M:=\linspan\{1_{\omega_1^M},\ldots, 1_{\omega_M^M}\}$ of the corresponding actuators satisfies~\eqref{eq:splitting} and  \eqref{bdOpnorm}.
\end{MT*}

Explicit sequences $c_1^M,\ldots, c_M^M$ that define the domains~$\omega_j^M$ are given hereafter.
We underline that by construction the total volume  covered by the (disjoint) actuators
is given by~${\rm vol}\left(\bigcup_{j=1}^M\omega_i^M\right)=\sum_{j=1}^M{\rm vol}(\omega_i^M)=rL$ remains the same for all $M\ge 1$. That is, by taking a larger number of actuators, the volume of the control support remains the exactly same, equal to~$rL$.
\smallskip
 A particularity of the feedback defined in~\eqref{FeedKy} is that it is explicit and its computation boils down to the
construction of the oblique projection, which in turn (as we will see) requires only the inversion of a~$M\times M$ matrix,
where~$M$ is the number of actuators. 
This makes the method numerically attractive because, for example, it is cheaper to compute when compared to Riccati based feedbacks,
which require the computation of the solution of a suitable differential Riccati equation  and  may
be a quite demanding numerical task, for accurate approximations.

 On the other hand, we cannot guarantee that the feedback~$\Ck$ in~\eqref{FeedKy} is stabilising whenever the Riccati based feedback is.
That is, in theory, when~$\Ck$ fails to stabilise the system (for a ``bad'' placement of the actuators), it may happen that a Riccati based feedback
is still stabilising.

So, in a real application, in case we are able to (choose where to) place the actuators,
then we propose the feedback~$\Ck$ which is computationally cheap. In the case
we are not able to place the actuators, then the feedback~$\Ck$ may be not stabilising, and in that case we can/must try with a Riccati-based feedback.
For works dealing with Riccati based stabilisation of parabolic-like nonautonomous systems we refer
to~\cite{PhanRod18-mcss,Rod18,BarRodShi11,BreKunRod17,KroRod15,KroRod-ecc15}.

Condition~\eqref{bdOpnorm} suggests us that ``an optimal'' placement of the actuators as in~\eqref{act-eqlen} for a given~$M$, is one placement that
minimises the operator norm of the oblique projection. This is a nontrivial interesting problem to be addressed in a future
work~\cite{RodSturm-ow18}. Concerning this point we would like to mention~\cite{Morris11,MorrisYang16} where the question of
optimal actuator location/placement is addressed, with a different goal: that 
of minimising a quadratic cost functional. 
See also the recent works~\cite{KaliseKunischSturm_arx17,PrivatTrelatZuazua17}  concerning the problem of optimal actuators design
(where the goal is again to minimise a suitable cost functional, where the shape of the control's (or each actuator's) support is not
fixed a priori).

\begin{remark}
The results in~\cite{KunRod-pp17}, for the explicit feedback~$\Ck$ in~\eqref{FeedKy}, besides a reaction term~$ay$ in~\eqref{sys-y-par-intro}
also allow the inclusion of a convention term of the form $\nabla\cdot\bigl(b(t)z\bigr)$ in the dynamics
of system~\eqref{sys-y-par-intro}, under Dirichlet
boundary conditions. However, the inclusion of a general convention term in system~\eqref{sys-y-par-intro} under Neumann boundary conditions is 
not covered by the results in~\cite{KunRod-pp17}. See in particular~\cite[Assumption~2.3]{KunRod-pp17}. Here the main goal
is the investigation of the properties of the oblique projection~$P_{U_M}^{E_M^\perp}$,
appearing in the feedback control~\eqref{FeedKy}, rather than investigating the regularity properties of the system, 
that is why here we consider only the case of a reaction term, for which the results in~\cite{KunRod-pp17} still hold for Neumann boundary conditions.
Notice that in~\cite[Sect.~2]{KunRod-pp17} the eigenvalues of the symmetric operator~$A$, 
are asked to be strictly positive. For homogeneous Neumann boundary conditions we can write~$-\nu\Delta y +ay=(-\nu\Delta +\Id)y +(a-1)y$, and
set~$A=-\nu\Delta +\Id$ whose  eigenvalues,~$\nu\alpha_i+1\ge1$, are strictly positive.
\end{remark}

We are particularly interested in the case of nonautonomous systems. In the particular case of autonomous
the spectral properties of the time independent operator~$-\nu\Delta +a\Id$ may be exploited, see~\cite[Section~5.6]{KunRod-pp17}. We refer to the
works~\cite{RaymThev10,BadTakah11,Barbu12,BarbuLasTri06,BarbuTri04,Barbu_TAC13,HalanayMureaSafta13} and references therein.
In~\cite{Barbu12,Barbu_TAC13,HalanayMureaSafta13} the feedbacks are
constructed explicitly, while
in~\cite{RaymThev10,BadTakah11,BarbuLasTri06,BarbuTri04} they are based on  Riccati equations. We recall that
the spectral properties of~$-\nu\Delta +a(t)\Id$ seem to be (at least, by themselves) not appropriate for studying the stability
of the corresponding nonautonomous system, see~\cite{Wu74}.

\medskip
The paper is organized as follows.
In Sect.~\ref{S:mainresults} we give a more precise description of the main results. In Sect.~\ref{S:oblique_proj} we recall suitable properties
of oblique projections. In Sect.~\ref{S:Dirichlet} we prove the main results for the case of Dirichlet boundary conditions.
The results for the case of Neumann boundary conditions are proven in Sect.~\ref{S:Neumann}. In Sect.~\ref{S:OtherLoc}  we give some remarks concerning the location of the actuators.
A few additional remarks concerning the oblique projection based feedback are given in Sect.~\ref{S:remOPorjFeed}.
Numerical simulations on the performance of the
feedback control are presented in Sect.~\ref{S:numeric}. Finally, the appendix includes the proof of an auxiliary result used in the main text.

\medskip\noindent
{\em Notation.}
We write~$\mathbb R$ and~$\mathbb N$ for the sets of real and nonnegative
integer numbers, respectively.

Given a (real and separable) Hilbert space~$X$, with scalar product~$(\Bigcdot,\Bigcdot)_X$ and associated~$\norm{\Bigcdot}{X}$, 
the subspace orthogonal to a subset~$S\subseteq X$ will be denoted~$S^\perp\coloneqq\{h\in X\mid (h,s)_H=0\mbox{ for all }s\in S\}$, as usual.

Given another Hilbert space~$Z$ the set of bounded linear mapping from~$X$ into~$Z$ will be denoted~$\LL(X,Z)$. When~$Z=X$ we
simply write~$\LL(X)\coloneqq\LL(X,X)$. The dual of~$X$ is denoted $X'=\LL(X,\R)$ and  equipped with the usual norm dual
norm~$\norm{f}{X'}=\sup\limits_{x\in X,\; |x|_X\le 1}|f(x)|$ for $f\in X'$.  

Given a linear operator $P\in\LL(X,Z)$ we denote its adjoint by $P^*\in \LL(Z',X')$, defined by
$\langle P^*f,h\rangle_{X',X}=\langle f,Ph\rangle_{Z',Z}$, for all~$(f,h)\in Z'\times X$, where $ \langle \cdot,\cdot \rangle_{X',X}$ denotes as usual the duality pairing.

\section{Main results}\label{S:mainresults}
 
The goal of this paper is to investigate suitable properties of the oblique
projection~$P_{U_M}^{E_M^{\perp}}$ in~$L^2(0,L)$, onto~$U_M$ along~$E_M^{\perp}$
which we define precisely below.  The space~$U_M$ is the space spanned by our actuators. More precisely for a given~$M\ge1$, we define 
\ben
 U(c^{M}):=U_M:=  \linspan\{1_{\omega_j^M}\mid i\in\{1,2,\dots,M\}\},
\een
where for fixed $r\in (0,1)$ the actuator domains $\omega_1^M,\dots, \omega_M^M\subset (0,L)$ 
are defined via \eqref{act-eqlen} and are disjoint (i.e., $\omega_i^{M}\cap\omega_j^{M}=\emptyset$ if~$i\ne j$), which is to say that, 
\begin{equation}\label{act-disj}
 c_{i}^{M}-c_{j}^{M}\ge \fractx{rL}{M}, \qquad\mbox{if}\qquad i\ne j.
\end{equation}
The space~$E_M^\perp$ is the orthogonal complement, in~$L^2(0,L)$, to the space
spanned by  the first~$M$  eigenfunctions of the Laplacian~$-\Delta$ in~$L^2(0,L)$, under either Dirichlet of Neumann boundary conditions. 

Now, for clarity of the exposition we will distinguish between Dirichlet of Neumann boundary conditions. We will use a (left) superscript
as ${}^{\tt d}(\Bigcdot)$ to underline the Dirichlet boundary conditions and a supercript as ${}^{\tt n}(\Bigcdot)$ to underline the Neumann boundary conditions.
For statements concerning either of these boundary conditions we write no superscript.

The Dirichlet Laplacian eigenvalues and eigenfunctions are denoted:
\begin{equation}\label{eigs-Dir}
-\Delta {}^{\tt d}\!e_i={}^{\tt d}\!\alpha_i{}^{\tt d}\!e_i,\quad  {}^{\tt d}\!e_i(0)=0={}^{\tt d}\!e_i(L),
\qquad 0<{}^{\tt d}\!\alpha_1 < {}^{\tt d}\!\alpha_2\le \ldots  \le{}^{\tt d}\!\alpha_N \rightarrow +\infty, 
\end{equation}
and we set
\[
 {}^{\tt d}\!E_M=\linspan\{{}^{\tt d}\!e_i\mid i\in\{1,2,\dots,M\}\}.
\]

 In the following to simplify the exposition for linear bounded operators~$\Cp:L^2(0,L)\to L^2(0,L)$, we denote the operator norm simply by
$\opnorm{\Cp}{}\coloneqq\norm{\Cp}{\LL(L^2(0,L))}$.

\begin{MT}\label{MT:Dirmxe}
Let~$r\in(0,1)$, and consider
 the location of our actuators given by
\begin{equation}\label{Dmxe}\stepcounter{equation}
 c_j^M\coloneqq\frac{(2j-1)L}{2M},
\qquad\mbox{with}\quad  j\in\{1,2,\dots,M\}.  \tag{(\theequation)--{\tt mxe}}
\end{equation}
Then the operator norm of the oblique projection
  $P_{{U(c^M)}}^{{}^{\tt d}\!E_M^{\perp}}:L^2(0,L)\to L^2(0,L)$  is given by
\[
  \opnorm{P_{{U(c^M)}}^{{}^{\tt d}\!E_M^{\perp}}}{}=\left(\fractx{4M^2}{r\pi^2(M-1)^2}\sin^2\left(\fractx{(M-1)r\pi}{2M}\right)\right)^{-\frac{1}{2}},
 \quad\mbox{for}\quad M\ge2.
\]
 In particular~$\left\{\opnorm{P_{{U(c^M)}}^{{}^{\tt d}\!E_M^{\perp}}}{}\right\}_{M\ge 2}$ increases and converges to $\frac{\sqrt{r}\pi}{2\sin\left(\frac{r\pi}{2}\right)}$. 
\end{MT}

\begin{MT}\label{MT:Diruni}
Let~$r\in(0,1)$, and consider for
 the location of our actuators given by
\begin{equation}\label{Duni}\stepcounter{equation}
c_j^M=\frac{j L}{M+1}, 
\qquad\mbox{for }\quad  M \ge \fractx{r}{1-r}  
\quad\mbox{and}\quad  j\in\{1,2,\dots,M\}, \tag{(\theequation)--{\tt uni}}
\end{equation}
Then the operator norm of the oblique projection $P_{{U(c^M)}}^{{}^{\tt d}\!E_M^{\perp}}:L^2(0,L)\to L^2(0,L)$ is given by
\[
  \opnorm{P_{{U(c^M)}}^{{}^{\tt d}\!E_M^{\perp}}}{}
 =\left(\fractx{4(M+1)}{r\pi^2M}\sin^2\left(\fractx{r\pi}{2}\right)\right)^{-\frac{1}{2}},\quad\mbox{for}\quad  M \ge \fractx{r}{1-r}.
\]
 In particular~$\left\{\opnorm{P_{{U(c^M)}}^{{}^{\tt d}\!E_M^{\perp}}}{}\right\}_{M\in\bbN}$ increases, for~$M\ge \fractx{r}{1-r}$,  and converges to $\frac{\sqrt{r}\pi}{2\sin\left(\frac{r\pi}{2}\right)}$.  
\end{MT}

The condition $M\ge\fractx{r}{1-r}$ in~\eqref{Duni} is necessary and sufficient to guarantee~\eqref{act-disj},
because  $c_{j+1}-c_j=\frac{L}{M+1}$ and thus we have~\eqref{act-disj} if, and only if,~$\frac{1}{M+1}\ge\frac{r}{M}$.

\bigskip

A similar results holds true, under Neumann boundary conditions.

Let ${}^{\tt n}\!E_M$ be the span ${}^{\tt n}\!E_M=\linspan\{{}^{\tt n}\!e_i\mid i\in\{1,2,\dots,M\}\}$
of the first $M$ Neumann eigenfunctions of~$-\Delta$ in~$L^2(0,L)$:
\begin{equation}\label{eigs-Neu}
-\Delta {}^{\tt n}\!e_i={}^{\tt n}\!\alpha_i {}^{\tt n}\!e_i,\quad  \fractx{\p}{\p{x}}{}^{\tt n}\!e_i\rest{x=0}=0=\fractx{\p}{\p{x}}{}^{\tt n}\!e_i\rest{x=L}
,\qquad 0={}^{\tt n}\!\alpha_1 < {}^{\tt n}\!\alpha_2\le \ldots  \le{}^{\tt n}\!\alpha_N \rightarrow +\infty. 
\end{equation}
Analogously $P_{{U_M}}^{{}^{\tt n}\!E_M^{\perp}}$
is the oblique projection in~$L^2(0,L)$ onto~${U_M}=\linspan\{1_{\omega_j}\mid j\in\{1,2,\dots,M\}\}$ along~${}^{\tt n}\!E_M^\perp$,
the orthogonal complement of~${}^{\tt n}\!E_M$.
\begin{MT}\label{MT:Neumxe}
Let~$r\in(0,1)$, and consider
 the location of our actuators given by~\eqref{Dmxe}. 
Then the operator norm of the oblique
projection~$P_{{U(c^M)}}^{{}^{\tt n}\!E_M^{\perp}}:L^2(0,L)\to L^2(0,L)$ is given by
\[
  \opnorm{P_{{U(c^M)}}^{{}^{\tt n}\!E_M^{\perp}}}{}
 =\left(\fractx{4M^2}{r\pi^2(M-1)^2}\sin^2\left(\fractx{r\pi(M-1)}{2M}\right)\right)^{-\frac{1}{2}},\quad\mbox{for}\quad M\ge2.
\]
 In particular~$\left\{\opnorm{P_{{U( c^M)}}^{{}^{\tt n}\!E_M^{\perp}}}{}\right\}_{M\ge 2}$
increases and converges to $\frac{\sqrt{r}\pi}{2\sin\left(\frac{r\pi}{2}\right)}$. 
\end{MT}

Next, we give some remarks on the above results.

\begin{remark}
In~\cite{KunRod-pp17} it was shown that the feedback \eqref{FeedKy} is stabilising for system \eqref{sys-y-par-intro},  
if \eqref{eq:splitting} and 
\begin{align}\label{suffcond}
 \nu\alpha_{M+1}&>\left(6+4\opnorm{P_{U_M}^{E_M^{\perp}}}{}^2\right)\norm{a\Id}{L^\infty((0,+\infty),\LL(L^2(0,L),H^{-1}(0,L)))}^2,
 \end{align}
are satisfied. The operator~$\Id$,  in~\eqref{suffcond}, stands for the identity/inclusion
operator. Since  Dirichlet and Neumann eigenvalues satisfy
$\alpha_M\to +\infty$, then it follows from the boundedness of $\opnorm{P_{U_M}^{E_M^\perp}}{}$ that \eqref{suffcond} is satisfied. 
Recall that under Dirichlet boundary conditions
we have~$\alpha_M={}^{\tt d}\!\alpha_M=(\frac{\pi}{L})^2M^2$, and under Neumann boundary conditions
we have~$\alpha_M={}^{\tt n}\!\alpha_M=(\frac{\pi}{L})^2(M-1)^2$.
\end{remark}

Using~\eqref{suffcond} and Main Theorems~\ref{MT:Dirmxe} and~\ref{MT:Diruni} we have the following.
\begin{corollary}
  Let $L>0$, ~$\lambda>0$, and~$\nu>0$ be given positive real numbers, and let $a\in L^\infty((0,L)\times(0,+\infty)))$ be given.
  Let us place the actuators either as in~\eqref{Dmxe} or as in~\eqref{Duni}.
  Then, the nonautonomous closed loop system 
\begin{equation}\label{sys-y-parDir}
  \fractx{\p}{\p t} y -\nu\Delta y +ay-P_ {U_M}^{{}^{\tt d}\!E_M^\perp}\left(-\nu\Delta y +ay-\lambda y\right)=0,\quad    y(0)=y_0,\quad
      y(0)=0=y(L)
\end{equation}
is stable, for~$M$ large enough, say for
\[
 M+1\ge \frac{1}{\sqrt\nu}\frac{L}{\pi}
\left(6+\left(\frac{\sqrt{r}\pi}{\sin\left(\frac{r\pi}{2}\right)}\right)^2\right)^\frac12\norm{a\Id}{L^\infty((0,+\infty),\LL(L^2(0,L),H^{-1}(0,L)))}.
\]
\end{corollary}

 Using~\eqref{suffcond} and Main Theorems~\ref{MT:Neumxe} we have the following.
\begin{corollary}
  Let $L>0$, ~$\lambda>0$, and~$\nu>0$ be given positive real numbers, and let $a\in L^\infty((0,L)\times(0,+\infty)))$ be given.
  Let us place the actuators as in~\eqref{Dmxe}. Then, the nonautonomous closed loop system 
\begin{equation}\label{sys-y-parNeu}
  \fractx{\p}{\p t} y -\nu\Delta y +ay-P_ {U_M}^{{}^{\tt n}\!E_M^\perp}\left(-\nu\Delta y +ay-\lambda y\right)=0,\quad    y(0)=y_0,\quad
      \fractx{\p}{\p{x}}y\rest{x=0}=0=\fractx{\p}{\p{x}}y\rest{x=L}
\end{equation}
is stable, for~$M$ large enough, say for
\[
 M\ge \frac{1}{\sqrt\nu}\frac{L}{\pi}
\left(6+\left(\frac{\sqrt{r}\pi}{\sin\left(\frac{r\pi}{2}\right)}\right)^2\right)^\frac12\norm{a\Id}{L^\infty((0,+\infty),\LL(L^2(0,L),H^{-1}(0,L)))}.
\]
\end{corollary}

\begin{remark}\label{R:samelimnormOP}
Notice that the limit of the operator norm of the oblique projections is the same in
 Main Theorems~\ref{MT:Dirmxe}, ~\ref{MT:Diruni}~\ref{MT:Neumxe}. 
 In spite of this observation, we
do not know the value of
~$\lim\limits_{M\to+\infty}\opnorm{P_{{U_M}}^{{}^{\tt n}\!E_M^{\perp}}}{}$ for the location as in~\eqref{Duni}.
Numerical simulations we shall present later on, do not allow us to conclude that the limit exists,
that is, they suggest that~$\opnorm{P_{{U_M}}^{{}^{\tt n}\!E_M^{\perp}}}{}$ could go to infinity as~$M$ increases. What is clear from the simulations
is that, even in
case the operator norm of the oblique projection would remain bounded, 
the value of the limit will be considerably larger
for the location~\eqref{Duni} than for the location~\eqref{Dmxe}.
This is a remarkable difference between Neumann and Dirichlet boundary conditions.
\end{remark}

\section{Basic properties of oblique projections}\label{S:oblique_proj}

We are going to use oblique  projection operators associated with a suitable direct sum splitting of a given (real) Hilbert space~$H$.
In what follows~$\bbR^{n\times m}$ stands for the linear space of matrices with~$n$ rows and~$m$ columns, and with real entries. For simplicity, we will also
identify a given element $z=(z_1,z_2,\dots,z_m)\in\bbR^m$ with
the column matrix~$z=\begin{bmatrix}z_1&z_2&\dots&z_m\end{bmatrix}^\top\in\bbR^{m\times1}$.
 The notation~$\MM^\top$ stands for the transpose of the matrix~$\MM$.
\begin{definition}
  Let $H=F\oplus E$ be a direct sum of two closed subspaces $F$ and~$E$ of~$H$. The \emph{oblique projection} onto~$F$ along~$E$ will be denoted
\[
 P_{F}^{E}\colon H\to F,\quad x\mapsto x_{F},
\]
where~$x_{F}$ is uniquely defined by
\[
 x=x_{F}+x_{E}\quad\mbox{and}\quad (x_{F},x_{E})\in F\times E.
\]
\end{definition}

Observe that $P_{F}^{E}$ is an \emph{orthogonal projection} if, and only if,~$E=F^\perp$. In such case we simply write $P_{F}\coloneqq P_{F}^{F^\perp}.$

\begin{remark}\label{R:contproj}
Since the spaces $E$ and $F$ are closed, it follows from the closed graph theorem that the projection $P_{F}^{E}$ is continuous. 
In addition~$P_{E}^{F}=\Id-P_{F}^{E}$, where $\Id:H\to H$ denotes the identity mapping on $H$, $\Id(x) := x$. 
\end{remark}

 Henceforth, let us fix two ordered sets  $\Cf := (f_1,f_2,\dots,f_M)\subset H$ and~$ \Cg:=(g_1,g_2,\dots,g_M)\subset H$, each being
 linearly independent.
 The associated $M$-dimensional subspaces are
 denoted $F\coloneqq\linspan\{f_1,f_2,\dots,f_M\}$ and~$ G\coloneqq\linspan\{g_1,g_2,\dots,g_M\}$. With the two given ordered sets $\Cf$ and $\Cg$, we associate the matrix 
\begin{equation}\label{MatGF}
  [(\Cg,\Cf)_H]\coloneqq[(g_i,f_j)_H]\in\bbR^{M\times M}  
\end{equation}
whose $(i,j)$-entry is~$(g_i,f_j)_H$. For any given~$y\in H$ we also denote the row and column vector matrices
\[
  [(y,\Cf)_H] \coloneqq \begin{bmatrix}(y,f_1)_H&(y,f_2)_H&\dots&(y,f_M)_H\end{bmatrix} \in \bbR^{1\times M} \quad\mbox{and}\quad
  [(\Cg, y)_H] \coloneqq [(y,\Cg)_H]^\top\in \bbR^{M\times 1}.
\]

The following lemma characterizes the direct sum $F\oplus G^\perp$ in terms of the matrix $[(\Cg,\Cf)_H]$; see \cite{KunRod-pp17}. 
\begin{lemma}\label{L:equiv_nop}
 The following conditions are equivalent:
 \begin{enumerate}[noitemsep,topsep=5pt,parsep=5pt,partopsep=0pt,leftmargin=2em]%
\renewcommand{\theenumi}{\rm ({\sf\alph{enumi}})} 
 \renewcommand{\labelenumi}{}
 
\item \theenumi:\label{eqsum} $H=F\oplus G^\perp$,

\item \theenumi:\label{eqmat} $[( \Cg,\Cf)_H]$ is invertible.

\end{enumerate}

In either case the projection~$P_F^{G^\perp}x$ of a vector~$x\in H$, is given by~$\Cp $ defined as follows 
\begin{equation}\label{eq:proj_form}
\Cp x\coloneqq \sum_{j=1}^M \alpha_j f_j, \quad \text{ where }\underline{\alpha}\in
\bbR^M \text{ solves} \quad [(\Cg,\Cf)_H]\underline{\alpha} = [(\Cg,x)_H].
\end{equation}
\end{lemma}

 \begin{proof}
   $(a)\Longleftarrow (b)$: Suppose first that $[(\Cg,\Cf)_H]$ is invertible and let $x\in H$ be given. Consider the unique solution $\underline{\alpha}  =(\alpha_1,\ldots, \alpha_M)\in \bbR^M$ of 
   \begin{equation}\label{eq:fg}
   [(\Cg,\Cf)_H]\underline\alpha = [(\Cg,x)_H]
   \end{equation}
   and define $f:= \sum_{j=1}^M \alpha_j f_j$.  By construction $f\in \Cf$ and $[(\Cg,\Cf)_H]\underline\alpha =  [(\Cg,f)_H]$. We also have $g^\perp:=x-f\in G^\perp$,  because
   \[
     [(\Cg,g^\perp)_H]\stackrel{\eqref{eq:fg}}{=}[(\Cg,\Cf)_H]\underline{\alpha}-[(\Cg,f)_H]=0.        
   \]
   Therefore we have the splitting $x=f+g^\perp$ for $f\in F$ and $g^\perp \in G^\perp$.
    Finally since $\Cp y=y$ for all $y\in F$ and $\Cp z=0$ for all $z\in G^\perp$, 
   we conclude that~$F\bigcap G^\perp=\{0\}$. This shows (a) and $P_F^{G^\perp}=\Cp$.
 
   $(a)\Longrightarrow (b)$: Conversely if $H=F\oplus G^\perp$, then we have for any $\underline \alpha = (\alpha_1,\ldots, \alpha_M)\in \bbR^M$  
 \begin{align*}
   [(\Cg,\Cf)_H]\underline\alpha=0 \quad & \Longleftrightarrow \quad \left(g_i,\textstyle\sum\limits_{j=1}^M \alpha_j f_j\right)_H = 0,
   \quad\mbox{for all}\quad i\in\{1,2,\dots, M\} \\
 				       & \Longleftrightarrow \quad \sum_{j=1}^M \alpha_j f_j \in G^\perp\textstyle\bigcap F 
   \quad \stackrel{F\bigcap G^\perp=\{0\}}{\Longleftrightarrow} \quad  \sum_{j=1}^M \alpha_j f_j=0 
 				        \quad \Longleftrightarrow \quad \underline{\alpha}=0,
 \end{align*}
 which shows that $[(\Cg,\Cf)_H]$ is injective and hence invertible.
 \end{proof}

\begin{remark}\label{R:orthprojoptim}
  If $F=G$, i.e., when $P_F^{G^\perp}$ is the orthogonal projection onto~$F$, then the equations on the right hand side of \eqref{eq:proj_form} are equivalent to the optimality conditions of the minimisation problem
  \begin{equation*}
  \inf_{v\in F}\norm{x-v}{H}^2
  \end{equation*}
  and, in particular, $P_Fx = \argmin_{F} \norm{x-\cdot}{H}^2$.
\end{remark}

\begin{corollary}\label{C:norm_proj}
If~$H=F\oplus G^\perp$, then
\begin{equation}\label{eq:proj_form_basis}
 \norm{P_{F}^{ G^{\perp}}}{\LL(H)}^2= \sup_{\underline{\alpha}^\top [(\Cg,\Cg)_H]\underline{\alpha}\le 1} |[(\Cg,\Cf)_H]^{-1}[(\Cg,\Cg)_H]\underline{\alpha}|^2_F,
\end{equation}
where $|\underline\alpha|_F^2 := \underline\alpha^\top [(\Cf,\Cf)_H]\underline\alpha$. If, in addition, each of the sets  $\Cf$ and $\Cg$ is orthonormal, then
\begin{equation}\label{eq:proj_form_basis2}
\norm{P_{F}^{ G^{\perp}}}{\LL(H)}^2=\left( \min\limits_\theta\left\{\theta\mbox{ is an eigenvalue of } [( \Cg,\Cf)_H][(\Cf, \Cg)_H]\right\}\right)^{-1}.
\end{equation}
\end{corollary}
\begin{proof}
We obtain from \eqref{eq:proj_form},   
\begin{equation*}
|P_F^{G^\perp}x|^2 = \sum_{i,j=1}^M \alpha_i\alpha_j(f_i,f_j) = \underline \alpha^\top [(\Cf,\Cf)_H]
\underline\alpha \stackrel{\eqref{eq:proj_form}}{=} |[(\Cg,\Cf)_H]^{-1}[(\Cg,x)_H]|^2_F, \quad x\in H.
\end{equation*}
In view of $H=G\oplus G^\perp$ and~$(\Cg,y)_H=0$, for all~$y\in G^\perp$, this yields
\begin{equation}\label{eq:proj_orth_basis}
 \norm{P_{F}^{ G^{\perp}}}{\LL(H)}^2 = \sup_{\substack{x\in G\\\norm{x}{H}\le 1}} |[(\Cg,\Cf)_H]^{-1}[(\Cg,x)_H]|^2_F.
\end{equation}
Now notice that $|x|^2 = \sum_{i,j=1}^M \alpha_i \alpha_j(g_i,g_j)_H$ with $\underline{\alpha} = [(\Cg,\Cg)_H]^{-1}[(\Cg,x)_H]$ for all $x\in G$. Therefore \eqref{eq:proj_orth_basis} is equivalent to
\begin{equation*}
\norm{P_{F}^{ G^{\perp}}}{\LL(H)}^2 = \sup_{\underline{\alpha}^\top [(\Cg,\Cg)_H]\underline{\alpha}\le 1} |[(\Cg,\Cf)_H]^{-1}[(\Cg,\Cg)_H]\underline{\alpha}|^2_F,
\end{equation*}
which is \eqref{eq:proj_form_basis}.

To show \eqref{eq:proj_form_basis2} notice that since $\Cf$ and $\Cg$ are orthonormal,
we have $[(\Cf,\Cf)_H]=[(\Cg,\Cg)_H]=I_M$, where $I_M\in \bbR^{M\times M}$ denotes the identity matrix.
Hence \eqref{eq:proj_form_basis2} follows from formula \eqref{eq:proj_form_basis}
and the Rayleigh quotient formula for the first eigenvalue of a matrix.
\end{proof}

\begin{figure}[h!]
\begin{minipage}{.1\textwidth}
\quad{}
\end{minipage}
\begin{minipage}{.4\textwidth}
\setlength{\unitlength}{4mm}
\input{figoblique1}
\end{minipage}
\begin{minipage}{.4\textwidth}
\setlength{\unitlength}{4mm}
\input{figoblique2}
\end{minipage}
\caption{Oblique projection onto~$F$ along~$G^\perp$.}\label{fig.obprojR2}
 \end{figure}
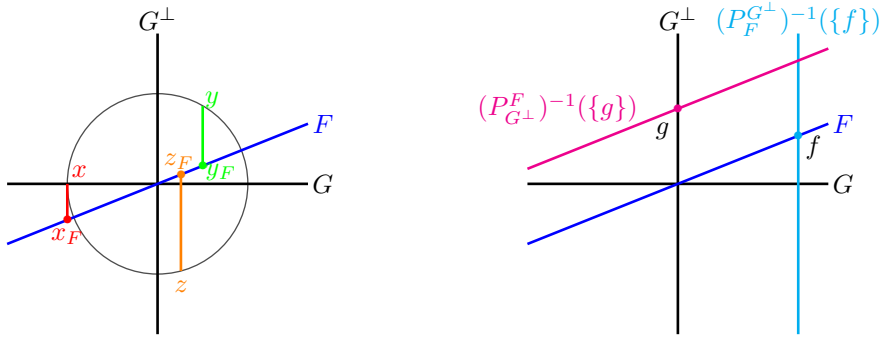

Figure~\ref{fig.obprojR2} illustrates the used terminology of projection onto $F$ along $G^\perp$. The point $x\in H$ is mapped to the unique
point~ $x_F\coloneqq P_F^{G^\perp}x$ in the intersection of~$F$ with the affine space~$x+G^\perp$, which contains $x$ and is parallel to $G^\perp$. The figure
also illustrates the
fact that an oblique nonorthogonal projection
has an operator norm larger than~$1$, for example, in the figure we see that
$\norm{x_F}{H}>\norm{x}{H}$.
That is, we have the following result.
\begin{lemma}\label{lem:geom_illu_proj}
Let~$H=F\oplus G^\perp$. Then the oblique projection onto~$F$ along~$G^\perp$ has the following properties:
\begin{itemize}
  \item[(a)] for any~$x\in H$, the intersection $(x+G^\perp)\textstyle\bigcap F$ consists of the single element $P_F^{G^\perp}x$,
  \item[(b)]  if~$F\ne G$, then~$\norm{P_F^{G^\perp}}{\LL(H)}>1$.
\end{itemize}
\end{lemma}
\begin{proof}
 Let $x\in H$. In view of $H=F\oplus G^\perp$ we have
\begin{align*}
  & y\in (x-G^\perp){\textstyle\bigcap} F \quad 
 \Longleftrightarrow \quad  \exists g^\perp \in G^\perp, \; y=x-g^\perp \in F
  \quad \Longleftrightarrow \quad  P_F^{G^\perp}x = y
\end{align*}
and therefore $(x-G^\perp)\bigcap F$ is the singleton~$\{P_F^{G^\perp}x\}$.

It remains to show that $\norm{P_F^{G^\perp}}{\LL(H)}>1$, if~$F\ne G$. For this it is sufficient to find $x\in H$ such that $|P_F^{G^\perp}x| > |x|$.
Notice also that~$H=F\oplus G^\top$ and~$F\ne G=(G^\perp)^\perp$
   imply that $F\ne\{0\}\ne G^\perp$. In such case there is a pair~$(f, g^\perp)\in F\times G^\perp$ such that~$(f, g^\perp)_H\ne0$. Without loss of generality we may assume that
~$(f, g^\perp)_H>0$ (otherwise consider $-g^\perp$).
Now, we consider the sequence $x^n\coloneqq f-\frac{1}{n}g^\perp$, $n\ge 1$. In view of $H=F\oplus G^\perp$, we have $f=P_F^{G^\perp}x^n$. Moreover, $\norm{x^n}{H}^2=\norm{f}{H}^2+\frac{1}{n^2}\norm{g^\perp}{H}^2
-\frac{2}{n}(f,g^\perp)_{H}$. Therefore we find~$\norm{x^n}{H}^2<\norm{f}{H}^2=\norm{P_F^{G^\perp}x^n}{H}^2$ for all $n>\frac{\norm{g^\perp}{H}^2}{2(f,g^\perp)_{H}}$, which finishes the proof. 
\end{proof}

\begin{remark}
  The hypothesis~$F\ne G$ in Lemma~\ref{lem:geom_illu_proj}(b) simply states that~$H=F\oplus G^\top$ is not an orthogonal sum, that is, the oblique
  projection~$P_F^{G^\perp}$ is a nonorthogonal projection.  Recall that the operator norm of an orthogonal projection on Hilbert spaces
  is always equal to~$1$, that is, $\norm{P_F}{\Cl(H)}=1$ for any nonzero closed space~$F\subseteq H$.
\end{remark}

The next lemma tells us that the adjoint of an oblique projection is still an oblique projection.
\begin{lemma}\label{L:adjProj}
Let~$H$ be a Hilbert space, and let~$F$ and~$G$ be closed subspaces of~$H$ such that~$H=F\oplus G^\perp$. Then the adjoint, in~$\Cl(H)$, of
the oblique projection~$P_{F}^{G^\perp}$ is the oblique projection~$P_{G}^{F^\perp}$.
\end{lemma}

\begin{proof} 
Let~$h\in H$. From Lemma~\ref{L:equiv_nop} it follows that~$H=G\oplus F^\perp$, which implies that
\begin{subequations}
\begin{equation}\label{adjProj+}
h=P_{G}^{F^\perp}h+P_{F^\perp}^{G}h.
\end{equation}
Now from~$(P_{F}^{G^\perp}v,w)_H=(v,(P_{F}^{G^\perp})^*w)_H$
it follows that (cf.~\cite[Section~6.6, Theorem~3]{Luenberger69})
\begin{equation}\label{adjProjKern}
 (P_{F}^{G^\perp})^*H\subseteq G\qquad\mbox{and}\quad (P_{F}^{G^\perp})^*F^\perp=\{0\}.
\end{equation}
Next we show that
\begin{equation}\label{adjProjRan}
(P_{F}^{G^\perp})^*g=g\quad\mbox{for all}\quad g\in G.
\end{equation}
\end{subequations}
Indeed, if~$g\in G$ then, for all~$v\in H$, it follows
that~$(v,(P_{F}^{G^\perp})^*g)_H=(P_{F}^{G^\perp}v,g)_H=(v-P_{G^\perp}^{F}v,g)_H=(v,g)_H$, that is,~$(P_{F}^{G^\perp})^*g=g$.

Applying~$(P_{F}^{G^\perp})^*$ to both sides of~\eqref{adjProj+}, and using~\eqref{adjProjRan} and~\eqref{adjProjKern},
we obtain
$
(P_{F}^{G^\perp})^*h
=P_{G}^{F^\perp}h 
$,
which shows that~$(P_{F}^{G^\perp})^*=P_{G}^{F^\perp}$. 
\end{proof}

\section{The case of Dirichlet boundary conditions}\label{S:Dirichlet}
Here we prove the Main Theorems~\ref{MT:Dirmxe} and~\ref{MT:Diruni}, concerning the case of Dirichlet boundary conditions.
We will start by giving some remarks on the relation~\eqref{eq:proj_form_basis2} in Corollary~\ref{C:norm_proj}.

\subsection{The matrix~$[{}^{\tt d}\!\Theta(c)]$}
Hereafter we will take~$L=\pi$ and for simplicity we denote shortly~$L^2\coloneqq L^2(0,\pi)$.  We note that there is  essentially  no lack of generality by taking~$L=\pi$. This follows by a rescaling argument which is given in the
Appendix, Sect.~\ref{sS:ProofP:L=pi}.

Consider the ordered set ${}^{\tt d}\!\Ce_M:= ({}^{\tt d}\!e_1,{}^{\tt d}\!e_2,\dots,{}^{\tt d}\!e_M)$ of the first $M$ 
 eigenvalues and normalised eigenfunctions of the Laplacian with Dirichlet boundary conditions as in~\eqref{eigs-Dir}, that is, 
\begin{equation}\label{eigs-Dir-expl}
 {}^{\tt d}\!\alpha_i(x)=i^2\quad\mbox{and}\quad {}^{\tt d}\!e_i(x)=(\fractx{2}{\pi})^\frac{1}{2}\sin(i x),\qquad i\ge 1.
\end{equation}
Consider also the ordered set of the orthonormalised actuators as in~\eqref{act-eqlen} and~\eqref{act-disj},
\begin{equation}\label{eq:U_Morthnor}
\Cu_M(c):= (\overline 1_{\omega_1},\overline 1_{\omega_2},\dots,
\overline 1_{\omega_M}),\qquad \overline 1_{\omega_1}
=(\fractx{M}{r\pi})^\frac{1}{2} 1_{\omega_1}
\end{equation}  

We also denote~${}^{\tt d}\!E_M=\linspan\{{}^{\tt d}\!e_1,{}^{\tt d}\!e_2,\dots,{}^{\tt d}\!e_M\}$ and~$U(c)=\linspan\{1_{\omega_1},1_{\omega_2},\dots,1_{\omega_M}\}$.

Hereafter, We will use the results in Sect.~\ref{S:oblique_proj} with the pair~$(\Cu(c),{}^{\tt d}\!\Ce_M)$ in the role of~$(\Cf,\Cg)$,
with~$(U(c),{}^{\tt d}\!E_M)$ in the role of~$(F,G)$, and with~$L^2$ in the role of~$H$.

\begin{definition}
For each $c\in (0,L)^M$ as in~\eqref{act-eqlen} and~\eqref{act-disj},  with~$\Cu_M(c)$ as in~\eqref{eq:U_Morthnor}, we define 
\begin{equation}\label{ThetaM}
  [{}^{\tt d}\!\Theta(c)]\coloneqq[(\Cu_M(c),{}^{\tt d}\!\Ce_M)_{L^2}]^\top[(\Cu_M(c),{}^{\tt d}\!\Ce_M)_{L^2}].
\end{equation}
Further, the set of eigenvalues~$[{}^{\tt d}\!\Theta(c)]$ is denoted by ${\rm Eig}([{}^{\tt d}\!\Theta(c)])$, and we set
\begin{equation}\label{eq:vartheta}
  {}^{\tt d}\!\vartheta(c)\coloneqq\min{\rm Eig}([{}^{\tt d}\!\Theta(c)]).
  \end{equation}
\end{definition}
Recall that, according to ~\eqref{eq:proj_form_basis2} in Corollary~\ref{C:norm_proj}, we have for every disjoint  actuator positions $c\in [0,\pi]^M$,
\begin{equation}\label{eq:proj_eigenval}
  \norm{P_{U(c^M)}^{{}^{\tt d}\!E_M^\perp}}{}^2 = \frac{1}{{}^{\tt d}\!\vartheta( c^M )}.
\end{equation}

In order to prove the Main Theorems~\ref{MT:Dirmxe} and~\ref{MT:Diruni}, we will investigate~$\frac{1}{{}^{\tt d}\!\vartheta(c)}$.
For that we will start with the explicit expression for~$[({}^{\tt d}\!\Ce_M,\Cu_M(c))_{L^2}]$. For this purpose we first compute, for given $M$,
\ben\label{eq:ei_chij}
({}^{\tt d}\!e_i, \overline 1_{\omega_j})_{L^2}=
 (\fractx{2M}{r\pi^2})^\frac{1}{2}\frac{1}{ i}(\cos(i(c_j-\fractx{r\pi}{2M}))-\cos(i(c_j+\fractx{r\pi}{2M})))
 =(\fractx{8M}{r\pi^2})^\frac{1}{2}\frac{\sin(i\fractx{r\pi}{2M})\sin(ic_j)}{i}. 
\een
From now, to simplify the formulas we denote
\begin{equation}\label{deltaM}
 \delta_M\coloneqq\fractx{r\pi}{2M}. 
\end{equation}
 We have using \eqref{eq:ei_chij},
\begin{equation}\label{eq:cg_cf}
 [({}^{\tt d}\!\Ce_M,\Cu(c^M))_{L^2}]=(\fractx{8M}{r\pi^2})^\frac{1}{2}\begin{bmatrix}
                 \frac{\sin(1\delta_M)\sin(1c_1)}{1} &   \frac{\sin(1\delta_M)\sin(1c_2)}{1} &  \dots &
                 \frac{\sin(1\delta_M)\sin(1c_M)}{1}\\
                 \frac{\sin(2\delta_M)\sin(2c_1)}{2} &   \frac{\sin(2\delta_M)\sin(2c_2)}{2} &   \dots &
                 \frac{\sin(2\delta_M)\sin(2c_M)}{2}\\
                 \vdots& \vdots& \ddots &\vdots\\
		 \frac{\sin(M\delta_M)\sin(Mc_1)}{M} &   \frac{\sin(M\delta_M)\sin(Mc_2)}{M} &   \dots &
                 \frac{\sin(M\delta_M)\sin(Mc_M)}{M}
               \end{bmatrix}.
\end{equation}

\begin{lemma}\label{L:act_coincide}
  Given actuator positions $c=(c_1,\ldots, c_M)\in(0,L)^M$, if~$c_i=c_j$ for some~$i\ne j$,
  then the matrix~$[({}^{\tt d}\!\Ce_M,\Cu_M(c))_{L^2}]$ is singular.  
\end{lemma}
\begin{proof}
  Suppose that $c_i=c_j$ for some pair~$(i,j)\in\{1,2,\dots,M\}^2$, $i\ne j$. Then the $i$-th and $j$-th column of $[({}^{\tt d}\!\Ce_M,\Cu_M(c))_{L^2}]$
  coincide and thus this matrix is singular.
\end{proof}
Notice that by Lemma~\ref{L:equiv_nop}, if $[({}^{\tt d}\!\Ce_M,\Cu_M(c))_{L^2}]$ is singular then $L^2\ne U(c)\oplus {}^{\tt d}\!E_M^\perp$,
and the projection~$P_{U(c)}^{{}^{\tt d}\!E_M^\perp}$  is not well defined.

In \cite[Lem. 4.2]{KunRod-pp17} it was proven that we have the direct sum $L^2= U(c)\oplus{}^{\tt d}\!E_M^\perp$,  as soon as all actuators are disjoint.
Here we will prove that, for actuators as in~\eqref{act-eqlen} (i.e., with the same length and no common center
location~$c_i$, but not necessarily disjoint), then we still have the direct sum.

\begin{lemma}\label{L:dirsum_indplace}
  Let~$c=(c_1,c_2,\dots,c_M)$ satisfy~\eqref{act-eqlen} and $\omega_j:=(c_j-\delta_M,c_j+\delta_M)$, with $c_i\ne c_j$ for all~$i\ne j$.
  Then, the space $U(c):= \linspan\{1_{\omega_1},1_{\omega_2},\ldots, 1_{\omega_M}\}$ satisfies $L^2= U(c)\oplus {}^{\tt d}\!E_M^\perp$. 
\end{lemma}
\begin{proof}
  By Lemma~\ref{L:equiv_nop} it is sufficient to prove that the matrix~\eqref{eq:cg_cf} is
  invertible, which in turn is equivalent to the invertibility of $\Xi_M := [\sin(ic_j)]_{i,j=1,\ldots, M}$. 
This follows at once by dividing the $i$-th row of \eqref{eq:cg_cf}
 by~$(\fractx{8M}{r\pi^2})^\frac{1}{2}\frac{\sin(i \delta_M)}{i}\ne0$. Let~$v\in\bbR^{M\times1}$ satisfy~$v^\top\Xi_M=0$,
which means that the function
$
 f(x)\coloneqq\textstyle\sum\limits_{i=1}^M v_i\sin(ix)
$
vanishes for $x\in\{c_1,c_2,\dots,c_M\}$. By~\cite[Prop. 4.1]{KunRod-pp17},  necessarily~$f=0$,
because such a nonzero~$f$ can have at most
$M-1$ zeros in $(0,\pi)$. Therefore, we conclude that necessarily~$v=0$, and thus $\Xi_M$ is invertible.
\end{proof}

\subsection{Proof of the Main Theorems~\ref{MT:Dirmxe} and~\ref{MT:Diruni}}\label{sS:proofDir}
Observe that the Main Theorems~\ref{MT:Dirmxe} and~\ref{MT:Diruni} are, respectively, corollaries of the following theorems.
\begin{theorem}\label{T:Dirmxe}
Let~$r\in(0,1)$, and consider
 the location  $c=c^M$ of our actuators given by~\eqref{Dmxe}.
 Then~$[{}^{\tt d}\!\Theta(c^M)]$ is a diagonal matrix and its smallest eigenvalue is given by 
\[
 {}^{\tt d}\!\vartheta( c^M)=\frac{4M^2}{r\pi^2(M-1)^2}\sin^2\left(\frac{(M-1)r\pi}{2M}\right),\quad\mbox{for}\quad M\ge2,
\]
and is simple.
For~$M=1$ the only eigenvalue is given by~$\frac{8}{r\pi^2}\sin^2(\frac{r\pi}{2})$.

 In particular,~$\{{}^{\tt d}\!\vartheta(c^M )\}_{M\ge 1}$ decreases and converges to 
$
\frac{4}{r\pi^2}\sin^2\left(\frac{r\pi}{2}\right)
$.

\end{theorem}
 
\begin{theorem}\label{T:Diruni}
Let~$r\in(0,1)$, and consider
 the location  $c=c^M$  of our actuators given by~\eqref{Duni}.
 Then~$[{}^{\tt d}\!\Theta(c^M)]$ is a diagonal matrix and its smallest eigenvalue is given by
\[
  {}^{\tt d}\!\vartheta( c^M )=\frac{4(M+1)}{r\pi^2M}\sin^2\left(\frac{r\pi}{2}\right),\quad\mbox{for}\quad M\ge \fractx{r}{1-r},
\]
and is simple.

In particular~$\{{}^{\tt d}\!\vartheta( c^M )\}_{M\ge \fractx{r}{1-r}}$ decreases and converges to 
$
\frac{4}{r\pi^2}\sin^2\left(\frac{r\pi}{2}\right).
$
\end{theorem}
The corresponding proofs are given below. We start by presenting some auxiliary results.
\begin{lemma}\label{L:decsintt}
 Let $r\in(0,1)$ and~$M\ge 1$. Then the function~$t\mapsto\fractx{\sin^2(\delta_M t)}{t^2}$ is strictly decreasing
in the interval~$(0,M]$.
\end{lemma}
\begin{proof}
 Recalling~\eqref{deltaM},  observe that with~$r\in(0,1)$ and~$t\in(0,M]$
we have~$t\delta_M=t\frac{r\pi}{2M}\in(0,\frac{\pi}{2})$. With~$\varphi(t)\coloneqq\fractx{\sin(\delta_M t)}{t}>0$, defined for~$t\in (0,M]$, 
 \[\fractx{\ed}{\ed t} \varphi^2(t)= 2\varphi(t)\fractx{t\delta_M\cos(\delta_M t)-\sin(\delta_M t)}{t^2} 
 =  2\varphi(t) \cos(\delta_M t) \fractx{ t\delta_M-\tan(\delta_M t)  }{t^2} <0,
  \]
  where in the last inequality we used $\tan(s)>s$, for~$s\in (0,\fractx{\pi}{2})$, which is true since
  \[
  \tan(s)=\tan(s)-\tan(0) = \int_0^s\fractx{\ed}{\ed\tau}\tan(\tau)\;\ed\tau 
  = \int_0^s \fractx{1}{\cos^2(\tau)}\;\ed \tau >s,
  \]
  for $ s\in (0,\fractx{\pi}{2})$.
  \end{proof}

\begin{lemma}\label{L:sumcos0}
  For the actuator locations   $c^M=(c_1,\ldots, c_M)$ given by \eqref{Dmxe} we have 
\begin{equation}\label{sum.ext}
\sum\limits_{k=1}^M\cos(mc_k)
=0, \quad\mbox{for all}\quad m\in\{1,2,\dots,2M-1\}.
\end{equation} 
\end{lemma}
\begin{proof}
  Let us fix $M\ge 1$.  The result follows from the fact that, for constants~$a\in\R$ and~$b\in\R$, we have the identity
\begin{equation}\label{sumcos0}
\sum\limits_{n=0}^{M-1}\cos(a +nb)= \fractx{\sin(\fractx{Mb}{2})}{\sin(\frac{b}{2})}\cos(a+(M-1)\fractx{b}{2}),\quad\mbox{if}\quad \sin(\fractx{b}{2})\ne0,
\end{equation}
whose proof can be found in~\cite{Knapp09}. Indeed, it is enough to observe that
\begin{equation}\label{eq:sin}
 \sum\limits_{k=1}^M\cos(mc_k) =\sum\limits_{n=0}^{M-1}\cos(mc_{n+1}) =\sum\limits_{n=0}^{M-1}\cos(m\fractx{(2(n+1)-1)\pi}{2M})
 =\sum\limits_{n=0}^{M-1}\cos(\fractx{m\pi}{2M}+n\fractx{m\pi}{M}).
\end{equation}
Hence using~\eqref{sumcos0} with $a:= \fractx{m\pi}{2M}$ and $b:=\fractx{m\pi}{M}$ we obtain 
\[
 \sum\limits_{n=0}^{M-1}\cos(\fractx{m\pi}{2M}+n\fractx{m\pi}{M}) =\fractx{\sin(\fractx{Mm\pi}{2M})}{\sin(\fractx{m\pi}{2M})}\cos(\fractx{m\pi}{2M}+(M-1)\fractx{m\pi}{2M})
=\fractx{\sin(\fractx{m\pi}{2})}{\sin(\fractx{m\pi}{2M})}\cos(\fractx{m\pi}{2})=\fractx{\sin(m\pi)}{2\sin(\fractx{m\pi}{2M})}=0,
\]
which together with \eqref{eq:sin} completes the proof. 
\end{proof}

\begin{proof}[Proof of Theorem~\ref{T:Dirmxe}]
  Let~$i$ and~$j$ be given in~$\{1,2,\dots,M\}$, with~$i>j$. For the $(i,j)$-entry of~$[{}^{\tt d}\!\Theta(c)]$, we find
\begin{align}
  {}^{\tt d}\!\Theta(c)_{ij}
&= \fractx{8M}{r\pi^2}\fractx{\sin(i\delta_M)\sin(j\delta_M)}{ij}\sum\limits_{k=1}^M\sin(ic_k)\sin(jc_k)\notag\\
&=\fractx{4M}{r\pi^2}\fractx{\sin(i\delta_M)\sin(j \delta_M)}{ij}\sum\limits_{k=1}^M\left(\cos((i-j)c_k)-\cos((i+j)c_k)\right)\label{orthext1},
\end{align}
where we used $2\sin(ic_k)\sin(jc_k)=\cos((i-j)c_k)-\cos((i+j)c_k)$.
Since $1<i\pm j<2M$, from~\eqref{sum.ext} it follows that
$[{}^{\tt d}\!\Theta(c)]_{ij}=0$, for all~$i>j$. Since~$[{}^{\tt d}\!\Theta(c)]$ is symmetric,
it follows that~$[{}^{\tt d}\!\Theta(c)]_{ij}=0$, for all~$i\ne j$. That is, the
matrix~$[{}^{\tt d}\!\Theta(c)]$ is diagonal. 

The eigenvalues of~$[{}^{\tt d}\!\Theta(c)]$ are given by the elements in its diagonal, which are
\begin{equation}\label{eig.ext}
  {\rm Eig}([{}^{\tt d}\!\Theta(c)])=\left\{\fractx{8M}{r\pi^2}\fractx{\sin^2(i\delta_M)}{i^2}{\textstyle\sum\limits_{k=1}^M}\sin^2(ic_k)\;\Bigl|\;
i\in\{1,2,\dots,M\}\right\}. 
\end{equation}
Now we observe that~$\sum\limits_{k=1}^M\sin^2(ic_k)=\frac{1}{2}\sum\limits_{k=1}^M(1-\cos(2ic_k))$, then using again~\eqref{sum.ext} (in case~$i<M$)
we obtain
\begin{align*}
\sum\limits_{k=1}^M\sin^2(ic_k)&
=\left\{\begin{array}{ll}\frac{M}{2},\qquad &i< M,\\
\frac{M}{2}-\frac{1}{2}\sum\limits_{k=1}^M\cos((2k-1)\pi)=M,\qquad &i=M.
\end{array}\right.
\end{align*}
Hence, from~\eqref{eig.ext} we see that the eigenvalues are
\begin{equation}\label{eig.ext1}
  {\rm Eig}([{}^{\tt d}\!\Theta(c)])=\left\{\fractx{4M^2}{r\pi^2}\fractx{\sin^2(\fractx{i r\pi}{2M})}{i^2}\;\Bigl|\; i\in\{1,2,\dots,M-1\}\right\}{\textstyle\bigcup}
 \left\{
 \fractx{8}{r\pi^2}\sin^2(\fractx{r\pi}{2})\right\}.
\end{equation}

Now, from Lemma~\ref{L:decsintt}, we have that the sequence~$i\mapsto\fractx{\sin^2(i \delta_M)}{i^2}$ is strictly decreasing, $1\le i\le M-1$,
which implies that we have only two possibilities for the smallest eigenvalue~${}^{\tt d}\!\vartheta(c)$, more precisely, 
\[
  {}^{\tt d}\!\vartheta(c)=\min\left\{\fractx{8}{r\pi^2}\fractx{M^2}{2(M-1)^2}\sin^2((M-1) \delta_M), \; \fractx{8}{r\pi^2}\sin^2(\fractx{r\pi}{2})\right\}.
\]

We now distinguish the four cases $M=1,2,3$, and $M\ge 4$.\\ 
\noindent
\underline{$M=1$}: In this case the only eigenvalue is~${}^{\tt d}\!\vartheta(c)=\fractx{8}{r\pi^2}\sin^2(\fractx{r\pi}{2})$ which is trivially simple.\\
\noindent
\underline{$M=2$}: In this situation, we find that
\begin{equation*}
\fractx{8}{r\pi^2}\sin^2(\fractx{r\pi}{2})>\fractx{4M^2}{r\pi^2(M-1)^2}\sin^2((M-1)\delta_M)
\quad \Longleftrightarrow \quad  8\sin^2(\fractx{r\pi}{2})> 16\sin^2(\fractx{r\pi}{4}) \quad \Longleftrightarrow \quad  \sqrt2\cos(\fractx{r\pi}{4})>1.
\end{equation*}
The last inequality holds true because~$0<r<1$, thus we conclude that~${}^{\tt d}\!\vartheta(c)=\fractx{4M^2}{r\pi^2(M-1)^2}\sin^2((M-1)\delta_M)$
and that~${}^{\tt d}\!\vartheta(c)$ is simple. \\
\noindent
\underline{$M=3$}: In this case we find that
\begin{align*}
\fractx{8}{r\pi^2}\sin^2(\fractx{r\pi}{2})>\fractx{4M^2}{r\pi^2(M-1)^2}\sin^2((M-1) \delta_M) \quad 
\Longleftrightarrow \quad 8\sin^2(\fractx{r\pi}{2})>9\sin^2(\fractx{r\pi}{3}).
\end{align*}
To see that the  last  inequality from the previous line holds true for all for~$r\in(0,1)$ we compute 
\begin{align*}
8\sin^2(\fractx{r\pi}{2})=9\sin^2(\fractx{r\pi}{3})&\quad\Longleftrightarrow\quad 2\sqrt2\sin(\fractx{r\pi}{3}+\fractx{r\pi}{6})=3\sin(\fractx{r\pi}{3})\\
&\quad\Longleftrightarrow\quad 2\sqrt2\left(\sin(\fractx{r\pi}{3})\cos(\fractx{r\pi}{6})+\sin(\fractx{r\pi}{6})\cos(\fractx{r\pi}{3})\right)
=6\sin(\fractx{r\pi}{6})\cos(\fractx{r\pi}{6})\\
&\quad\Longleftrightarrow\quad \sin(\fractx{r\pi}{6})\left(4\sqrt2\cos^2(\fractx{r\pi}{6})+2\sqrt2\left(2\cos^2(\fractx{r\pi}{6})-1\right)
-6\cos(\fractx{r\pi}{6})\right)=0\\
&\quad\Longleftrightarrow\quad 8\sqrt2\cos^2(\fractx{r\pi}{6})-6\cos(\fractx{r\pi}{6})-2\sqrt2=0
\quad\Longleftrightarrow\quad\cos(\fractx{r\pi}{6})=\fractx{6\pm\sqrt{164}}{16\sqrt2},
\end{align*}
from which we conclude, since~$0<\fractx{r\pi}{6}<\fractx{\pi}{6}$, that
\begin{equation*}
 8\sin^2(\fractx{r\pi}{2})=9\sin^2(\fractx{r\pi}{3})
\quad \Longrightarrow \quad\cos(\fractx{r\pi}{6})=\fractx{3+\sqrt{41}}{8\sqrt2}.
\end{equation*}
Next we observe that~$\fractx{3+\sqrt{41}}{8\sqrt2}<\fractx{\sqrt{3}}{2}$, which by elementary manipulations is equivalent to $369<529$.
Therefore we can conclude that
\begin{align*}
8\sin^2(\fractx{r\pi}{2})=9\sin^2(\fractx{r\pi}{3}) \quad \Longrightarrow \quad \cos(\fractx{r\pi}{6})<\fractx{\sqrt{3}}{2}=\cos(\fractx{\pi}{6}),
\end{align*}
which is not possible with~$r\in(0,1)$. Hence,
we have~$8\sin^2(\fractx{r\pi}{2})\ne 9\sin^2(\fractx{r\pi}{3})$, which shows that the first eigenvalue is simple. Furthermore, 
for the function~$g(r)\coloneqq 8\sin^2(\fractx{r\pi}{2})- 9\sin^2(\fractx{r\pi}{3})$, $r\in(0,1)$, we find that~$g(\frac{1}{2})=4- \fractx{9}{4}>0$.
Necessarily~$g(r)>0$ in the entire interval~$(0,1)$, because~$g$ has no zeros in~$(0,1)$. That is, we
have~${}^{\tt d}\!\vartheta(c)=\fractx{4M^2}{r\pi^2(M-1)^2}\sin^2((M-1)\delta_M)$.  Finally~${}^{\tt d}\!\vartheta(c)$ is simple due to
Lemma~\ref{L:decsintt}.\\
\noindent
\underline{$M\ge 4$}: In this case we have 
\begin{equation*}
M\ge4 \quad \Longrightarrow\quad  \left(2>\fractx{M^2}{(M-1)^2} \quad \mbox{and} \quad 1>\fractx{M-1}{M}\right) \quad \Longrightarrow \quad
\fractx{8}{r\pi^2}\sin^2(\fractx{r\pi}{2})>\fractx{4M^2}{r\pi^2(M-1)^2}\sin^2((M-1) \fractx{r\pi}{2M} ).
\end{equation*}
It follows that the smallest eigenvalue satisfies~${}^{\tt d}\!\vartheta(c)=\fractx{4M^2}{r\pi^2(M-1)^2}\sin^2((M-1)\delta_M)$, if~$M\ge4$.
 Using Lemma~\ref{L:decsintt}, it follows that~${}^{\tt d}\!\vartheta(c)$ is simple.
\end{proof}

\begin{proof}[Proof of Theorem~\ref{T:Diruni}]
Let~$r\in(0,\frac{M}{M+1}]\subset(0,1)$, and let~$i$ and~$j$ be
given in~$\{1,2,\dots,M\}$, with~$i>j$. We proceed as in the proof of Theorem~\ref{T:Dirmxe}.
For the entry~$[{}^{\tt d}\!\Theta(c)]_{ij}$ of~$[{}^{\tt d}\!\Theta(c)]$ we find
\begin{align}
  [{}^{\tt d}\!\Theta(c)]_{ij}
&=\fractx{4M}{r\pi^2}\fractx{\sin(\fractx{i r\pi}{2M})\sin(\fractx{j r\pi}{2M})}{ij}\sum\limits_{k=1}^M\left(\cos((i-j)c_k)-\cos((i+j)c_k)\right)
\label{orthuni1}
\end{align}
and for any given~$m\ge 1$ such that~$1\le m\le 2M$,
we find that
\begin{align*}
\sum\limits_{k=1}^M\cos(mc_k)&=\sum\limits_{k=1}^M\cos(\fractx{mk\pi}{M+1})
=\sum\limits_{n=0}^{M-1}\cos(\fractx{m\pi}{M+1}+n\fractx{m\pi}{M+1}).
\end{align*}
Since~$0<\fractx{m\pi}{2(M+1)}<\frac{\pi}{2}$, for all~$m\le 2M$, we can use formula~\eqref{sumcos0}
to derive that
\begin{equation*}
\sum\limits_{k=1}^M\cos(mc_k)
=\fractx{\sin(\fractx{Mm\pi}{2(M+1)})}{\sin(\fractx{m\pi}{2(M+1)})}\cos(\fractx{m\pi}{M+1}+(M-1)\fractx{m\pi}{2(M+1)})
=\fractx{\sin(\fractx{Mm\pi}{2(M+1)})}{\sin(\fractx{m\pi}{2(M+1)})}\cos(\fractx{m\pi}{2}),\qquad 1\le m\le2M,
\end{equation*}
from which we obtain
\begin{subequations}\label{sum.uni}
 \begin{align}
\sum\limits_{k=1}^M\cos(mc_k)&=0,\quad &&\mbox{if~$m$ is odd}.\\
\sum\limits_{k=1}^M\cos(mc_k)&=\fractx{\sin(\fractx{m}{2}\fractx{(M+1-1)\pi}{M+1})}{\sin(\fractx{m\pi}{2(M+1)})}(-1)^\frac{m}{2}
=-\fractx{\sin(\fractx{m}{2}\fractx{\pi}{M+1})}{\sin(\fractx{m\pi}{2(M+1)})}=-1,\quad &&\mbox{if~$m$ is even}.
\end{align}
\end{subequations}

Therefore, from~\eqref{orthuni1} it follows that
$[{}^{\tt d}\!\Theta(c)]_{ij}=0$, for all~$i>j$, because $i-j$ is even if, and only if,~$i+j$ is even. The symmetry of~$[{}^{\tt d}\!\Theta(c)]$
implies that the
matrix~$[{}^{\tt d}\!\Theta(c)]$ is diagonal, and its eigenvalues are
\begin{equation}\label{eig.uni}
  {\rm Eig}([{}^{\tt d}\!\Theta(c)])=\left\{\fractx{8M}{r\pi^2}\fractx{\sin^2(\fractx{i r\pi}{2M})}{i^2}{\textstyle\sum\limits_{k=1}^M}\sin^2(ic_k)\;\Bigl|\;
i\in\{1,2,\dots,M\}\right\}. 
\end{equation}
Writing~$\sum\limits_{k=1}^M\sin^2(ic_k)=\frac{1}{2}\sum\limits_{k=1}^M(1-\cos(2ic_k))$, by~\eqref{sum.uni}
we obtain
\begin{align*}
\sum\limits_{k=1}^M\sin^2(ic_k)&
=\frac{M+1}{2}
\end{align*}
and from~\eqref{eig.uni} we see that the eigenvalues are
\begin{equation}\label{eig.uni1}
  {\rm Eig}([{}^{\tt d}\!\Theta(c)])=\left\{\fractx{4M(M+1)}{r\pi^2}\fractx{\sin^2(\fractx{i r\pi}{2M})}{i^2}\;\Bigl|\; i\in\{1,2,\dots,M\}\right\},
\end{equation}
and from Lemma~\ref{L:decsintt}, we have that
\[
  {}^{\tt d}\!\vartheta(c^M)= \min\limits_{1\le i\le M}\fractx{4M(M+1)}{r\pi^2}\fractx{\sin^2(\fractx{i r\pi}{2M})}{i^2} =\fractx{4(M+1)}{r\pi^2M}\sin^2(\fractx{r\pi}{2}),
\]
and all the eigenvalues are simple.
\end{proof}


\section{The case of Neumann boundary conditions}\label{S:Neumann}
 In this section we present the proof of the Main Theorems~\ref{MT:Neumxe}.  Similarly to the Dirichlet case we consider the
indicator actuators~$\Cu_M(c)$ as in~\eqref{eq:U_Morthnor}. Now we consider the set of eigenvalues and orthonormalised eigenfunctions of the Laplacian
in~$(0,\pi)$ under Neumann boundary conditions,  which are given by  
\[{}^{\tt n}\!\alpha_i\coloneqq
 (i-1)^2,\quad i\ge 1;
\qquad\mbox{and}\quad {}^{\tt n}\!e_i\coloneqq
\begin{cases}
 (\frac{1}{\pi})^\frac{1}{2}\cos((i-1) x) ,&\mbox{if }i=1,\\
 (\frac{2}{\pi})^\frac{1}{2}\cos((i-1) x),&\mbox{if }i\ge2.
\end{cases}
\]
We set the ordered set~${}^{\tt n}\!\Ce_M\coloneqq({}^{\tt n}\!e_1,{}^{\tt n}\!e_2,\dots,{}^{\tt n}\!e_M)$ consisting of the
 of the first $M$ orthonormalised eigenfunctions and the linear spans
\[
U_M=\linspan\{1_{\omega_1},1_{\omega_2},\dots,1_{\omega_M}\},\qquad
{}^{\tt n}\!E_M\coloneqq\linspan\{{}^{\tt n}\!e_1,{}^{\tt n}\!e_2,\dots,{}^{\tt n}\!e_M\}.
\]

In this case we find
 \[
({}^{\tt n}\!e_i,\overline 1_{\omega_j})_{L^2}=
\begin{cases}
 (\frac{M}{r\pi^2})^\frac{1}{2}\frac{r\pi}{M}=(\frac{r}{M})^\frac{1}{2},&\mbox{if }i=1;\\
 (\frac{8M}{r\pi^2})^\frac{1}{2}\frac{\sin((i-1) \delta_M )\cos((i-1)c_j)}{i-1},&\mbox{if }i\ge2.
\end{cases}.
 \]

 We have the analogous of Lemma~\ref{L:dirsum_indplace}, for Neumann boundary conditions.
 Again we denote~$L^2=L^2(0,\pi)$, and~$\delta_M=\frac{r\pi}{2M}$.

\begin{lemma}\label{L:dirsum_indplace.Neu}
  Let~$c=(c_1,c_2,\dots,c_M)\in(0,\pi)^M$ and $\omega_j$ be as in~\eqref{act-eqlen},  with $c_i\ne c_j$ for all~$i\ne j$. 
  Then we have $L^2= U_M(c)\oplus {}^{\tt n}\!E_M^\perp$. 
\end{lemma}
\begin{proof}
  By Lemma~\ref{L:equiv_nop} it is sufficient to prove that the matrix~$[({}^{\tt n}\!\Ce_M,\Cu_M(c))_{L^2}]$ is invertible,
 which in turn is equivalent to the invertibility of $
   \Xi_M\coloneqq      [ \cos((i-1)c_j)]$, because we may divide the $1$-st row of~$[({}^{\tt n}\!\Ce_M,\Cu_M(c))_{L^2}]$
   by ~$(\frac{r}{M})^\frac{1}{2}$ and the $i$-th row, $i>1$, by
~$(\fractx{8M}{r\pi^2})^\frac{1}{2}\frac{\sin((i-1) \delta_M)}{i-1}\ne0$.

Let~$v\in\bbR^{1\times M}$ satisfy~$v\Xi_M=0$,
which means that the function
$
 f(x)\coloneqq\textstyle\sum\limits_{i=1}^M v_i\cos((i-1)x)
$
vanishes for $x\in\{c_1,c_2,\dots,c_M\}$. Necessarily the smooth function~$f$ must have a critical value~$\xi_i\in(c_i,c_{i+1})$,
$i\in\{1,2,\dots,M-1\}$. That is, $g(x)\coloneqq\frac{\ed}{\ed x}f(x)=-\textstyle\sum\limits_{j=1}^{M-1} jv_{j+1}\sin(jx)$ vanishes for
$x\in\{\xi_1,\xi_2,\dots,\xi_{M-1}\}$. Observe that $\xi_n\ne\xi_m$, for all~$n\ne m$, because $c_i\ne c_j$, for all~$i\ne j$.

By
~\cite[Prop. 4.1]{KunRod-pp17}, we can conclude that necessarily~$g=0$. Therefore we conclude that ~$f(x)=v_1\cos(0x)=v_1$ is constant.
Necessarily~$f=0$ because
$f(c_i)=0$. Therefore~$v=0$, and thus $\Xi_M$ is invertible.
\end{proof}

Now consider the matrix~$[{}^{\tt n}\!\Theta(c)]=[(\Cu_M(c),{}^{\tt n}\!\Ce_M)_{L^2}]^\top[(\Cu_M(c),{}^{\tt n}\!\Ce_M)_{L^2}]$.

Observe that the Main Theorem~\ref{MT:Neumxe} is a corollary of the following theorem.
\begin{theorem}\label{T:Neumxe}
Let~$r\in(0,1)$, and consider
 the location $c^M$ of our actuators given by~\eqref{Dmxe}.
 Then~$[{}^{\tt n}\!\Theta( c^M)]$ is a diagonal matrix and its smallest eigenvalue is given by
\[
 {}^{\tt n}\!\vartheta( c^M)=\fractx{4M^2}{r\pi^2(M-1)^2}\sin^2((M-1)\delta_M),\quad\mbox{for}\quad M\ge2,
\]
and is simple.
For~$M=1$ the only eigenvalue is~$1$.

In particular~$\{{}^{\tt n}\!\vartheta( c^M)\}_{M\ge 1}$ decreases and converges to 
$
\frac{4}{r\pi^2}\sin^2\left(\frac{r\pi}{2}\right)
$. 
\end{theorem}

\begin{proof}
We compute now the $(i,j)$-entry of~$[{}^{\tt n}\!\Theta(c)]$. Let us first consider the case~$i>j=1$. Then,
for the entry~$[{}^{\tt n}\!\Theta(c)]_{i1}$, using~\eqref{sum.ext} we find
\begin{equation}\label{ThetaNeui1}
  [{}^{\tt n}\!\Theta(c)]_{i1}
= \fractx{\sqrt8}{\pi}\fractx{\sin(\fractx{(i-1) r\pi}{2M})}{i-1}\sum\limits_{k=1}^M\cos((i-1)c_k)=0.
\end{equation}
Now for any given~$m\ge 1$, such that~$1\le m<2M$,
we find that
\begin{align*}
\sum\limits_{k=1}^M\cos(mc_k)&=\sum\limits_{k=1}^M\cos(\fractx{m(2k-1)\pi}{2M})=\sum\limits_{n=0}^{M-1}\cos(\fractx{m(2n+1)\pi}{2M})
=\sum\limits_{n=0}^{M-1}\cos(\fractx{m\pi}{2M}+n\fractx{m\pi}{M}).
\end{align*}
Using again Lemma~\ref{L:sumcos0} we obtain
\begin{equation}\label{diag.Neu1}
  [{}^{\tt n}\!\Theta(c)]_{i1}=0,\quad\mbox{for all}\quad i>1.
\end{equation}

Let us now consider the case~$i>j>1$. Then, for the entry~$[{}^{\tt n}\!\Theta(c)]_{ij}$, we find
\begin{align*}
  [{}^{\tt n}\!\Theta(c)]_{ij}
&= \fractx{8M}{r\pi^2}\fractx{\sin((i-1) \delta_M)\sin((j-1) \delta_M)}{(i-1)(j-1)}\sum\limits_{k=1}^M\cos((i-1)c_k)\cos((j-1)c_k)
\end{align*}
and using Lemma~\ref{L:sumcos0} once more we get,
\begin{align*}
\sum\limits_{k=1}^M\cos((i-1)c_k)\cos((j-1)c_k)=\frac{1}{2}\sum\limits_{k=1}^M(\cos((i-j)c_k)+\cos((i+j-2)c_k))=0.
\end{align*}
Therefore
\begin{equation}\label{diag.Neu2}
  [{}^{\tt n}\!\Theta(c)]_{ij}=0,\quad\mbox{for all}\quad i>j>1.
\end{equation}
From~\eqref{diag.Neu1}, \eqref{diag.Neu2}, and the symmetry of~$[{}^{\tt n}\!\Theta(c)]$, it follows that~$[{}^{\tt n}\!\Theta(c)]$ is a diagonal matrix,
and its eigenvalues are its diagonal elements
The eigenvalues of~$[{}^{\tt n}\!\Theta(c)]$ are given by the elements in its diagonal, which are
\begin{equation}\label{eig.ext.Neu}
  {\rm Eig}([{}^{\tt n}\!\Theta(c)])=\left\{r\right\}\textstyle\bigcup\left\{\fractx{8M}{r\pi^2}\fractx{\sin^2((i-1)\delta_M)}{(i-1)^2}{\textstyle\sum\limits_{k=1}^M}\cos^2((i-1)c_k)\;\Bigl|\;
i\in\{2,\dots,M\}\right\}. 
\end{equation}
Now we observe that~$\sum\limits_{k=1}^M\cos^2((i-1)c_k)=\frac{1}{2}\sum\limits_{k=1}^M(1+\cos(2(i-1)c_k))$,
then using again~\eqref{sum.ext} (in case~$i>1$), 
\begin{equation}\label{eig.ext1.Neu}
  {\rm Eig}([{}^{\tt n}\!\Theta(c)])=\left\{r\right\}\textstyle\bigcup\left\{\fractx{4M^2}{r\pi^2}\fractx{\sin^2((i-1)\delta_M)}{(i-1)^2}
 \;\Bigl|\;
i\in\{2,\dots,M\}\right\}.
\end{equation}
In view of Lemma~\ref{L:decsintt}, we conclude that ${}^{\tt n}\!\vartheta(c)=\min\left\{r,\fractx{4M^2}{r\pi^2(M-1)^2}\sin^2((M-1) \delta_M)\right\}$. Now, from
\[\fractx{4M^2}{r\pi^2(M-1)^2}\sin^2((M-1)\delta_M)=r\left(\fractx{2M}{r\pi(M-1)}\right)^2\sin^2\left(\fractx{(M-1)r\pi}{2M}\right)
<r\]
because~$\frac{\sin(x)}{x}<1$ for all~$x\in(0,\frac{\pi}{2}]$ and~$0<\fractx{(M-1)r\pi}{2M}<\frac{\pi}{2}$, we can conclude that
${}^{\tt n}\!\vartheta(c)=\fractx{4M^2}{r\pi^2(M-1)^2}\sin^2((M-1) \delta_M)$.
In particular~${}^{\tt n}\!\vartheta(c)$ is simple. Finally, for~$M=1$ we have that~$[{}^{\tt n}\!\Theta(c)]=[r]$,  whose only eigenvalue is~$r$.
\end{proof}

\begin{remark}\label{R:compDirNeu}
Comparing Theorems~\ref{T:Dirmxe} and~\ref{T:Neumxe} we see that the smallest eigenvalue do coincide for Dirichlet and Neumann boundary conditions with
the exception of the case~$M=1$. It is also interesting to observe that in Theorems~\ref{T:Dirmxe}, \ref{T:Diruni}, and~\ref{T:Neumxe}, the limit of the
smallest eigenvalues as $M$ increases is the same, and equals~$\frac{4}{r\pi^2}\sin^2\left(\frac{r\pi}{2}\right)$. However, we must
underline that for the location~\eqref{Duni} under Neumann boundary conditions, we could not find the expression for the smallest eigenvalue.
Numerical simulations that we will
present later in  Section~\ref{S:OtherLoc}, show that that eigenvalue is considerably smaller in the latter setting.
We should mention also that finding the eigenvalues
explicitly was possible in Theorems~\ref{T:Dirmxe}, \ref{T:Diruni}, and~\ref{T:Neumxe}, due to the fact that the corresponding
matrices~$[\Theta(c^M)]$ are diagonal. However, for a general location those matrices are not necessarily diagonal, and so
finding the smallest eigenvalue becomes a more difficult problem. This is actually the reason we could not find the explicit expression for the smallest
eigenvalue for the location~\eqref{Duni} under Neumann boundary conditions.
More comments are given in  Sect.~\ref{S:OtherLoc}.
\end{remark}

\section{Other particular locations for the actuators}\label{S:OtherLoc}
We investigate here also the location
\begin{align}\stepcounter{equation}
c_j^M&=\frac{(1-r)\pi}{2}+\frac{(2j-1)r\pi}{2M},
\qquad\mbox{with}\qquad  j\in\{1,2,\dots,M\}. \tag{(\theequation)--{\tt con}}\label{Dcon}
\end{align}
which have also been studied numerically in~\cite[section~4.7]{KunRod-pp17}, and where the results of numerical
 simulations, for Dirichlet boundary conditions,
 have shown that, with~$c=c^M$, the smallest eigenvalue~${}^{\tt d}\!\vartheta(c)$ takes very small values
 and are not conclusive to decide whether~${}^{\tt d}\!\vartheta(c)$ remains away from zero or not.
 
 Recall that in the settings in Sections~\ref{sS:proofDir} and~\ref{S:Neumann}
 the matrices~$[{}^{\tt d}\!\Theta(c)]$ and~$[{}^{\tt n}\!\Theta(c)]$ are diagonal.  
 Below, we show that for the location~\eqref{Dcon} 
 the matrices~$[{}^{\tt d}\!\Theta(c)]$ and~$[{}^{\tt n}\!\Theta(c)]$ are not necessarily diagonal.
 In such cases finding the eigenvalues can be a difficult problem, because they are not necessarily
 the entries of the diagonal. That is why it seems difficult to
 find an analytical expression for the first eigenvalue.

 \begin{theorem}\label{T:Dircon}
   Under Dirichlet boundary conditions, with the location as in~\eqref{Dcon}, the matrix~$[{}^{\tt d}\!\Theta(c^M)]$ is not necessarily diagonal.
\end{theorem}

\begin{proof}
First of all we observe that, in the case~$M=2$ and for
a symmetric position~$c=c^2=(c_1,c_2)=(c_1,\pi-c_1)$ as in~\eqref{Dcon}, we have that
 the matrix~$[{}^{\tt d}\!\Theta(c^2)]$ is diagonal, this is a consequence of~\eqref{orthext1}, from which we have that
the entry in the~$2$-nd row and~$1$-st column, of~$[{}^{\tt d}\!\Theta(c^2)]$, is proportional to
\[ 
\sum\limits_{k=1}^2\left(\cos(c_k)-\cos(3c_k)\right)=\left(\cos(c_1)-\cos(3c_1)\right)+\left(\cos(c_2)-\cos(3c_2)\right)=0.
\]

Now we prove that in the case~$M=3$ and~$r=\frac{1}{2}$ the matrix~$[{}^{\tt d}\!\Theta(c^3)]$ is not diagonal.
In this case for the entry in the~$1$-st row and~$3$-rd column we find
 \begin{align*}
[{}^{\tt d}\!\Theta(c^3)]_{ij}
&=\fractx{24}{r\pi^2}\fractx{\sin(\fractx{1 r\pi}{6})\sin(\fractx{3 r\pi}{6})}{3}\sum\limits_{k=1}^3\sin(1c_k)\sin(3c_k)
=\fractx{4\sin(\fractx{1 r\pi}{6})\sin(\fractx{3 r\pi}{6})}{r\pi^2}\sum\limits_{k=1}^3\left(\cos(2c_k)-\cos(4c_k)\right).
\end{align*}
By computing the sums
\begin{subequations}\label{sumcon}
 \begin{align}
\sum\limits_{k=1}^3\cos(2c_k)&=\sum\limits_{k=1}^3\cos(\fractx{\pi}{2}+\fractx{(2j-1)\pi}{2M})=-\sum\limits_{k=1}^3\sin(\fractx{(2j-1)\pi}{6})
 =-2,\\
  \sum\limits_{k=1}^3\cos(4c_k)&=\sum\limits_{k=1}^3\cos(\pi+\fractx{(2j-1)\pi}{M})
 =-\sum\limits_{k=1}^3\cos(\fractx{(2j-1)\pi}{3})=0,
 \end{align}
\end{subequations}
we find~$[{}^{\tt d}\!\Theta(c^3)]_{ij}
=-\fractx{16\sin(\fractx{\pi}{12})\sin(\fractx{\pi}{4})}{\pi^2}\ne0$, with~$(i,j)=(1,3)$, which shows that~$[{}^{\tt d}\!\Theta(c^3)]$ is not a diagonal matrix.
\end{proof}

 \begin{theorem}\label{T:Neucon}
   Under Neumann boundary conditions, with the location as in~\eqref{Dcon}, the matrix~$[{}^{\tt n}\!\Theta(c^M)]$ is not necessarily diagonal.
\end{theorem}

\begin{proof}
Again, in the case~$M=2$ and for
a symmetric position~$c=c^2=(c_1,c_2)=(c_1,\pi-c_1)$ as in~\eqref{Dcon}, we have that
 the matrix~$[{}^{\tt n}\!\Theta(c^2)]$ is diagonal, as a consequence of~\eqref{ThetaNeui1}, from which we have that
the entry in the~$2$-nd row and~$1$-st column, of~$[{}^{\tt n}\!\Theta(c^2)]$, is proportional to
\[ 
\sum\limits_{k=1}^M\cos(c_k)=\cos(c_1)+\cos(c_2)=0.
\]

Now we prove that in the case~$M=3$ and~$r=\frac{1}{2}$ the matrix~$[{}^{\tt n}\!\Theta(c^3)]$ is not diagonal. 
In this case for the entry in the~$1$-st row and~$3$-rd column we find
 \begin{align*}
[{}^{\tt n}\!\Theta(c^3)]_{ij}
&=\fractx{24}{r\pi^2}\fractx{\sin(\fractx{1 r\pi}{6})\sin(\fractx{3 r\pi}{6})}{3}\sum\limits_{k=1}^3\cos(1c_k)\cos(3c_k)
=\fractx{4\sin(\fractx{1 r\pi}{6})\sin(\fractx{3 r\pi}{6})}{r\pi^2}\sum\limits_{k=1}^3\left(\cos(2c_k)+\cos(4c_k)\right).
\end{align*}
By~\eqref{sumcon}
we find~$[{}^{\tt n}\!\Theta(c^3)]_{ij}
=-\fractx{16\sin(\fractx{\pi}{12})\sin(\fractx{\pi}{4})}{\pi^2}\ne0$, with~$(i,j)=(1,3)$. Thus~$[{}^{\tt n}\!\Theta(c^3)]$ is not a diagonal matrix.
\end{proof}

The following theorem illustrates a difference between Dirichlet and Neumann boundary conditions. Recall that from Theorem~\ref{T:Diruni},
for the former boundary conditions the matrix is diagonal.

\begin{theorem}\label{T:Neuuni}
  Under Neumann boundary conditions, with the location as in~\eqref{Duni}, the matrix~$[{}^{\tt n}\!\Theta(c^M)]$ is not necessarily diagonal.
\end{theorem}

\begin{proof}
The location $c=c^2=(c_2,c_2)$ as in~\eqref{Duni} is symmetric, $c_2=\pi-c_1$, so proceeding as in as in the proof of 
Theorem~\ref{T:Neucon} we have that, in the case~$M=2$, 
 the matrix~$[{}^{\tt n}\!\Theta(c^2)]$ is diagonal.
 
Now we prove that in the case~$M=3$ and~$r=\frac{1}{2}$ the matrix~$[{}^{\tt n}\!\Theta(c^3)]$ is not diagonal. 
In this case  for the entry in the~$1$-st row and~$3$-rd column we find
 \begin{align*}
[{}^{\tt n}\!\Theta(c^3)]_{ij}
&=\fractx{24}{r\pi^2}\fractx{\sin(\fractx{1 r\pi}{6})\sin(\fractx{3 r\pi}{6})}{3}\sum\limits_{k=1}^3\cos(1c_k)\cos(3c_k)
=\fractx{4\sin(\fractx{1 r\pi}{6})\sin(\fractx{3 r\pi}{6})}{r\pi^2}\sum\limits_{k=1}^3\left(\cos(2c_k)+\cos(4c_k)\right).
\end{align*}
By computing the sums
\begin{subequations}\label{sumuniN}
 \begin{align}
\sum\limits_{k=1}^3\cos(2c_k)&=\sum\limits_{k=1}^3\cos(\fractx{k\pi}{2})=-1\\
  \sum\limits_{k=1}^3\cos(4c_k)&=\sum\limits_{k=1}^3\cos(k\pi)
 =-1;
 \end{align}
\end{subequations}
we find~$[{}^{\tt n}\!\Theta(c^3)]_{ij}
=-\fractx{16\sin(\fractx{\pi}{12})\sin(\fractx{\pi}{4})}{\pi^2}\ne0$, with~$(i,j)=(1,3)$,
which shows that~$[{}^{\tt n}\!\Theta(c^3)]$ is not a diagonal matrix.
\end{proof}

\subsubsection*{Comparison between the different locations} \label{sS:comp.mxeuni}
To simplify the writing, following~\cite{KunRod-pp17}, we will use the following notation concerning the distribution/location of the actuators:
 \begin{align*}
  \DD=\DD_{\rm act}&={\rm mxe}\quad\mbox{stand for the location as in~\eqref{Dmxe}},\\
  \DD=\DD_{\rm act}&={\rm uni}\quad\mbox{stand for the location as in~\eqref{Duni}},\\
  \DD=\DD_{\rm act}&={\rm con}\quad\mbox{stand for the location as in~\eqref{Dcon}},
 \end{align*}
which underlines that for~$\DD={\rm mxe}$, the actuators are located at the extremisers of the~$M$-th eigenfunction~$\sin(Mx)$, while~$\DD={\rm uni}$
stands for the uniform distribution of the actuators, and~$\DD={\rm con}$ underlines that the actuators are concentrated at the center of the interval
domain.

Recall that, for Dirichlet boundary conditions, the locations~$\DD={\rm mxe}$ and~$\DD={\rm uni}$ lead to a diagonal matrix~$[\Theta(c_M)]$.
Therefore, we know that its eigenvalues are its diagonal entries. This is what allowed us to
 derive an analytical expression for the smallest eigenvalue, as in  Theorems~\ref{MT:Dirmxe}, \ref{MT:Diruni}, and~\ref{MT:Neumxe}, respectively. Instead for the location~$\DD={\rm con}$ we
 we do not know an analytical expression for the smallest eigenvalue, and we do not know whether such eigenvalue
 remains bounded away from zero as~$M$ increases. 
 See the simulations in~\cite{KunRod-pp17}, which show that the first eigenvalue associated with~$\DD={\rm con}$
 is considerably smaller than the ones
 associated with~$\DD={\rm mxe}$ and~$\DD={\rm uni}$.

Observe that from Theorems~\ref{MT:Dirmxe}, \ref{MT:Diruni}, we know that the smallest eigenvalue
in both cases~$\DD={\rm mxe}$ and~$\DD={\rm uni}$ converges to the same
limit, that is, to~$\frac{4}{r\pi^2}\sin^2(\frac{r\pi}{2})$.

Figure~\ref{Fig:1Dsmallest_eigenvalue} shows a comparison between the smallest
eigenvalue~$\vartheta_M\coloneqq \vartheta(c^M)$ given in Theorems~\ref{MT:Dirmxe}, \ref{MT:Diruni}, and~\ref{MT:Neumxe}, respectively.
As expected we obtain the
behaviour as in~\cite{KunRod-pp17}, for the Dirichlet case, where those eigenvalues have been computed numerically. In
Figure~\ref{Fig:1Dsmallest_eigenvalue} and the following ones the annotation~${\rm bc}={\rm Dir}$, respectively~${\rm bc}={\rm Neu}$
means that the homogenous Dirichlet, respectively Neumann, boundary conditions have been considered in the computations.

\begin{figure}[ht]
\centering
\subfigure
{\epsfig{file=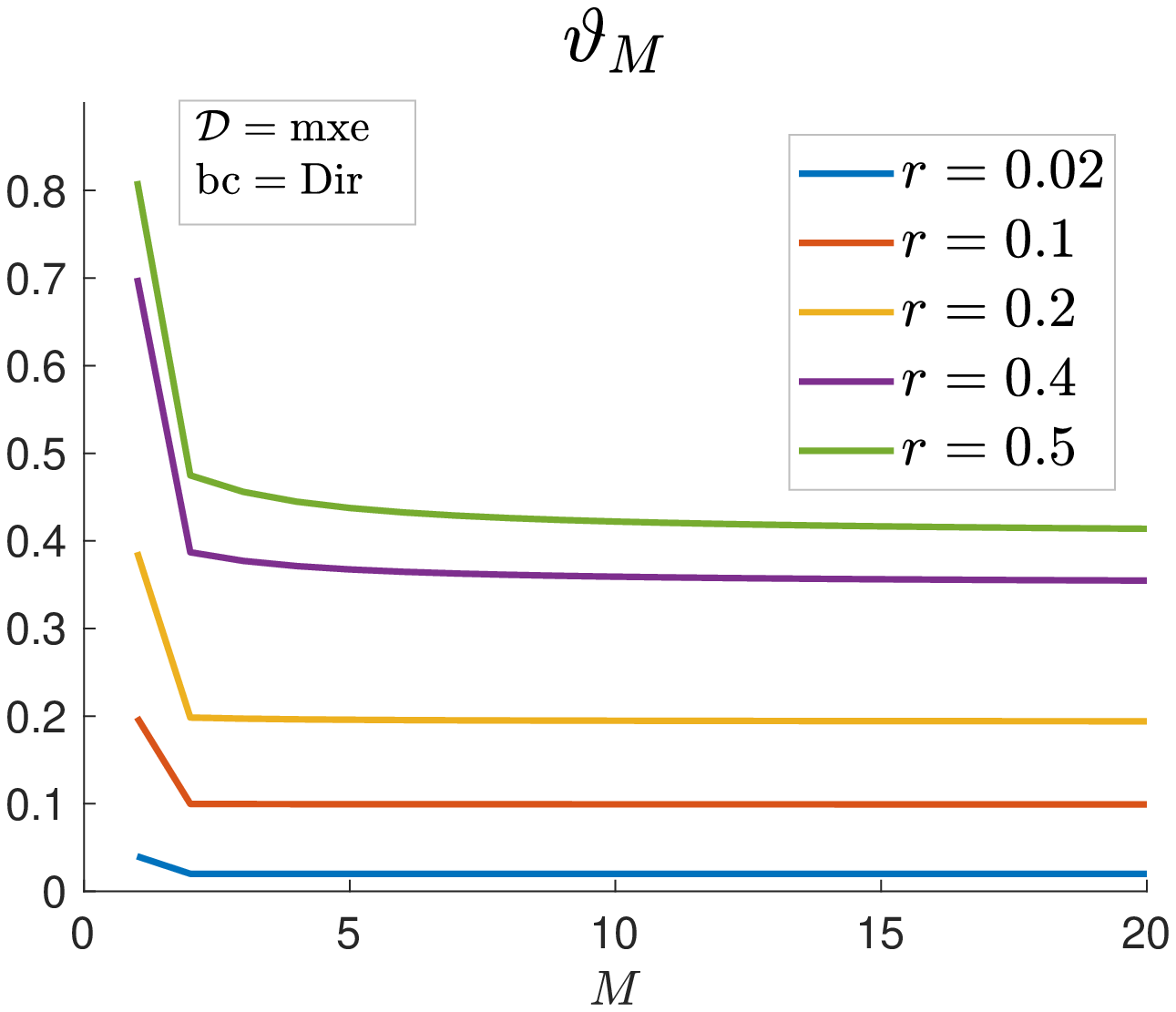,width=.325\linewidth,clip=}}
\subfigure
{\epsfig{file=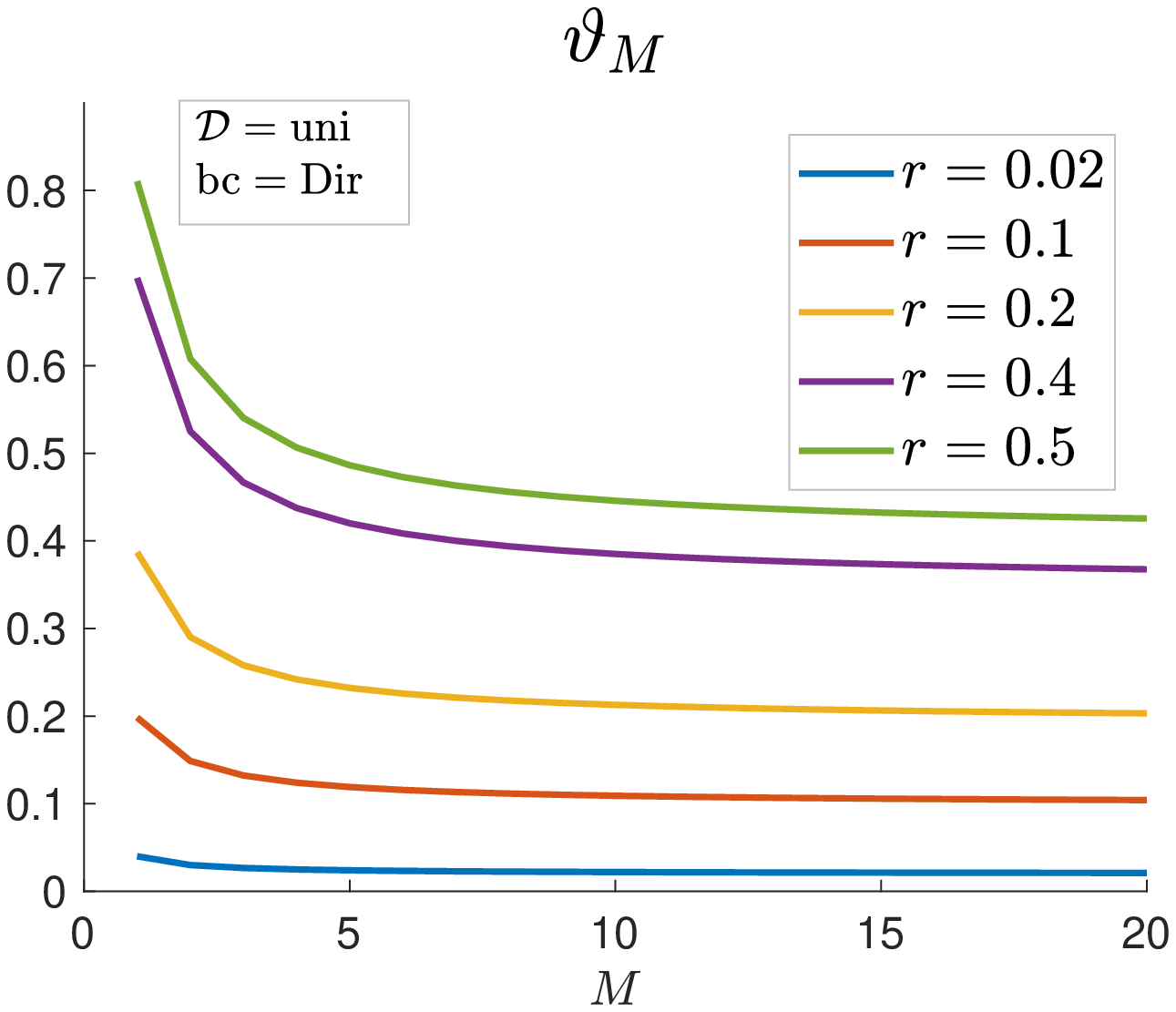,width=.325\linewidth,clip=}}
 \subfigure
 {\epsfig{file=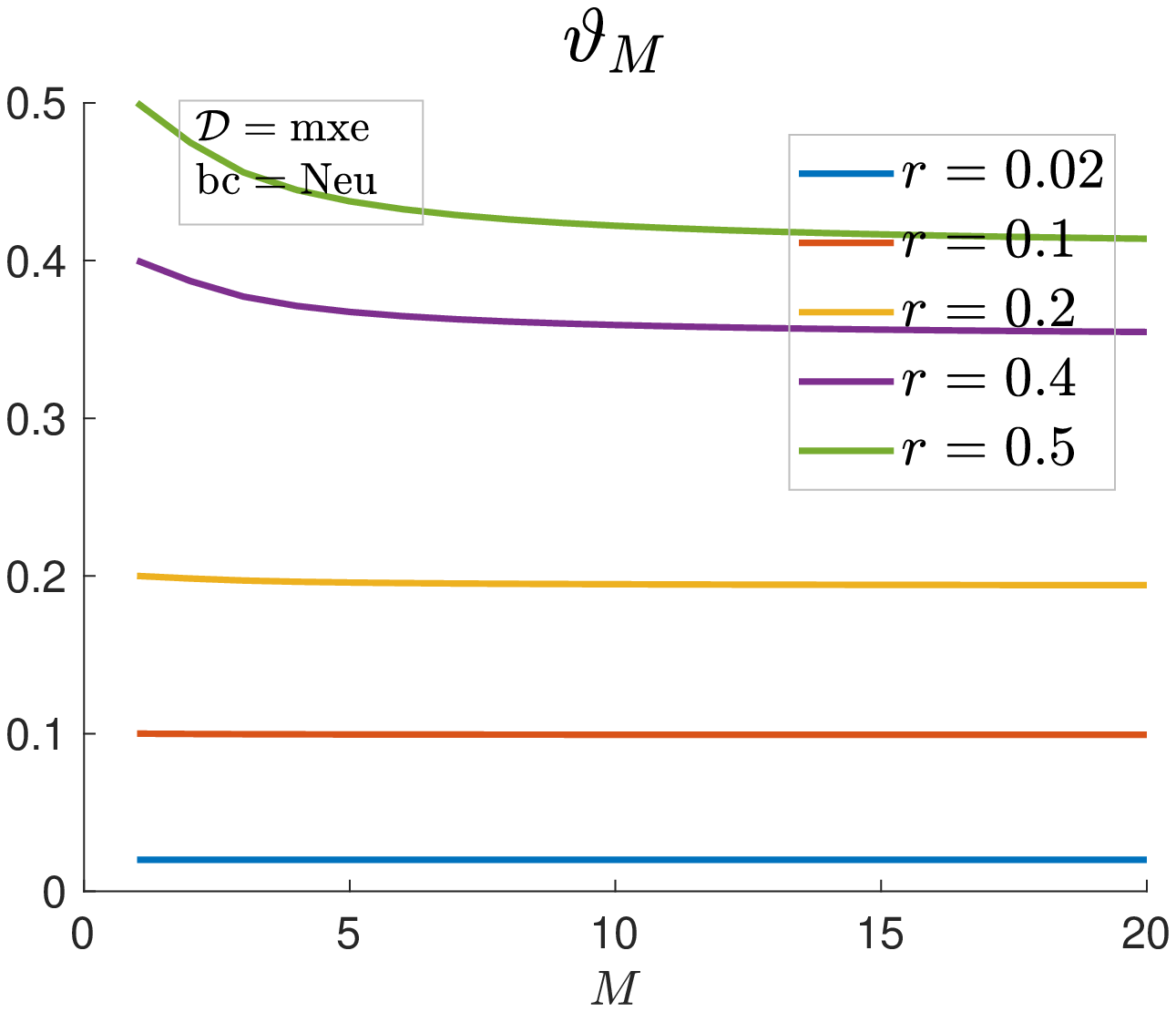,width=.325\linewidth,clip=}}
\caption{Smallest eigenvalue. Plot of analytical expressions}
\label{Fig:1Dsmallest_eigenvalue}
\end{figure}

Figure~\ref{Fig:1Dsmallest_eigenvalue.Neu} shows the smallest
eigenvalue~$\vartheta_M$ computed numerically for the Neumann case for~$\DD={\rm mxe}$ and~$\DD={\rm con}$. Recall that for the Neumann case the analytical expression
for~$\vartheta_M$ is known only for the case~$\DD={\rm mxe}$.
\begin{figure}[ht]
\centering
\subfigure
{\epsfig{file=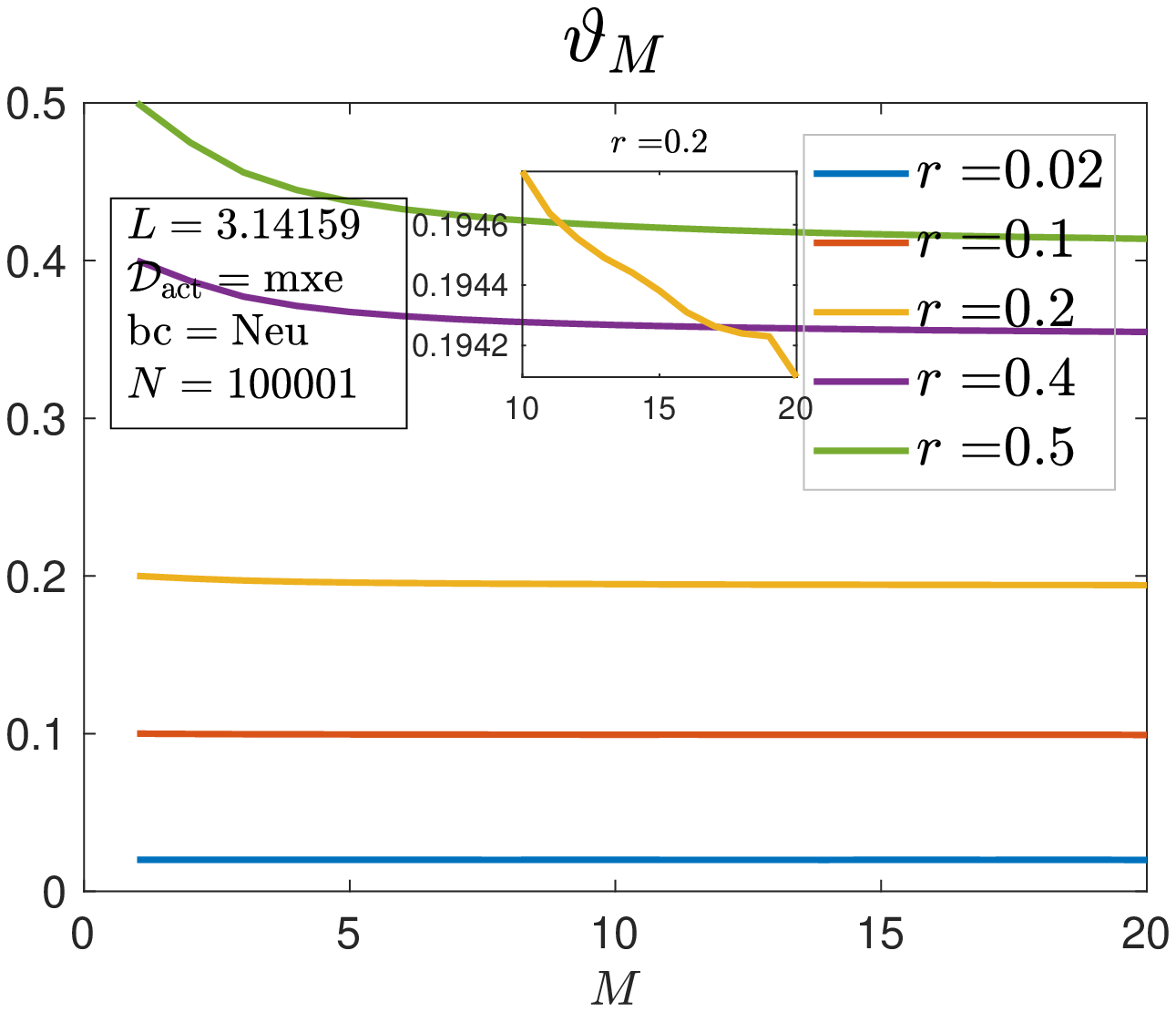,width=.325\linewidth,clip=}}
\subfigure
{\epsfig{file=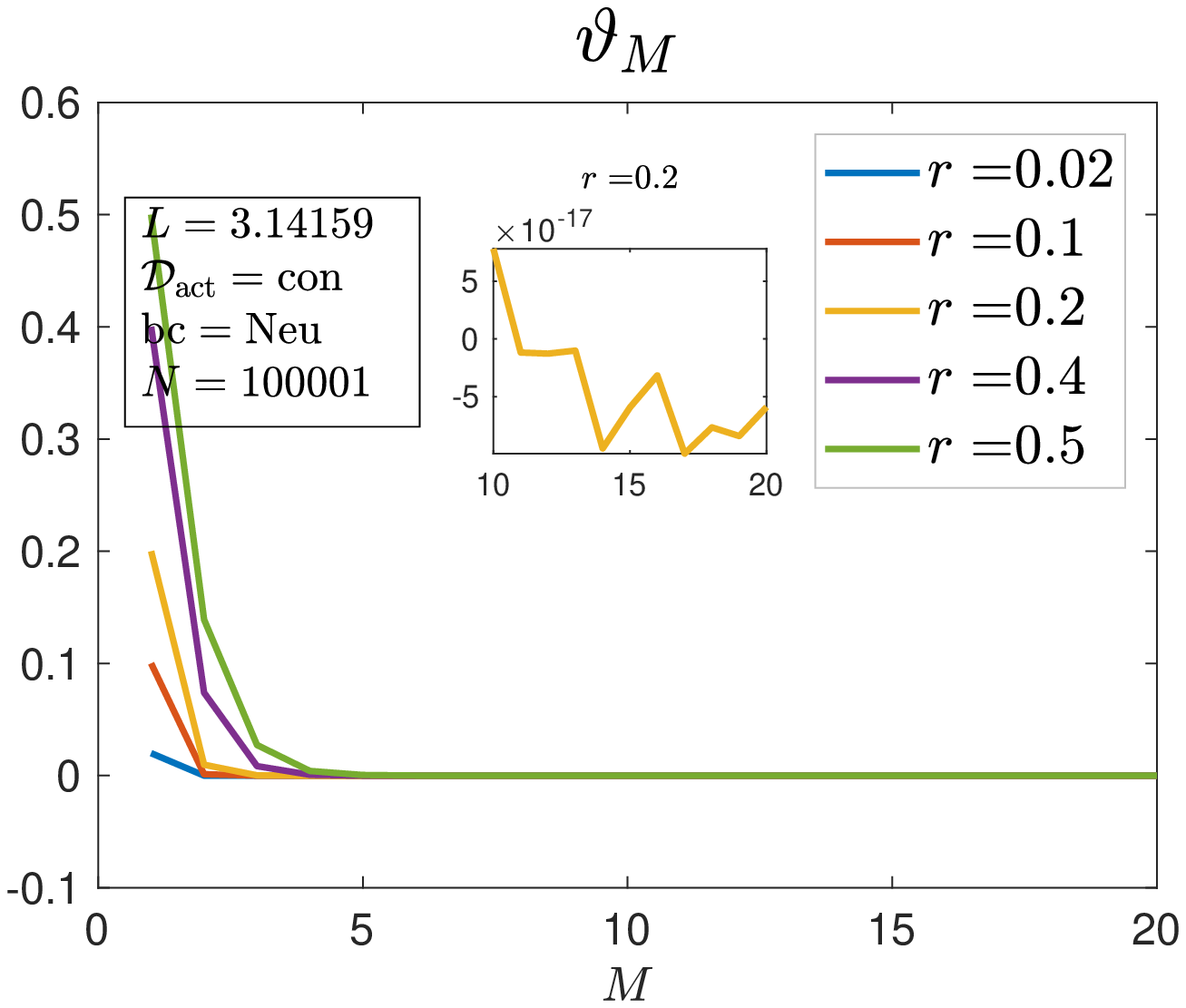,width=.325\linewidth,clip=}}
\caption{Numerical results. Neumann case.}
\label{Fig:1Dsmallest_eigenvalue.Neu}
\end{figure}

\begin{figure}[ht]
\centering
\subfigure
{\epsfig{file=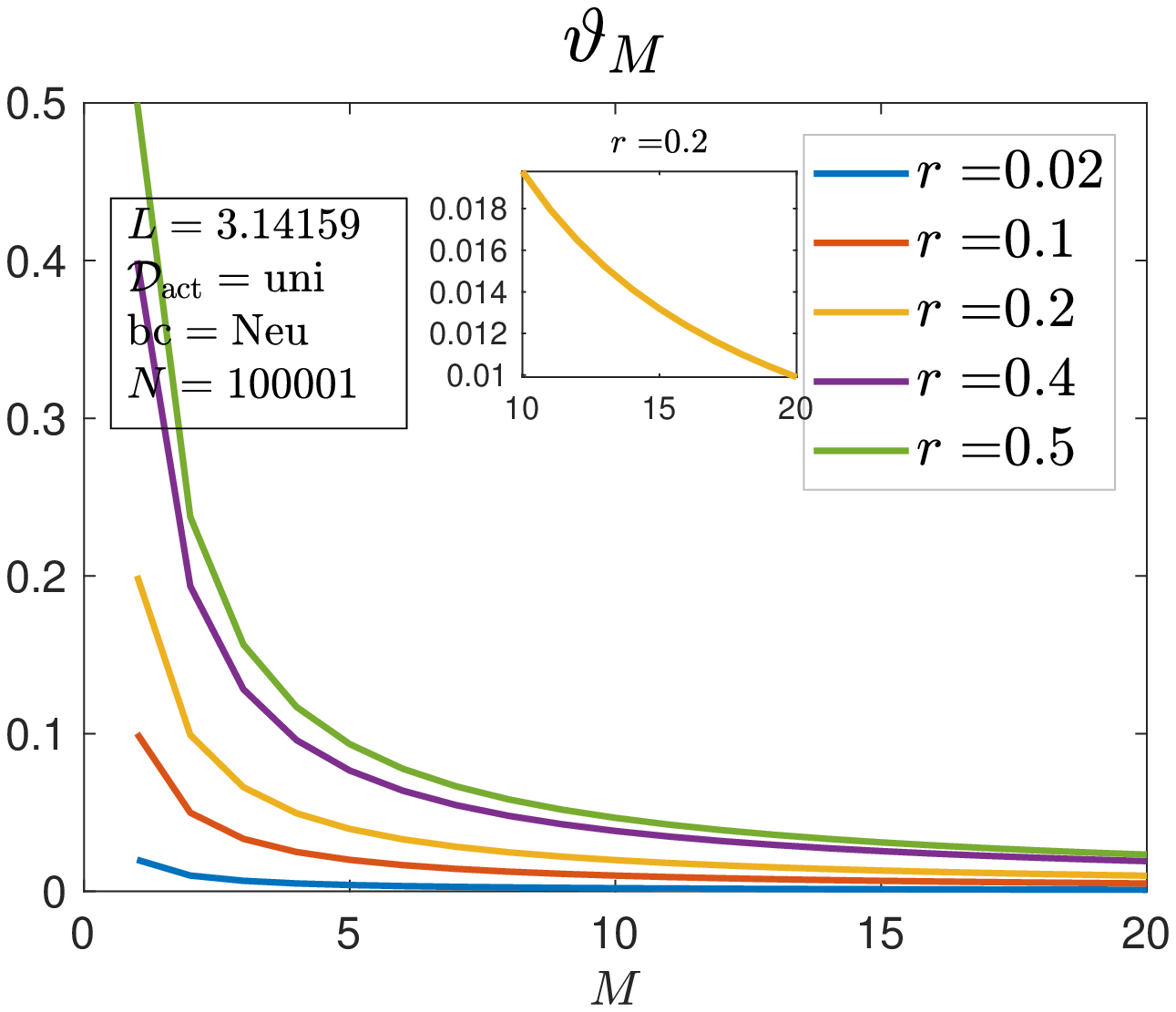,width=.325\linewidth,clip=}}
\subfigure
{\epsfig{file=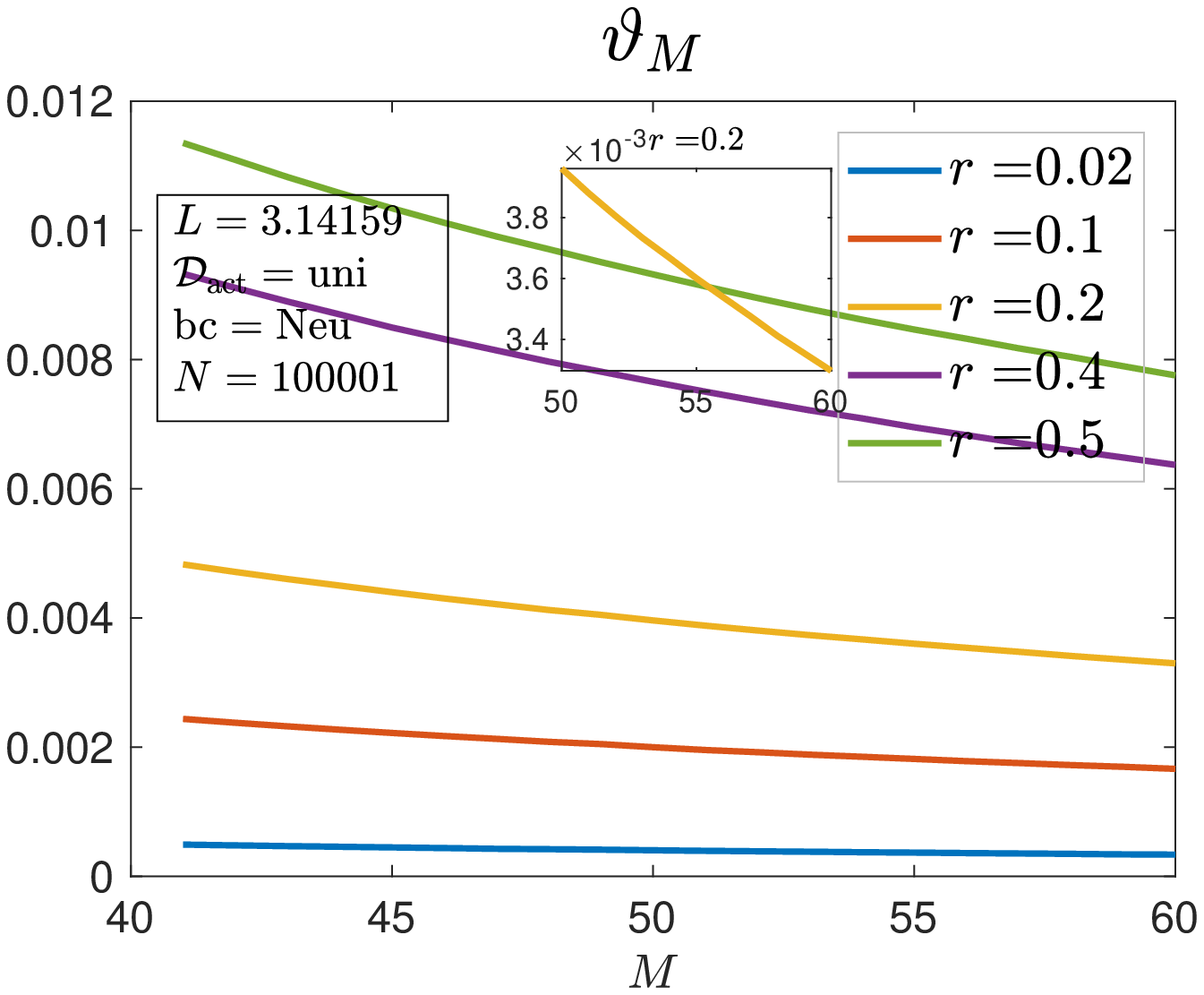,width=.325\linewidth,clip=}}
\subfigure
{\epsfig{file=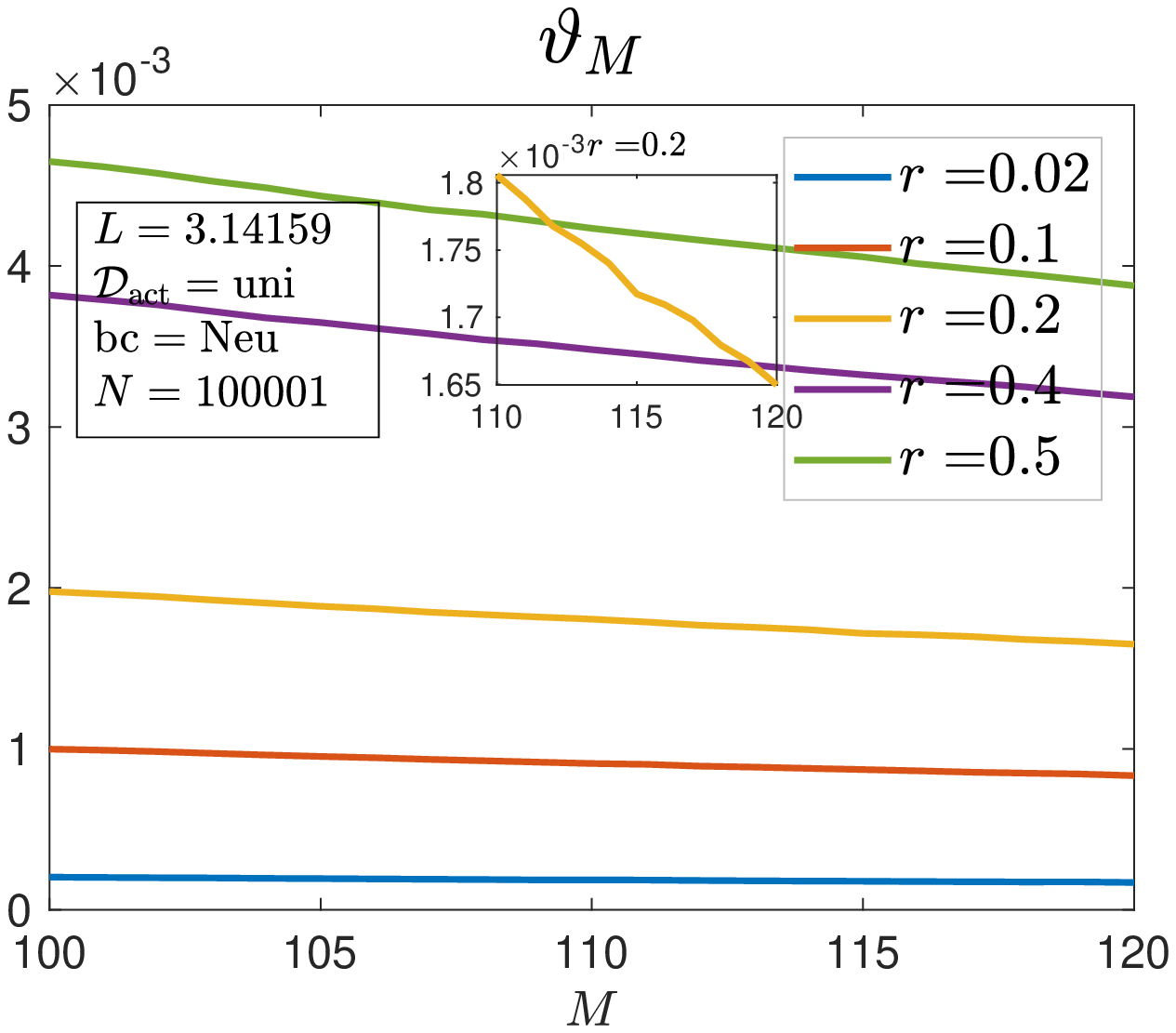,width=.325\linewidth,clip=}}
\caption{Numerical results. Neumann case.}
\label{Fig:1Dsmallest_eigenvalue.Neu-uni}
\end{figure}

In Figure~\ref{Fig:1Dsmallest_eigenvalue.Neu-uni}, we plot the  smallest
eigenvalue~$\vartheta_M$ computed numerically for the Neumann case for~$\DD={\rm uni}$.

From Figure~\ref{Fig:1Dsmallest_eigenvalue.Neu-uni} it is not totally clear whether the eigenvalue will remain away from zero as~$M$ increases.
Indeed, for $r=0.2$ roughly speaking we can see that:
\begin{itemize}
 \item for~$M\in[10,20]$ we have~$\frac{\ed\vartheta_M}{\ed M}\approx\frac{\vartheta_{20}-\vartheta_{10}}{20-10}\approx-10^{-3}$,
 \item for~$M\in[50,60]$ we have~$\frac{\ed\vartheta_M}{\ed M}\approx\frac{\vartheta_{60}-\vartheta_{50}}{60-50}\approx-6\cdot10^{-5}$,
 \item for~$M\in[110,120]$ we have~$\frac{\ed\vartheta_M}{\ed M}\approx\frac{\vartheta_{120}-\vartheta_{110}}{120-110}\approx-1.5\cdot10^{-5}$,
\end{itemize}
from which we see that~$\frac{\ed\vartheta_M}{\ed M}$ is increasing, but it also seems that it increases too slowly.  

In any case it is clear that,
in the Neumann case,
the eigenvalue~$\vartheta_M$ presents a remarkably different behaviour,
for the locations~$\DD={\rm mxe}$ and~$\DD={\rm uni}$. Even if for~$\DD={\rm uni}$ the eigenvalue remains bounded
away from zero as~$M$ increases, it is clear that its minimum is considerably smaller.

Recall that
in the Dirichlet case the behaviours for those locations are quite close
and with the same limit as~$M$ increases.

This also shows that the best location of the actuators, maximising~$\vartheta_M$ for a given~$M$, is not a trivial problem,
likely depending strongly
on the boundary conditions. 

Note that, recalling~\eqref{eq:proj_eigenval}, maximising~$\vartheta_M=\vartheta(c^M)$ is equivalent to minimise the norm of
the oblique projection~$P_{U_M(c^M)}^{E_M^\perp}$, which can lead to a better performance of the stabilising feedback control~$\Ck(t)$ in~\eqref{FeedKy}, and
to guarantee that the sufficient stabilisability condition in~\cite[Section~3.1]{KunRod-pp17} is satisfied for a smaller~$M$.

That is, it would be important to know the best location of the actuators for a given~$M$. This is an optimisation problem which will
require different tools, and so will be addressed in a separate work~\cite{RodSturm-ow18}. 
\section{Additional remarks on the oblique projection based feedback}\label{S:remOPorjFeed}
We first illustrate that oblique projections are substantially different from their orthogonal counterpart.
In Figure~\ref{Fig:proj_example} we see the orthogonal and oblique (along $E_M^\perp$) projections of the function $f(x) = 1_{(0,\frac{1}{2})}(x-1)(x-2)(x-3)$ 
onto the span~$U_M$ of~$6$ actuators distributed as in~\eqref{Dmxe},
 and the total actuator volume is $r\pi$  with $r=0.1$.

\begin{figure}[ht]
\subfigure
{\epsfig{file=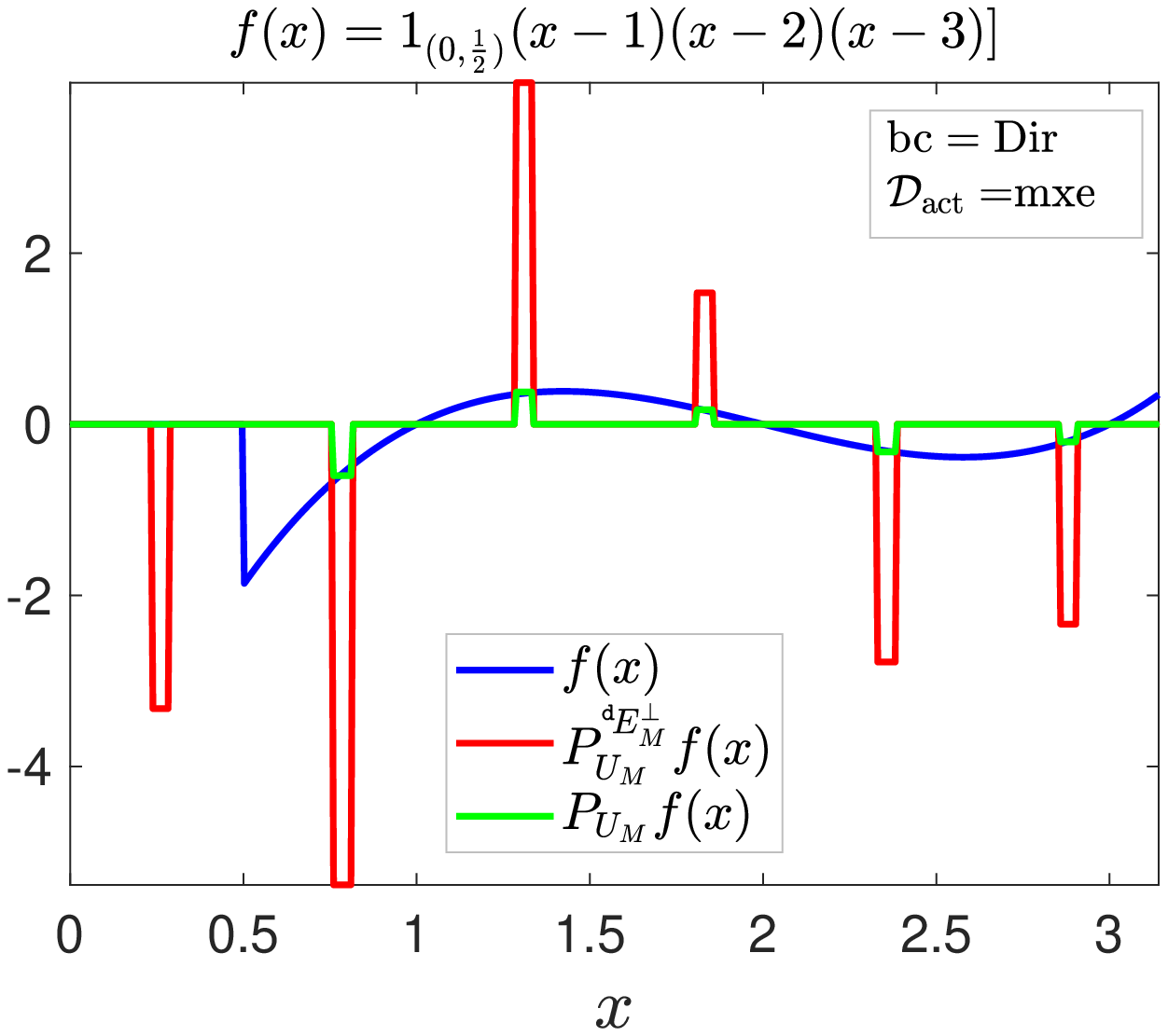,width=.325\linewidth,clip=}}%
\qquad
\subfigure
{\epsfig{file=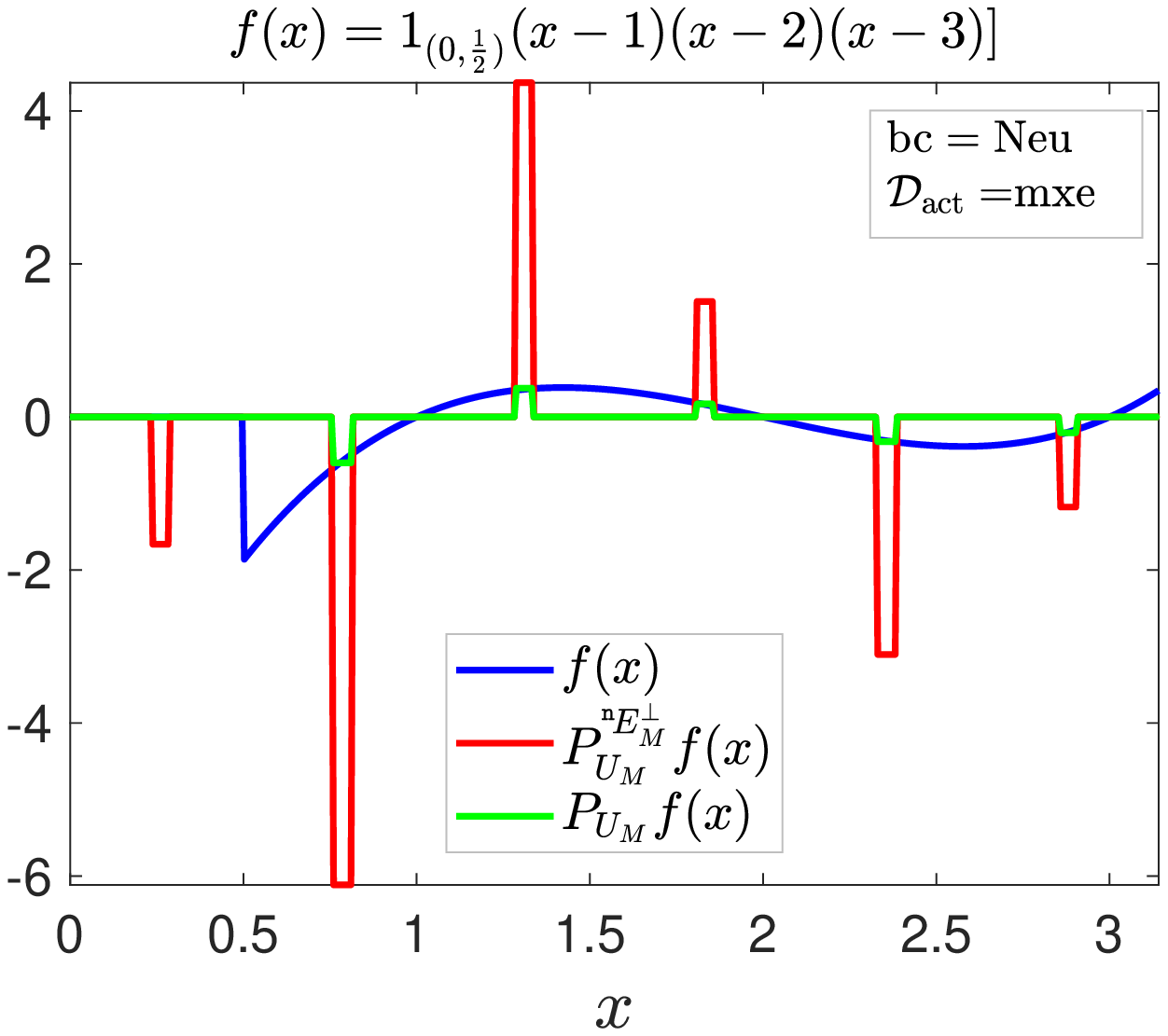,width=.325\linewidth,clip=}}%
\caption{Oblique and orthogonal projections onto~$U_M$.}
\label{Fig:proj_example}
\end{figure}

We recall that in the procedure in~\cite{PhanRod18-mcss,BreKunRod17,KroRod-ecc15}, in order to prove the existence of an open-loop stabilising control,
the orthogonal projection onto the span~$U_M$ of actuators
has been used in order to approximate a suitable
infinite-dimensional control~$\eta$, driving the system to zero in a finite interval,
together with a concatenating argument. That is, by taking~$M$ larger we have that~$P_{U_M}\eta$ is closer to~$\eta$, and a continuity argument is used.
Then, finally a Riccati based feedback is
found by solving a suitable differential Riccati equation.

Instead, in here the feedback is explicit and the idea behind the oblique projection based feedback~\eqref{FeedKy}
is not based on an approximation reason.
That is, the purpose of taking~$P_ {U_M}^{E_M^\perp}\left(-\Delta z +a(t)z-\lambda z\right)$ is not to approximate
$-\Delta z +a(t)z-\lambda z$. Actually, we can see that since the total volume of the actuators is fixed,
then taking~$M$ larger does not necessarily imply neither that~$P_ {U_M}^{E_M^\perp}f$ is closer to~$f$ nor that~$P_{U_M}f$ is closer to~$f$.
As an illustration, let us consider the case where~$f(x)=f_0$ is a constant function, $x\in(0,\pi)$. Then, we can see that
$P_{U_M}f=\begin{cases}
           f_0&\mbox{if }x\in\bigcup_{i=1}^M\omega_i,\\
           0&\mbox{if }x\notin\bigcup_{i=1}^M\omega_i,
          \end{cases}
$
Therefore, if~$r\pi=\int_{\bigcup_{i=1}^M\omega_i}1\,\ed x$ is the total volume of
the actuators~$\{1_{\omega_1},1_{\omega_2},\dots,1_{\omega_M}\}$, then
we obtain~$\norm{f-P_{U_M}f}{L^2}^2=\int_{(0,\pi)\setminus\bigcup_{i=1}^M\omega_i}\norm{f_0}{}^2\,\ed x=\norm{f_0}{}^2(1-r)\pi$.
This means that~$P_{U_M}f$ is not approximating~$f$, since the
distance to~$f$, that is~$\norm{f-P_{U_M}f}{L^2}=\norm{f_0}{}\sqrt{(1-r)\pi}$, remains the same as~$M$ increases.
Finally, recalling Remark~\ref{R:orthprojoptim}, we also have that~$\norm{f-P_{U_M}^{E_M^\perp}f}{L^2}\ge\norm{f-P_{U_M}f}{L^2}=\norm{f_0}{}\sqrt{(1-r)\pi}$,
that is, $P_{U_M}^{E_M^\perp}f$ is also not approximating~$f$, as~$M$ increases.

\section{Numerical simulations for the closed-loop system}\label{S:numeric}
Here we present some simulations for the oblique projection based feedback systems~\eqref{sys-y-parDir} and~\eqref{sys-y-parNeu}, under Dirichlet
and Neumann boundary conditions, respectively.

\subsection{Discretization}
We follow a finite element discretization of our system based on the standard piecewise linear hat functions. 
Let~$N\ge2$ be a positive integer and let~$\Omega_D=(0h,1h,2h,\dots,(N-1)h)$, $h=\pi/(N-1)$, be a discretization of the space interval~$[0,\pi]$.
We briefly recall that the so called  hat  functions are explicitly given by
\begin{align*}
 h_{i+1}(x)&=\begin{cases}
       1-i+\frac{x}{h}&\mbox{if }x\in[(i-1)h,ih],\\
       1+i-\frac{x}{h}&\mbox{if }x\in[h,(i+1)h],\\
       0&\mbox{if }x\notin[(i-i)h,(i+1)h],
      \end{cases}\qquad\mbox{for the interior points }i\in\{1,\dots,N-2\}.
      \intertext{Instead for the boundary points,}
 h_1(x)&=\begin{cases}
       1+i-\frac{x}{h}&\mbox{if }x\in[0,h],\\
       0&\mbox{if }x\notin[0,h],
      \end{cases} \qquad \mbox{and}\qquad
      h_{N}(x)=\begin{cases}
      1-i+\frac{x}{h}&\mbox{if }x\in[(N-2)h,\pi],\\
      0&\mbox{if }x\notin[(N-2)h,\pi].
      \end{cases}  
\end{align*}

The mass and stiffness matrices are defined as matrices in~$\R^{N\times N}$ by
\[
\Ma\coloneqq[(h_n,h_m)_{L^2}]\qquad\mbox{and}\qquad
\St\coloneqq[(\nabla h_n,\nabla h_m)_{L^2}].
\]
For given functions~$z$ and~$w$ in~$H^1\coloneqq H^1((0,\pi))$ we have the approximations, for both Dirichlet and Neumann homogeneous boundary conditions,
\begin{align}
 \langle z,w\rangle_{V',V}&=(z,w)_H\approx \mathbf w^\top \Ma \mathbf z,\label{discI}\\
 \langle-\fractx{\ed}{\ed x}\fractx{\ed}{\ed x} z, w\rangle_{V',V}
 &=( \fractx{\ed}{\ed x} z,\fractx{\ed}{\ed x} w)_H-\fractx{\ed}{\ed x} z(\pi)w(\pi)+\fractx{\ed}{\ed x} z(0)w(0)\approx\mathbf w^\top\St\mathbf z\label{discLap}
\end{align}
where~$\mathbf z= [\mathbf z_i]\coloneqq[z((i-1)h)]$
and~$\mathbf w= [\mathbf w_i]\coloneqq [w((i-1)h)]$ are the vectors, in~$\R^{N\times 1}$,
containing the (ordered) values of~$z$ and~$w$ at the mesh points in~$\Omega_D$. We have  $z\approx\sum\limits_{m=1}^{N}\mathbf z_m h_{m}(x)$
and~$w\approx\sum\limits_{n=1}^{N}\mathbf z_n h_{n}(x)$. 

\subsubsection*{The reaction term}
The approximation of the  reaction term (for fixed time~$t$) is taken as follows, 
\begin{equation}\label{discR}
 \langle a(t)z,w\rangle_{V',V}=( a(t)z,w)_{L^2}=( z,a(t)w)_{L^2}
 \approx\mathbf w^\top\RRR(t)\mathbf z,\quad\mbox{with}\quad\RRR(t)\coloneqq\fractx{\Ma[{\rm Diag}(\mathbf a(t))]+[{\rm Diag}(\mathbf a(t))]\Ma}{2}.
\end{equation}
Here~$[{\rm Diag}(\mathbf a(t))]$ stands for the diagonal matrix whose diagonal elements
are given by the entries of~$\mathbf a(t)$, that is,~$[{\rm Diag}(\mathbf a(t))]_{(i,i)}\coloneqq\mathbf a_i(t)=[a(t)((i-1)h)]$.

Note that with~$\mathbf p=[{\rm Diag}(\mathbf a(t))]\mathbf z$, we have
that~$a(t)z\approx\sum\limits_{i=1}^{N} \mathbf a_i(t)\mathbf z_i h_{i}(x)=\sum\limits_{i=1}^{N}\mathbf p_i h_{i}(x)$.

\subsubsection*{The feedback operator}
For the discretization of the oblique projection we will follow~\eqref{eq:proj_form}.
We recall that $P_{ U_M}^{E_M^\perp}=P_{ U_M}^{E_M^\perp}P_{E_M}$.
Analogously,  we also denote the eigenvalues and eigenfunctions (either in~\eqref{eigs-Dir} or in~\eqref{eigs-Neu}) by
$\alpha_i$ and~$ e_i$, and we denote the ordered sets~$\Ce_M=( e_1, e_1,\dots, e_M)$, and take
$\Cu_M=(1_{\omega_1},1_{\omega_2}, \dots,1_{\omega_M})$. That is, $ E_M=\linspan\{ e_1, e_1,\dots, e_M\}$ and
$U_M=\linspan\{1_{\omega_1},1_{\omega_2}, \dots,1_{\omega_M}\}$.
Notice that~\eqref{eq:proj_form} does not require neither the basis~$\Cu_M$ nor the basis~$\Ce_M$
to be orthonormal.

Following~\eqref{eq:proj_form}
we have that
\[
 P_{ U_M}^{ E_M^\perp}z(x)=\sum\limits_{i=1}^{M}\, p_i 1_{\omega_i}(x),\quad\mbox{with}\quad 
 p\coloneqq([(\Ce_M,\Cu_M)_{L^2}])^{-1}[(\Ce_M,z)_{L^2}].
\]

Now, let us denote the matrices in~$\R^{(N+1)\times M}$
\[
 [\Ce_M]=\begin{bmatrix}
          {\mathbf e}_1&{\mathbf e}_2&\dots&{\mathbf e}_M
         \end{bmatrix}\qquad\mbox{and}\qquad
         [\Cu_M]=\begin{bmatrix}
          \mathbf {1}_{\omega_1}&\mathbf {1}_{\omega_2}&\dots&\mathbf {1}_{\omega_M}
         \end{bmatrix}
\]
whose columns are the vectors of the evaluations at the mesh points in~$\Omega_D$. With this notations we can see that
\[
p\approx  [{\mathbf e}_i^\top\Ma\mathbf 1_{\omega_j}]^{-1}[{\mathbf e}_j^\top\Ma z]
=[{\mathbf e}_i^\top\Ma \mathbf  1_{\omega_j}]^{-1}[\Ce_M]^\top\Ma\mathbf z. 
\]
Note that~$[{\mathbf e}_i^\top\Ma\mathbf 1_{\omega_j}]\in\R^{M\times M}$ is a relatively small matrix (if the number of actuators is small).
In the simulations the inverse~$[{\mathbf e}_i^\top\Ma\mathbf 1_{\omega_j}]\in\R^{M\times M}$ was computed a priori, and the matrix
\[
 \PPP_M=[{\mathbf e}_i^\top\Ma\mathbf 1_{\omega_j}]^{-1}[\Ce_M]^\top
\]
was saved. In this way, we  compute 
\[
p\approx q\coloneqq \PPP_M\Ma\mathbf z,\qquad
P_{ U_M}^{ E_M^\perp}z\approx\sum\limits_{i=1}^{M}\,q_i 1_{\omega_i}.
 \]
That is,
\begin{equation}\label{approxP}
\mbox{with}\quad v=P_{ U_M}^{ E_M^\perp}z,\qquad\mathbf v= [\mathbf v_i]\coloneqq[v((i-1)h)]\approx[\Cu_M]\PPP_M\Ma\mathbf z.
\end{equation}
Therefore, by using~\eqref{discI},
\begin{equation}\label{discP.I}
 \langle P_{ U_M}^{ E_M^\perp}z,w\rangle_{V',V}\approx \mathbf w^\top\Ma[\Cu_M]\PPP_M\Ma\mathbf z
\end{equation}

To complete the discretization of the feedback~$\Ck(t)$ in~\eqref{FeedKy} it remains to
discretize~$\langle P_{ U_M}^{ E_M^\perp}a(t)z,w\rangle_{V',V}$ and~$\langle P_{ U_M}^{ E_M^\perp}\Delta z,w\rangle_{V',V}$.

We will use Lemma~\ref{L:adjProj}.
Observe that
\begin{align*}
 \langle P_{ U_M}^{ E_M^\perp}a(t)z,w\rangle_{V',V}&=( P_{ U_M}^{ E_M^\perp}a(t)z,w)_H=\langle a(t)z,P_{  E_M}^{U_M^\perp}w\rangle_{V,V'}\\
 \langle P_{ U_M}^{ E_M^\perp}(-\Delta) z,w\rangle_{V',V}&=-( P_{ U_M}^{ E_M^\perp}P_{  E_M}\Delta z,w)_{H}
 =-\langle \Delta z,P_{  E_M}^{U_M^\perp}w\rangle_{V',V},
\end{align*}
and using~\eqref{discR}, with $h=P_{ E_M }^{U_M^\perp}w$, and the analogous to~\eqref{approxP}
\[
\mathbf h= [\mathbf h_i]\coloneqq[h((i-1)h)]\approx[\Ce_M]\widehat\PPP_M\Ma\mathbf w,\qquad
\widehat\PPP_M\coloneqq[\mathbf 1_{\omega_i}^\top\Ma{\mathbf e}_j]^{-1}[\Cu_M]^\top,
\]
we arrive to
\begin{align*}
\langle P_{ U_M}^{ E_M^\perp}a(t)z,w\rangle_{V',V}&=\langle a(t)z,h\rangle_{V,V'}
\approx \mathbf h^\top\RRR(t) \mathbf z=\mathbf w^\top\Ma(\widehat\PPP_M)^\top[\Ce_M]^\top\RRR(t) \mathbf z\\
&=\mathbf w^\top\Ma[\Cu_M][{\mathbf e}_i^\top\Ma\mathbf 1_{\omega_j}]^{-1}[\Ce_M]^\top\RRR(t) \mathbf z.
\end{align*}
Proceeding similarly for~$\langle P_{ U_M}^{ E_M^\perp}\Delta z,w\rangle_{V',V}$ we arrive to
\begin{align}
\langle P_{ U_M}^{ E_M^\perp}a(t)z,w\rangle_{V',V}&
=\mathbf w^\top\Ma[\Cu_M]\PPP_M\RRR(t) \mathbf z,\label{discP.R}\\
\langle P_{ U_M}^{ E_M^\perp}(-\Delta) z,w\rangle_{V',V}&
=\mathbf w^\top\Ma[\Cu_M]\PPP_M\St \mathbf z.\label{discP.Lap}
\end{align}

Therefore the dicretization of our feedback operator, as in~\eqref{FeedKy}, is taken as
\begin{equation}\label{discK}
\langle P_ {U_M}^{E_M^\perp}\left(-\Delta z +a(t)z-\lambda z\right) ,w\rangle_{V,V'}\approx
\mathbf w^\top\KKK(t)\mathbf z,\qquad \KKK(t)\coloneqq \Ma[\Cu_M]\PPP_M\left(\St+\RRR(t)-\lambda\Ma\right). 
\end{equation}

\subsubsection*{The closed-loop system under Dirichlet boundary conditions}
Though we restrict ourselves to homogeneous (zero) boundary conditions
we include  here the case of nonhomogeneous conditions.

Observe that the space discretization of the closed loop system reads
\begin{equation}\label{discSys}
 \Ma\fractx{\ed}{\ed t}\mathbf y=-\nu\St\mathbf y-\RRR(t)\mathbf y-\Ma[\Cu_M]\PPP_M\left(-\nu\St-\RRR(t)+\lambda\Ma\right)\mathbf y.
\end{equation}
From now we will consider the reaction term and the feedback as  external forces
\[
\mathbf h=\mathbf h(\mathbf y)\coloneqq\mathbf r+\Ma\mathbf f,\qquad \mathbf r\coloneqq-\RRR(t)\mathbf y,\qquad \mathbf f\coloneqq-[\Cu_M]\PPP_M\left(-\nu\St-\RRR(t)+\lambda\Ma\right)\mathbf y.
\]
Now we recall how we can solve~\eqref{discSys} with a general external forcing~$\mathbf h$.

Recall that the dynamics in system~\eqref{sys-y-parDir} has to be seen in~$L^2((0,T),V')$, that is,
for Dirichlet boundary conditions, we need to take test functions~$w$
taking values in~$V=H^1_0((0,\pi))$. That is,
we have to test~\eqref{discSys} with vectors~$\mathbf w$ with~$w(1)=0=w(N)$.

Up to a permutation, the mesh points in~$\Omega_D$ can be reorganized so that the boundary points appear at last as,
$\Pi\Omega_D=\Omega_D^\circ=(1h,2h,\dots,(N-3)h,(N-2)h,0h,(N-1)h)$. We also write the solution~$\mathbf y$ in these new coordinates
$\mathbf y^\circ\coloneqq\Pi\mathbf y
=\begin{bmatrix}\mathbf y^\circ_{\rm i}&\mathbf y^\circ_{\rm b}\end{bmatrix}^\top$
where~$\mathbf y^\circ_{\rm i}\coloneqq\begin{bmatrix}\mathbf y_2&
\mathbf y_3&\dots&\mathbf y_{N-2}&\mathbf y_{N-1}\end{bmatrix}^\top$
and~$\mathbf y^\circ_{\rm b}\coloneqq\begin{bmatrix}\mathbf y_1& \mathbf y_{N}\end{bmatrix}^\top$,
stand for the interior and boundary coordinates of~$\mathbf y^\circ$, respectively. Similarly we rewrite the external
forcing~$\mathbf h^\circ\coloneqq\Pi\mathbf h
=\begin{bmatrix}\mathbf h^\circ_{\rm i}&\mathbf h^\circ_{\rm b}\end{bmatrix}^\top$. In this way we arrive at
\begin{equation}\label{discSysNeuPif}
 \Ma^\circ \fractx{\ed}{\ed t}\mathbf y^\circ=-\nu\St^\circ\mathbf y^\circ+\mathbf h^\circ
\end{equation}
where, recalling that the inverse of a permutation coincides with the transpose,
\[
\Ma^\circ\coloneqq\Pi\Ma\Pi^\top,\qquad  \St^\circ\coloneqq\Pi\St\Pi^\top.
\]

Now rewriting the last matrices in blocks
\[
 \Ma^\circ\eqqcolon\begin{bmatrix}\Ma^\circ_{\rm ii}&\Ma^\circ_{\rm ib}\\
 \Ma^\circ_{\rm bi}&\Ma^\circ_{\rm bb}\end{bmatrix},\qquad
 \St^\circ\eqqcolon\begin{bmatrix}\St^\circ_{\rm ii}&\St^\circ_{\rm ib}\\
 \St^\circ_{\rm bi}&\St^\circ_{\rm bb}\end{bmatrix},
\]
so that the last columns/rows correspond to the boundary coordinates, we find after testing~\eqref{discSysNeuPif} with~$\mathbf w^\circ$,
\[
(\mathbf w^\circ_{\rm i})^\top\begin{bmatrix}\Ma^\circ_{\rm ii}&\Ma^\circ_{\rm ib}\end{bmatrix} \fractx{\ed}{\ed t}\mathbf y^\circ
=-\nu(\mathbf w^\circ_{\rm i})^\top\begin{bmatrix}\St^\circ_{\rm ii}&\St^\circ_{\rm ib}\end{bmatrix}\mathbf y^\circ
+(\mathbf w^\circ_{\rm i})^\top\mathbf  h^\circ_{\rm i} 
\]
because~$\mathbf w^\circ=\begin{bmatrix} \mathbf w^\circ_{\rm i}&\mathbf w^\circ_{\rm b}\end{bmatrix}^\top$ vanishes at boundary coordinates
(i.e., $\mathbf w^\circ_{\rm b}$ vanishes). Now  we obtain
\begin{align*}
\Ma^\circ_{\rm ii} \fractx{\ed}{\ed t}\mathbf y^\circ_{\rm i}
&=-\nu\St^\circ_{\rm ii}\mathbf y^\circ_{\rm i}+\mathbf  h^\circ_{\rm i}
  -\Ma^\circ_{\rm ib} \fractx{\ed}{\ed t}\mathbf y^\circ_{\rm b}-\nu\St^\circ_{\rm ib}\mathbf y^\circ_{\rm ib}.
 \end{align*}
Finally to solve the system above we can discretize the time interval~$[0,+\infty)_D\coloneqq[0k,1k,2k,\dots)$ and use a Crank-Nicolson scheme as follows.
Essentially, we consider the approximations
\begin{align*}
\fractx{\ed}{\ed t}f\rest{\frac{jk+(j-1)k}{2}}&\approx\fractx{f(jk)-f((j-1)k)}{k},\qquad
f(\fractx{jk+(j-1)k}{2})\approx\fractx{f((j-1)k)+(jk)}{2}.
\end{align*}
for a given (differentiable) function~$f$. This will lead us to, by
denoting~$\mathbf y^{j}\coloneqq \mathbf y((j-1)k)$,
 \begin{align}\label{fulldisc}
\left(2\Ma^\circ_{\rm ii} +k\nu\St^{\circ}_{\rm ii}\right)\mathbf y^{j,\circ}_{\rm i}
&=\left(2\Ma^\circ_{\rm ii} -k\nu\St^\circ_{\rm ii}\right)\mathbf y^{j-1,\circ}_{\rm i}
+k\left((\mathbf  h(\mathbf y^{j-1}))_{\rm i}^\circ+(\mathbf  h(\mathbf y^{j}))_{\rm i}^\circ\right)\notag\\
&\quad  -(2\Ma^\circ_{\rm ib} +k\nu\St^\circ_{\rm ib})\mathbf y^{j,\circ}_{\rm b}
+(2\Ma^\circ_{\rm ib}-k\nu\St^\circ_{\rm ib})\mathbf y^{j-1,\circ}_{\rm b}
 \end{align}
We assume that we know~$\mathbf y^{j-1,\circ}_{\rm i}$ and we want to know~$\mathbf y^{j,\circ}_{\rm i}$.
Therefore, all terms on the right hand side are known data, with the exception of~$(\mathbf  h(\mathbf y^{j}))_{\rm i}^\circ$.
So instead of solving~\eqref{fulldisc}, we will solve the similar system where~$(\mathbf  h(\mathbf y^{j}))_{\rm i}^\circ$
is replaced by a suitable approximation~$\mathbf h_{\rm ext}^j$ as follows.
\begin{itemize}
 \item we know $\mathbf y^{1}=\mathbf y_0$, because~$y(0)=y_0$ is given in~\eqref{sys-y-parDir}.
 Then, we can compute $\mathbf  h(\mathbf y^{1})$.
 \item we also set a ``ghost'' point~$\mathbf  h(\mathbf y^0)\coloneqq\mathbf  h(\mathbf y^{1})$.
 \item for~$j\ge 2$, we define~$\mathbf h_{\rm ext}^j$ as the linear extrapolation ~$\mathbf h_{\rm ext}^j
 \coloneqq2\mathbf  h(\mathbf y^{j-1})-\mathbf  h(\mathbf y^{j-2})$.
 \end{itemize}
 That is we solve the system
 \begin{align}\label{fulldisc-ext}
\left(2\Ma^\circ_{\rm ii} +k\nu\St^{\circ}_{\rm ii}\right)\mathbf y^{\circ,j}_{\rm i}
&=\left(2\Ma^\circ_{\rm ii} -k\nu\St^\circ_{\rm ii}\right)\mathbf y^{\circ,j-1}_{\rm i}
+k\left(3(\mathbf  h(\mathbf y^{j-1}))_{\rm i}^\circ-(\mathbf  h(\mathbf y^{j-2}))_{\rm i}^\circ\right)\notag\\
&\quad  -(2\Ma^\circ_{\rm ib} +k\nu\St^\circ_{\rm ib})\mathbf y^{\circ,j}_{\rm b}
+(2\Ma^\circ_{\rm ib}-k\nu\St^\circ_{\rm ib})\mathbf y^{\circ,j-1}_{\rm b}
 \end{align}
Notice that once we have~$\mathbf y^{j-1}$ and~$\mathbf h(\mathbf y^{j-2})$, then we can find~$\mathbf y^{j,\circ}_{\rm i}$ from~\eqref{fulldisc-ext}
because the matrix
~$\left(2\Ma^\circ_{\rm ii} +k\nu\St^{\circ}_{\rm ii}\right)$ is symmetric and positive definite (thus, invertible). Finally, we can construct
~$\mathbf y^{j}=\Pi^\top\begin{bmatrix}\mathbf y^{j,\circ}_{\rm i}&\mathbf y^{j,\circ}_{\rm b}\end{bmatrix}^\top$.

\subsubsection*{The closed-loop system under Neumann boundary conditions}
The dynamics in system~\eqref{sys-y-parNeu} has to be seen in~$L^2((0,T),V')$, that is,
for Neumann boundary conditions, we need to take test functions~$w$
taking values in~$V=H^1(0,\pi)$. That is, now the test vectors do not necessarily satisfy
~$\mathbf w(1)=0=\mathbf w(N)$.

Observe, recalling~\eqref{discLap}, that the space discretization of the closed loop system reads
\begin{equation}\label{discSys.Neu}
 \Ma\fractx{\ed}{\ed t}\mathbf y=-\nu\St\mathbf y +\GGG(t) 
 -\RRR(t)\mathbf y-\Ma[\Cu_M]\PPP_M\left(-\nu\St-\RRR(t)+\lambda\Ma\right)\mathbf y,
\end{equation}
where~$\GGG(t)\in\R^{N\times 1}$ is the vector with all entries equal to zero with the exception the first and last
corresponding to the boundary coordinates
\[
\GGG(t)=\begin{bmatrix}
      g(1)(t)&0&0&\dots&0&0&-g(2)(t)
     \end{bmatrix}^\perp
\]
where $g$ is the vector of boundary data
\[
 g(1)(t)=\fractx{\ed}{\ed x}y(0,t)\quad\mbox{and}\quad g(2)(t)=\fractx{\ed}{\ed x}y(\pi,t).
\]

After the permutation of coordinates~$\Pi$, collecting the boundary coordinates at the end, and proceeding as in the Dirichlet case, we obtain
 \begin{align*}
\left(2\Ma^\circ +k\nu\St^{\circ}\right)\mathbf y^{\circ,j}
&=\left(2\Ma^\circ -k\nu\St^\circ\right)\mathbf y^{\circ,j-1}
+k\left(3(\mathbf  h(\mathbf y^{j-1}))^\circ-(\mathbf  h(\mathbf y^{j-2}))^\circ\right)\notag\\
 &\quad +k(\GGG^{\circ,j}+\GGG^{\circ,j-1})
 \end{align*}
with~$\GGG^{\circ,j}\coloneqq\GGG^{\circ}((j-1)k)$, $j\ge1$.

\subsection{Feedback performance}
Here we present some simulations for both Dirichlet and Neumann boundary conditions. 
We take
\begin{equation}\label{data_perf}
 L=\pi,\qquad \nu=0.1,\qquad\mbox{and}\qquad y_0(x)=0.1x.
\end{equation}
In Figure~\ref{Fig:perf_cr} we see the performance in the simple case of the constant reaction
\begin{equation}\label{data_cr}
a(x,t)=-35\nu. 
\end{equation}

\begin{figure}[ht]
\subfigure
{\epsfig{file=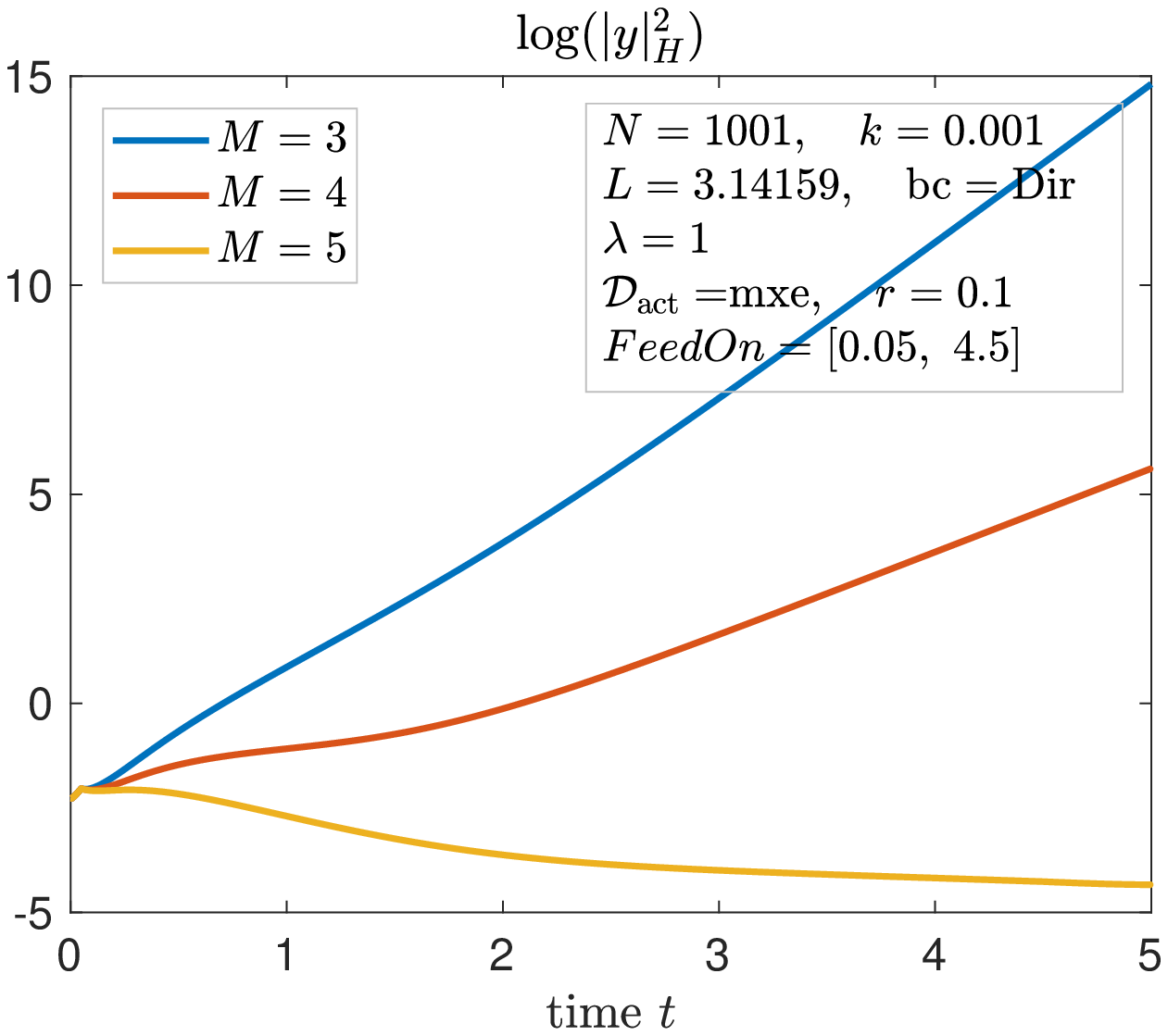,width=.325\linewidth,clip=}}%
\qquad
\subfigure
{\epsfig{file=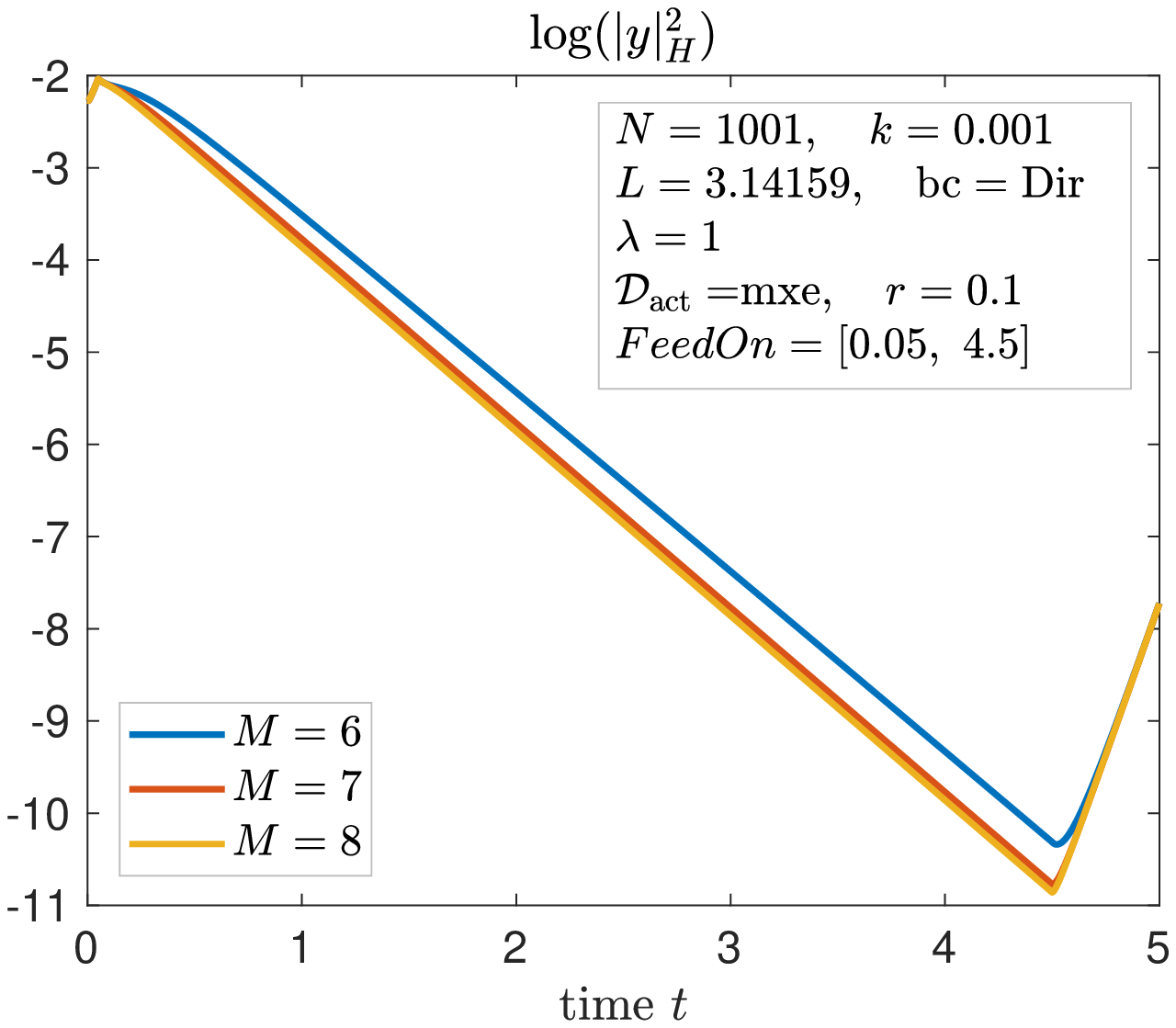,width=.325\linewidth,clip=}}\\%
\subfigure
{\epsfig{file=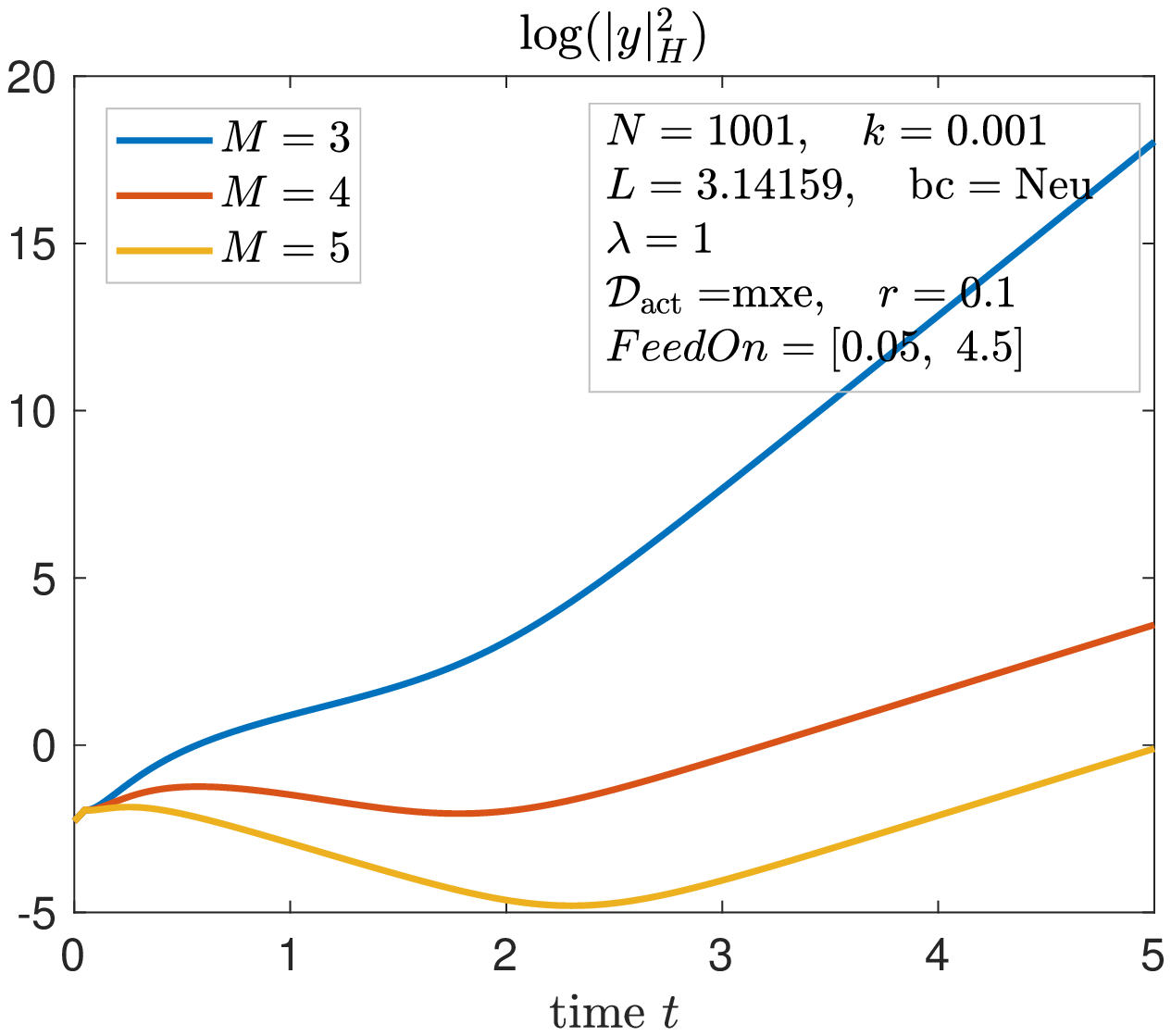,width=.325\linewidth,clip=}}%
\qquad
\subfigure
{\epsfig{file=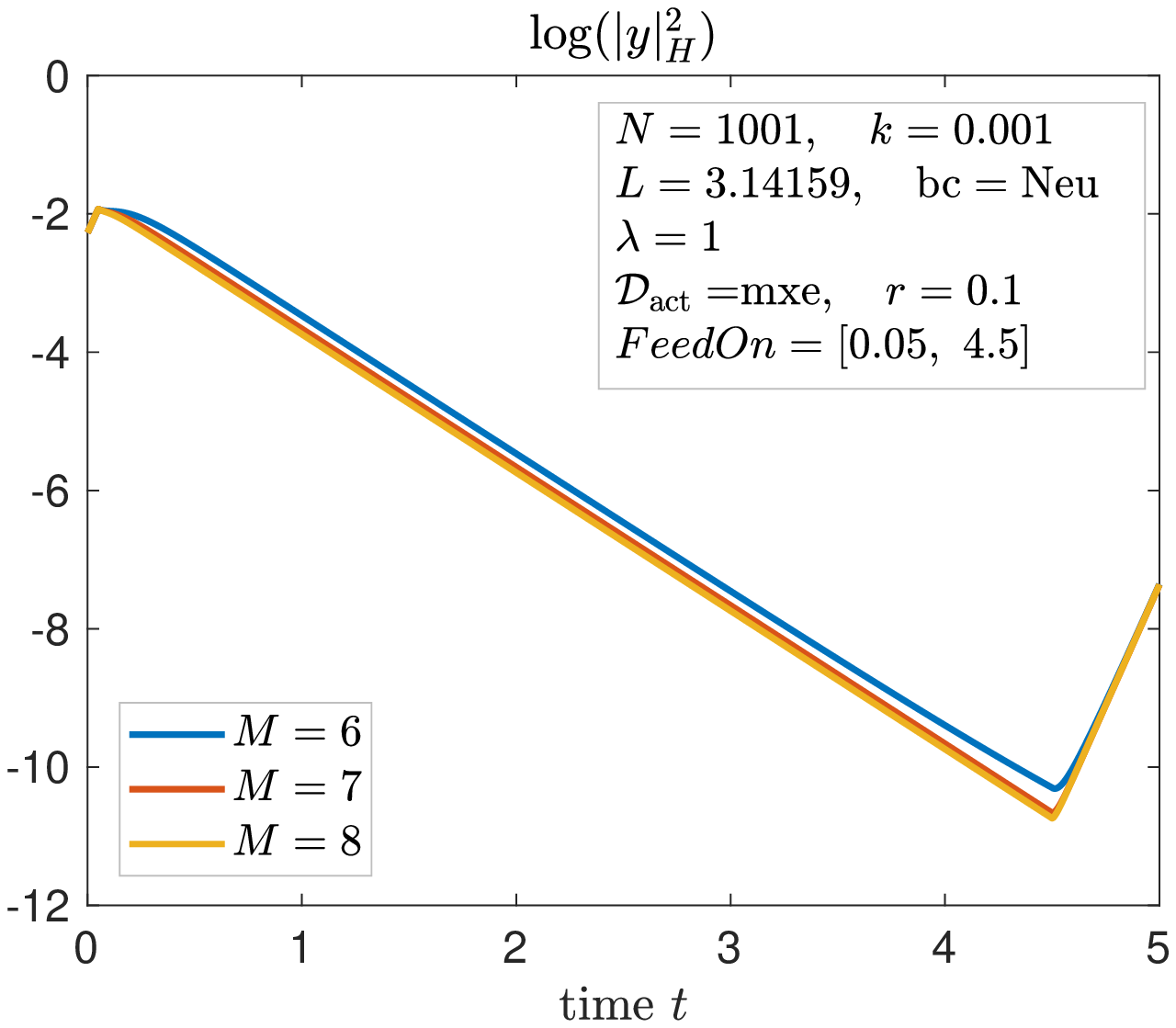,width=.325\linewidth,clip=}}%
\caption{Controlled solution under homogeneous Dirichlet and homogeneous Neumann boundary conditions. Feedback switched on only on the time
interval~$FeedOn$. Case of the constant reaction~\eqref{data_cr}.}
\label{Fig:perf_cr}
\end{figure}

\begin{figure}[ht]
\subfigure
{\epsfig{file=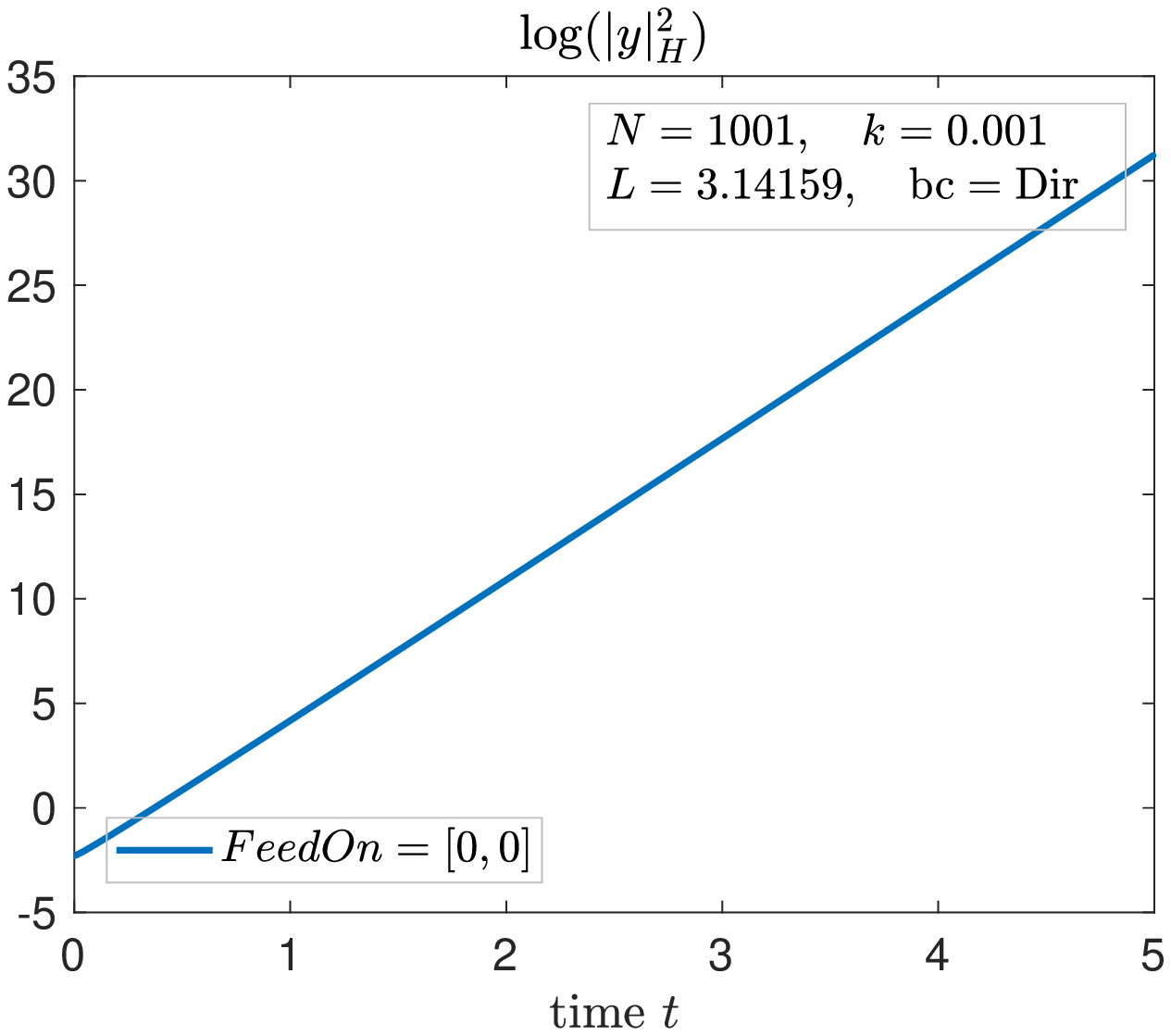,width=.325\linewidth,clip=}}%
\qquad
\subfigure
{\epsfig{file=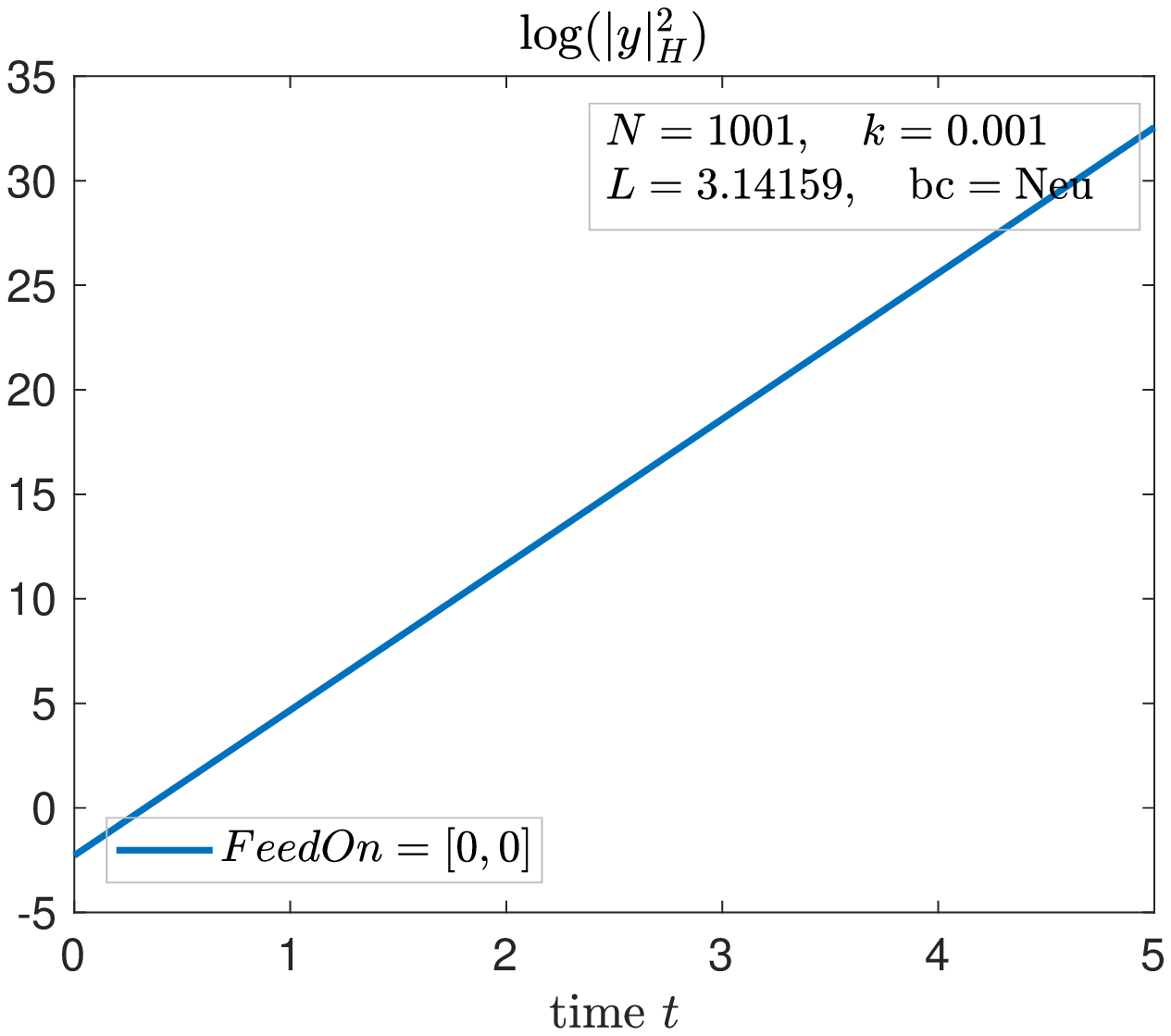,width=.325\linewidth,clip=}}%
\caption{Uncontrolled solution under homogeneous
Dirichlet and homogeneous Neumann boundary conditions. Case of the constant reaction~\eqref{data_cr}.}
\label{Fig:perf_cr_unc}
\end{figure} 
We observe that the oblique projection based feedback is able to stabilise the system for both Dirichlet and Neumann
boundary conditions, when~$M\ge6$.

We also observe that~$5$ actuators are able to stabilise 
the system under Dirichlet boundary conditions, but they are not are able to stabilise 
the system under Neumann boundary conditions.
This can be explained from the fact that the~$6$-th eigenvalue of~$-\nu\Delta$, that is~${}^{\tt d}\!\alpha_{5+1}=36\nu$
in the Dirichlet case, and~${}^{\tt n}\!\alpha_{5+1}=25\nu$ in the Neumann case, satisfy
\[
 {}^{\tt d}\!\alpha_{5+1}>-a>{}^{\tt n}\!\alpha_{5+1}.
\]
See also the discussion on~\cite[section~5.1]{KunRod-pp17}.

To test the performance of the feedback, we switched the control off for time~$t\notin FeedOn$. This means that for~$t\notin FeedOn$,
the free dynamics is followed. In Figure~\ref{Fig:perf_cr} we see that when we switch the control off (after time t=4.5) the
norm of the solution starts to increase, which shows/suggests the instability of the free dynamics. 
In Figure~\ref{Fig:perf_cr_unc}, the control is switched off in the entire time interval; we confirm the instability of the free dynamics.

Now we take the data as in~\eqref{data_perf}, and consider the case of a reaction depending both in space and time 
\begin{equation}\label{data_genr}
 a=a(t,x)=-35\nu(\fractx{\pi}{L})^2-2\norm{\cos(4t)\cos(xt)x}{\bbR}.
\end{equation}
We observe, in Figure~\ref{Fig:perf_genr}, that the feedback is able to stabilise the system for both Dirichlet and Neumann
boundary conditions, when~$M\ge8$.

We also observe that~$7$ actuators are likely able to stabilise 
the system under Dirichlet boundary conditions, but they are likely not  able to stabilise 
the system under Neumann boundary conditions.

 In Figure~\ref{Fig:perf_genr_unc} we observe that the free dynamics in unstable.

\begin{figure}[ht]
\subfigure
{\epsfig{file=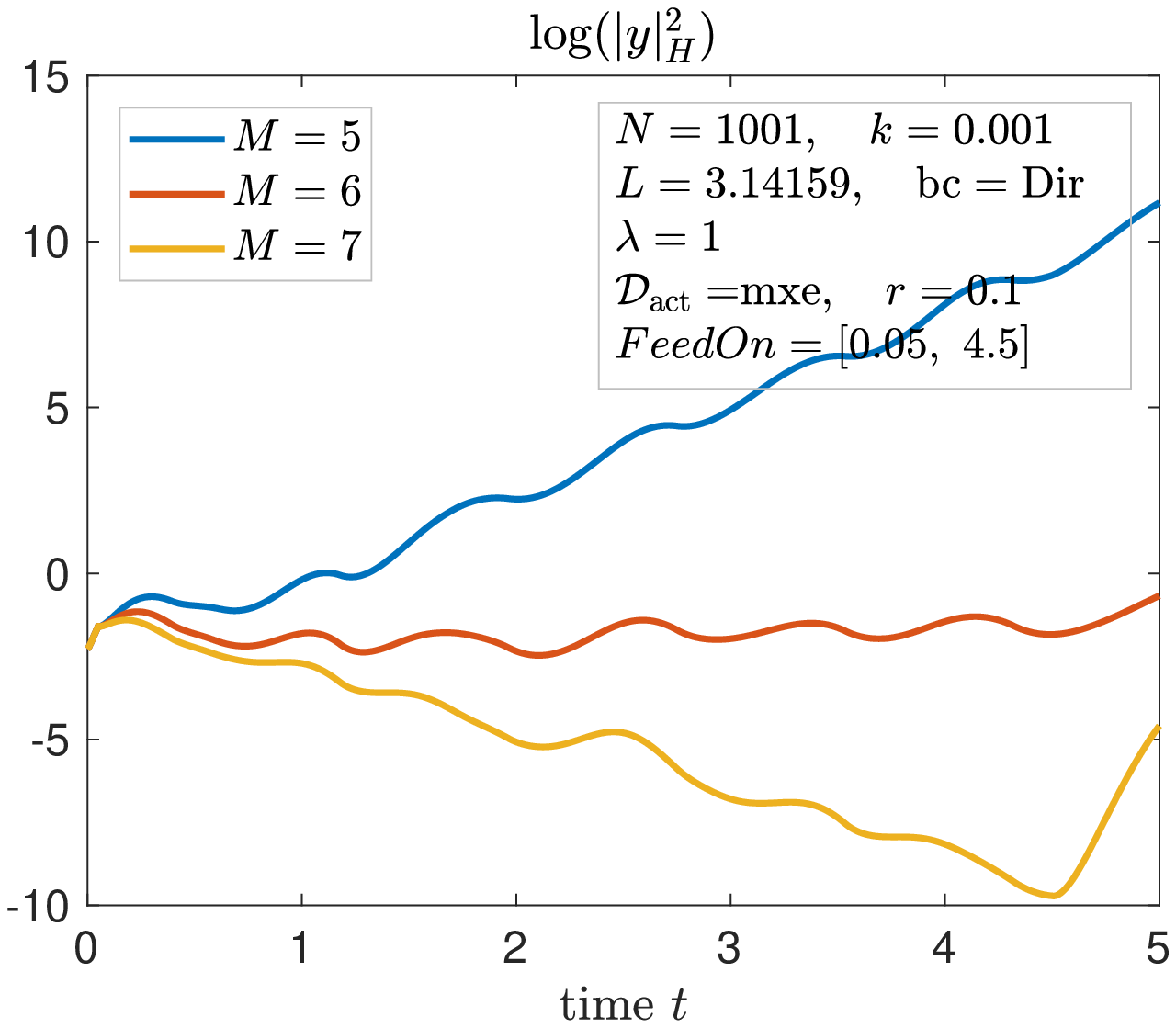,width=.325\linewidth,clip=}}%
\qquad
\subfigure
{\epsfig{file=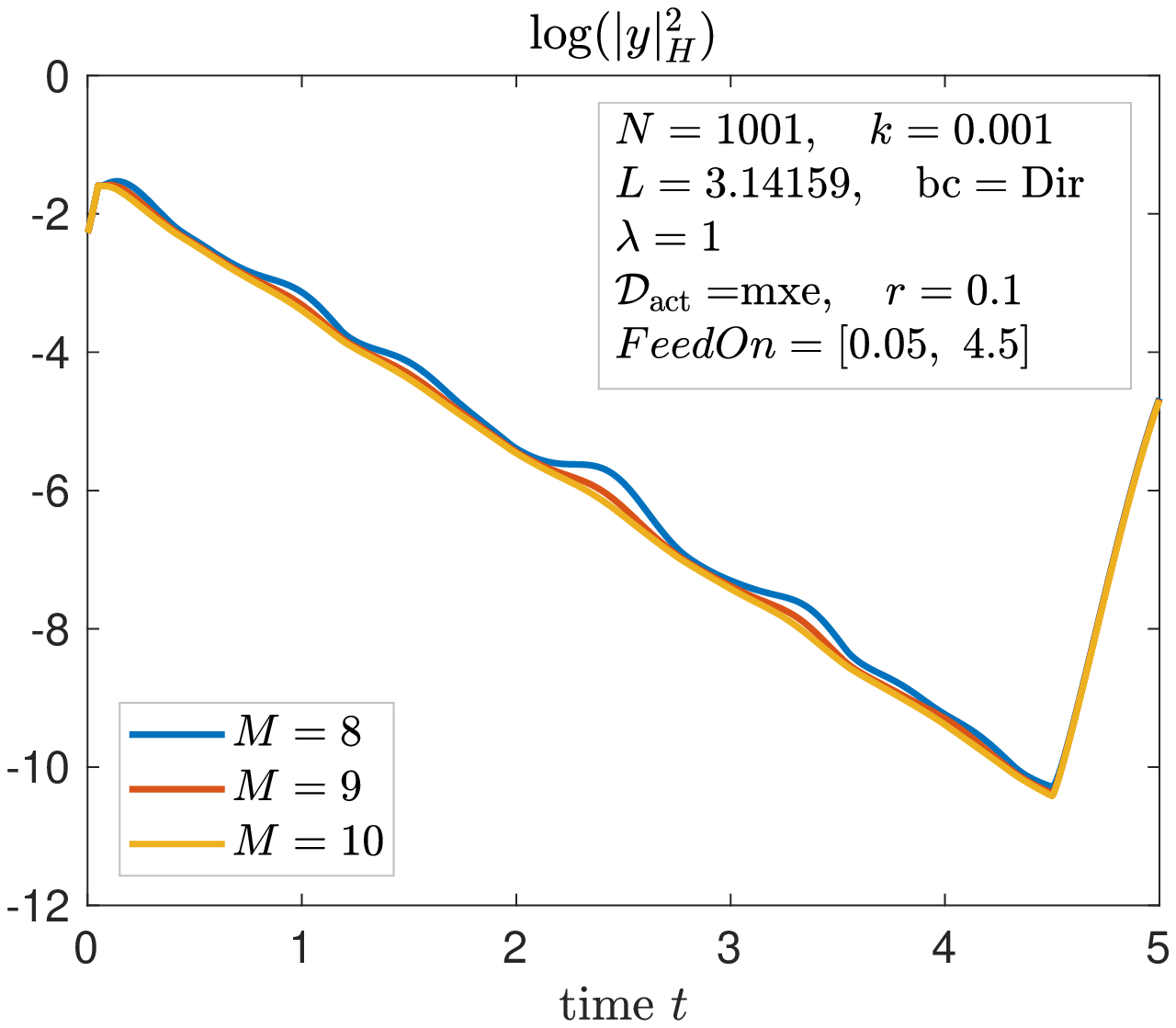,width=.325\linewidth,clip=}}\\%
\subfigure
{\epsfig{file=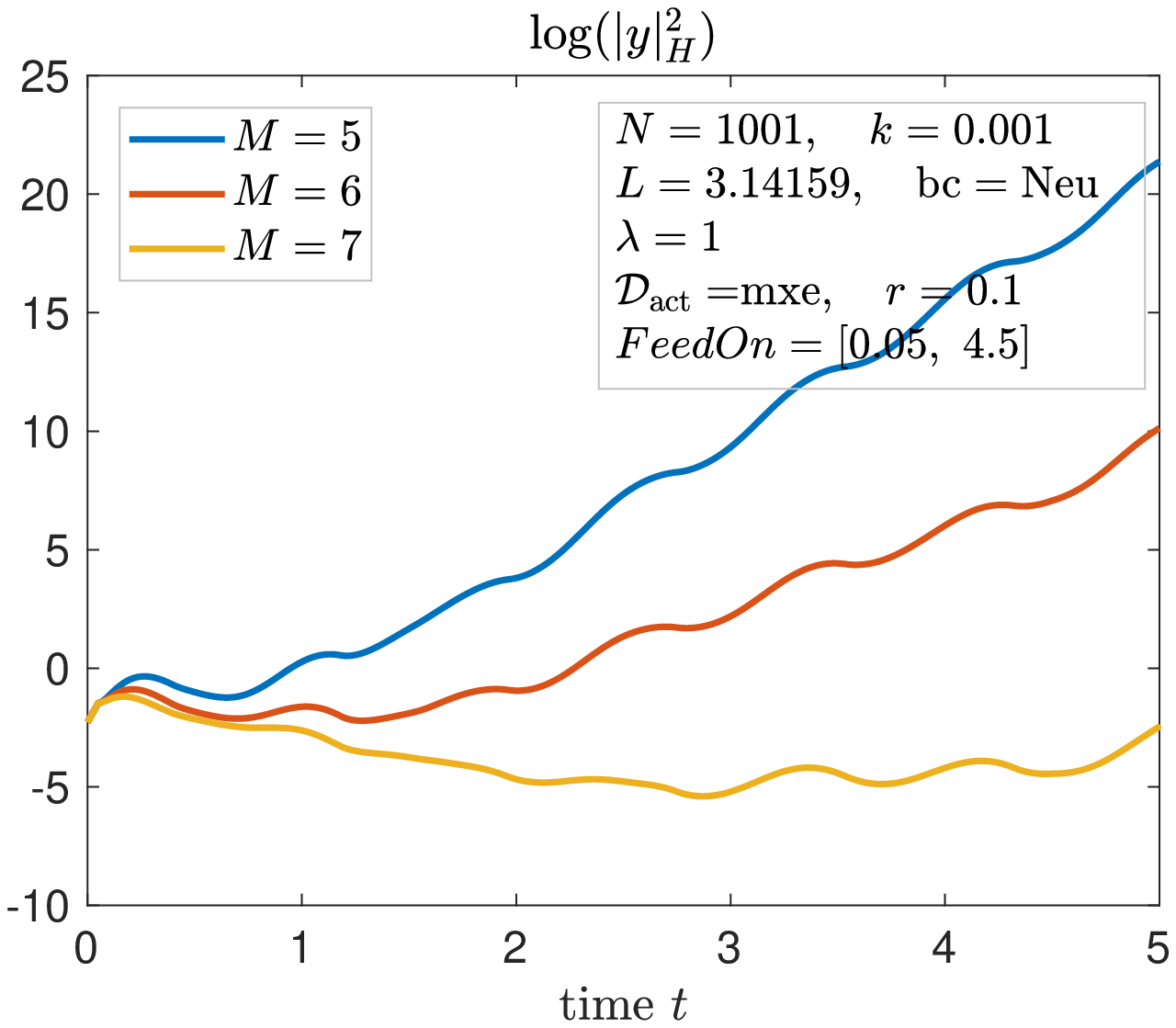,width=.325\linewidth,clip=}}%
\qquad
\subfigure
{\epsfig{file=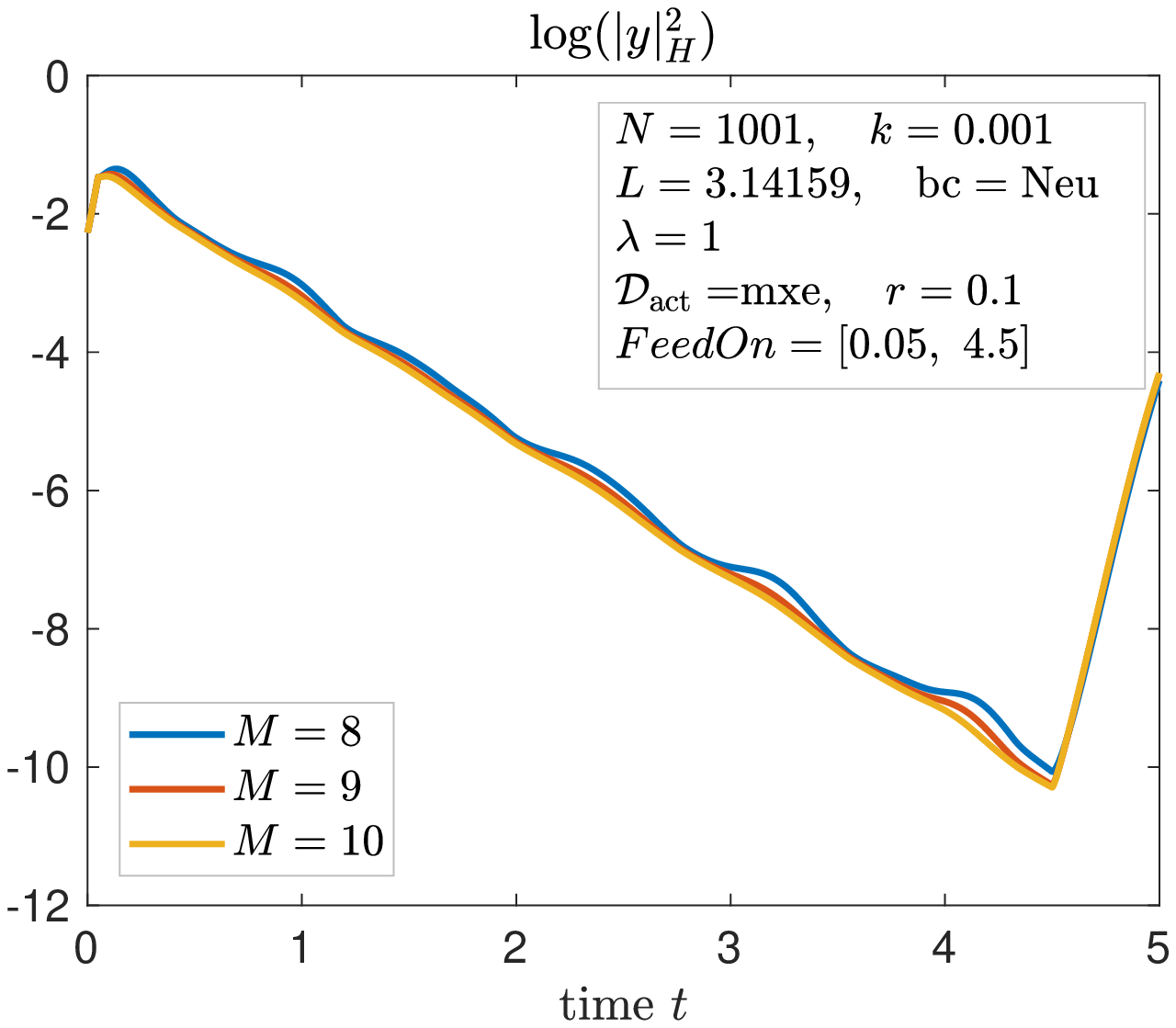,width=.325\linewidth,clip=}}%
\caption{Controlled solution under homogeneous Dirichlet and homogeneous Neumann boundary conditions. Feedback switched on only on the time
interval~$FeedOn$. Case of the reaction~\eqref{data_genr}.}
\label{Fig:perf_genr}
\end{figure}

\begin{figure}[ht]
\subfigure
{\epsfig{file=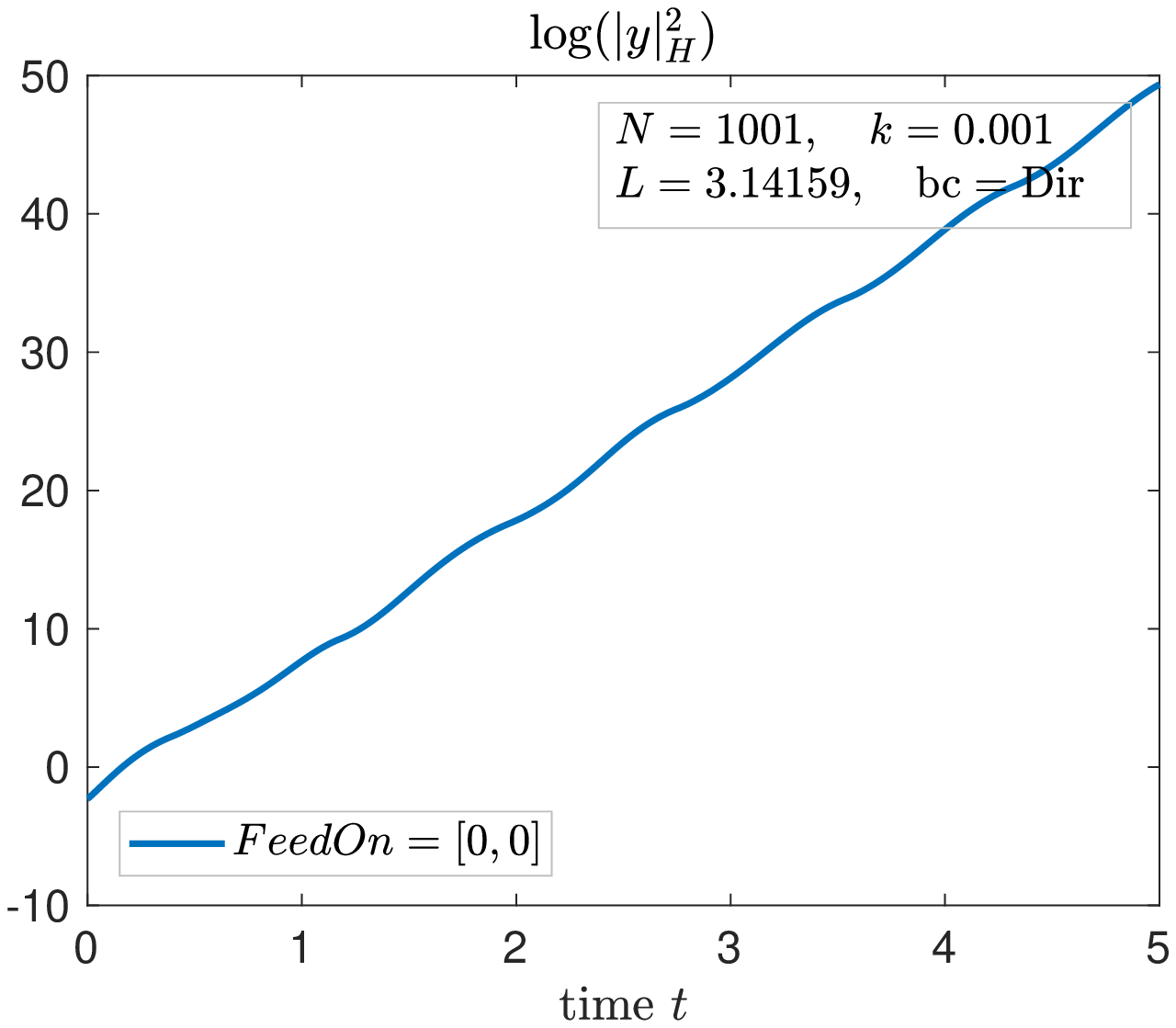,width=.325\linewidth,clip=}}%
\qquad
\subfigure
{\epsfig{file=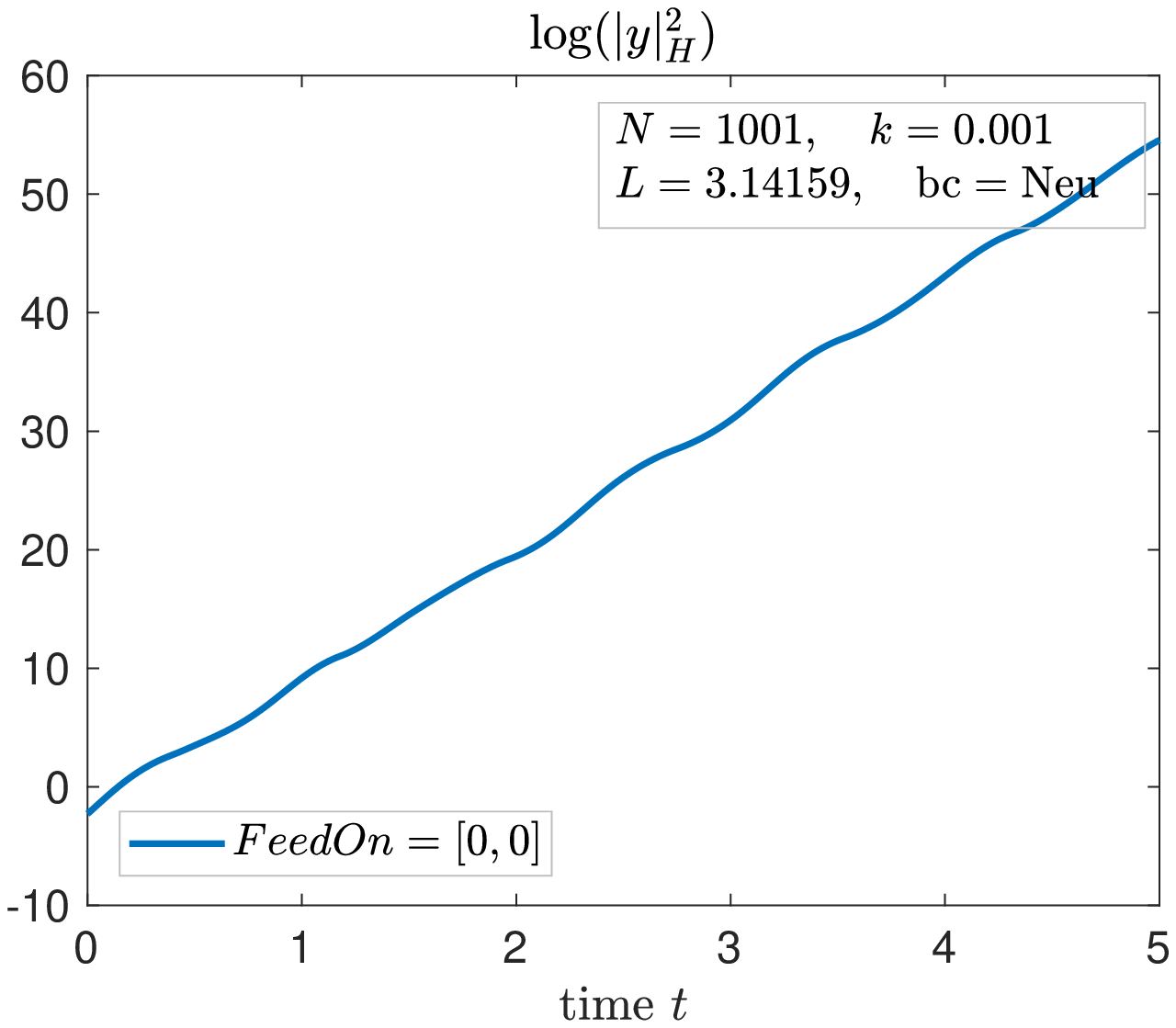,width=.325\linewidth,clip=}}%
\caption{Uncontrolled solution under homogeneous Dirichlet and
homogeneous Neumann boundary conditions. Case of the reaction~\eqref{data_genr}.}
\label{Fig:perf_genr_unc}
\end{figure} 

For further simulations and discussions (under Dirichlet boundary conditions)
we refer to~\cite{KunRod-pp17}.

\bigskip
\appendix
\addcontentsline{toc}{section}{Appendix}
\begin{center}
{\sc --- Appendix ---}
\end{center}
\setcounter{section}{1}
\setcounter{theorem}{0} \setcounter{equation}{0}
\numberwithin{equation}{section}
\nopagebreak

\subsection{Proof that it suffices to consider $L=\pi$}\label{sS:ProofP:L=pi}

\begin{proposition}\label{P:L=pi}
 Let $L>0$ and let us denote the eigenfunctions of the Dirichlet
 Laplacian in the interval~$(0,L)$ by~${}^{\tt d}\!e_i^{[L]}(x)$, $x\in(0,L)$. Then we have that
 ${}^{\tt d}\!e_i^{[\pi]}(w)={}^{\tt d}\!e_i^{[L]}(\frac{L}{\pi}w)$, $w\in(0,\pi)$. For given actuators~$1_{\omega_j}^{[L]}(x)$ we set
 $1_{\omega_j}^{[\pi]}(w)\coloneqq 1_{\omega_j}^{[L]}(\frac{L}{\pi}w)$. Then we have that
 $
 \norm{P_{U_M[L]}^{{}^{\tt d}\!E_M^\perp[L]}}{\LL(L^2(0,L))}^2
=\norm{P_{U_M[\pi]}^{{}^{\tt d}\!E_M^\perp[\pi]}}{\LL(L^2(0,\pi))}^2. 
 $ The analogous result also follows in the case of the Neumann
 Laplacian.
\end{proposition}
Let us consider the rescaling bijective linear
mapping~$R:x\mapsto y=\frac{L}{\pi}x$ mapping~$(0,\pi)$ onto~$(0,L)$. Then the mapping
$S=S_{[\pi: L]}\in\Cl(L^2((0,\pi)),L^2((0,L)))$, defined by~$S:f(x)\mapsto Sf(y)\coloneqq f(R^{-1}y)$ satisfies
\[
 \norm{S}{\Cl(L^2((0,\pi)),L^2((0,L)))}=(\fractx{L}{\pi})^\frac{1}{2},
\]
because~$\int_0^{L}|Sf(y)|^2\ed y=\int_0^{\pi}|f(x)|^2\frac{L}{\pi}\ed x=\frac{L}{\pi}\int_0^{\pi}|f(x)|^2\ed x$.

Further observe that~$S_{[L:\pi]}=S_{[\pi:L]}^{-1}$, and recall that
${}^{\tt d}\!e_i^{[\pi]}(w)={}^{\tt d}\!e_i^{[L]}(Rw)$, and
 $1_{\omega_j}^{[\pi]}(w)\coloneqq 1_{\omega_j}^{[L]}(\frac{L}{\pi}w)$, $w\in(0,\pi)$. 

 We may write
 \[
P_{U_M[L]}^{{}^{\tt d}\!E_M^\perp[L]} = S_{[L:\pi]}\circ P_{U_M[\pi]}^{{}^{\tt d}\!E_M^\perp[\pi]}\circ S_{[\pi:L]}\qquad\mbox{and}\qquad
S_{[\pi:L]}\circ P_{U_M[L]}^{{}^{\tt d}\!E_M^\perp[L]}\circ S_{[L:\pi]} =  P_{U_M[\pi]}^{{}^{\tt d}\!E_M^\perp[\pi]},
\]
 from which we can conclude that~$\norm{P_{U_M[L]}^{{}^{\tt d}\!E_M^\perp[L]}}{\LL(L^2(0,L))}^2
=\norm{P_{U_M(c)[\pi]}^{{}^{\tt d}\!E_M^\perp[\pi]}}{\LL(L^2((0,\pi)))}^2$.

Following the same argument, for the case of Neumann eigenfunctions of the Laplacian, we will arrive to the
identity~$\norm{P_{U_M[L]}^{{}^{\tt n}\!E_M^\perp[L]}}{\LL(L^2(0,L))}^2
=\norm{P_{U_M(c)[\pi]}^{{}^{\tt n}\!E_M^\perp[\pi]}}{\LL(L^2(0,\pi))}^2$.

Therefore, the norm of the oblique projection does not depend on the length~$L$ of the interval.
See also~\cite[Section~4.6]{KunRod-pp17}.
 The proof of Proposition~\ref{P:L=pi} is finished.
\qed

\bibliography{refsActPlac}
\bibliographystyle{plainurl}

\end{document}

%% file: slidesymbols.tex
\providecommand{\argmin}{\operatorname*{argmin}}  
\providecommand{\Id}{\Op{Id}}                     

\providecommand{\Ce}{{\cal E}}
\providecommand{\Cf}{{\cal F}}
\providecommand{\Cg}{{\cal G}}

\providecommand{\Ck}{{\cal K}}
\providecommand{\Cl}{{\cal L}}

\providecommand{\Cp}{{\cal P}}

\providecommand{\Cu}{{\cal U}}

\providecommand{\bbN}{\mathbb{N}}

\providecommand{\bbR}{\mathbb{R}}

\newcommand*{\Op}[1]{\mathsf{#1}} 



%% file: figoblique1.tex
%
\begin{picture}(11,11)(-5,-5)
\definecolor{lightgray}{rgb}{.8,.8,.8}%
\definecolor{darkgray}{rgb}{.3,.3,.3}%
\linethickness{.5pt}%
{\color{darkgray}\put(0,0){\circle{6}}}%
\linethickness{1pt}%
\moveto(0,-5)
\lineto(0,5)
\strokepath
\moveto(-5,0)
\lineto(5,0)
\strokepath
\moveto(-5,-2)
\lineto(5,2)
{\color{blue}
\strokepath
}
%
%
\moveto(-3,0)
\lineto(-3,-1.2)
{\color{red}
\strokepath
}
\put(-3,-1.2){\makebox(0,0){\red\scriptsize $\bullet$}}
\put(-2.6,0.4){\makebox(0,0){\red $x$}}
\put(-3,-1.75){\makebox(0,0){\red $x_F$}}
\moveto(1.5,2.5881)
\lineto(1.5,.6)
{\color{green}
\strokepath
}
\put(1.5,.6){\makebox(0,0){\green\scriptsize $\bullet$}}
\put(1.8,2.8){\makebox(0,0){\green $y$}}
\put(2.1,.4){\makebox(0,0){\green $y_F$}}
\moveto(0.7765,-2.8978)
\lineto(0.7765,.3106)
{\color{orange}
\strokepath
}
\put(0.7765,.3106){\makebox(0,0){\orange\scriptsize $\bullet$}}
\put(0.7765,-3.4){\makebox(0,0){\orange $z$}}
\put(0.7,.75){\makebox(0,0){\orange $z_F$}}

\put(0,5.5){\makebox(0,0){$G^\perp$}}
\put(5.5,0){\makebox(0,0){$G$}}
\put(5.5,2){\makebox(0,0){\blue $F$}}

\end{picture}

%% file: figoblique2.tex
%
\begin{picture}(11,11)(-5,-5)
\definecolor{lightgray}{rgb}{.8,.8,.8}%
\definecolor{darkgray}{rgb}{.3,.3,.3}%
\linethickness{1pt}%
\moveto(0,-5)
\lineto(0,5)
\strokepath
\moveto(-5,0)
\lineto(5,0)
\strokepath
\moveto(-5,-2)
\lineto(5,2)
{\color{blue}
\strokepath
}

\put(4,5.5){\makebox(0,0){\cyan $(P_{F}^{G^\perp})^{-1}(\{f\})$}}
\put(4.5,1.2){\makebox(0,0){$f$}}
\put(4,1.6){\makebox(0,0){\cyan\scriptsize $\bullet$}}
\moveto(4,-5)
\lineto(4,5)
{\color{cyan}
\strokepath
}

\moveto(-5,.5)
\lineto(5,4.5)
{\color{magenta}
\strokepath
}

\put(-4,2.5){\makebox(0,0){\color{magenta} $(P_{G^\perp}^{F})^{-1}(\{g\})$}}
\put(-.5,1.75){\makebox(0,0){$g$}}
\put(0,2.5){\makebox(0,0){\magenta\scriptsize $\bullet$}}

\put(0,5.5){\makebox(0,0){$G^\perp$}}
\put(5.5,0){\makebox(0,0){$G$}}
\put(5.5,2){\makebox(0,0){\blue $F$}}

\end{picture}